\documentclass[oupthm,crop,info]{CUP-JNL-CJM}
\usepackage{subfiles}

\usepackage{amsmath,amssymb,amsfonts}
\usepackage[final]{microtype}
\usepackage{xcolor}
\definecolor{links}{rgb}{0.65, 0.00, 0.36}
\usepackage[style=alphabetic,natbib=true,maxalphanames=7,maxbibnames=99]{biblatex}
\usepackage[capitalize,noabbrev]{cleveref}
\usepackage{datetime}
\usepackage[inline]{enumitem}
\usepackage{environ}
\usepackage{etoolbox}
\usepackage{graphicx}
\usepackage{manfnt}
\usepackage{mathpartir}
\usepackage{mathrsfs}
\usepackage{mathtools}
\usepackage{multicol,multirow}
\usepackage{pict2e}
\usepackage{rotating}
\usepackage{stmaryrd}
\usepackage{tikz-cd}
\usepackage[safe]{tipa}
\usepackage{todonotes}
\usepackage{xparse}
\usepackage{xspace}
\usepackage{appendix}

\usepackage{tikz}
\usetikzlibrary{arrows}

\DeclarePairedDelimiter\parens\lparen\rparen
\DeclarePairedDelimiter\floors\lfloor\rfloor
\DeclarePairedDelimiter\ceils\lceil\rceil

\DeclarePairedDelimiter\verts\lvert\rvert

\DeclarePairedDelimiter\angles\langle\rangle

\DeclarePairedDelimiter\braces\lbrace\rbrace
\DeclarePairedDelimiterX\set[2]\lbrace\rbrace{#1 \mathrel{\delimsize\vert} #2}

\makeatletter
\newcommand{\DeclareAbbrevation}[2]{\newcommand{#1}{\@ifnextchar{.}{#2}{#2.\@\xspace}}}
\makeatother

\DeclareAbbrevation{\ie}{i.e}
\DeclareAbbrevation{\eg}{e.g}
\DeclareAbbrevation{\cf}{cf}
\DeclareAbbrevation{\etc}{etc}
\DeclareAbbrevation{\resp}{resp}
\DeclareAbbrevation{\ibid}{ibid}
\DeclareAbbrevation{\ca}{ca}

\DeclareCiteCommand{\thmcite}
  {\usebibmacro{prenote}}
  {\usebibmacro{citeindex}%
   \usebibmacro{cite}}
  {\multicitedelim}
  {\usebibmacro{postnote}}

\theoremstyle{oupdefinition}
\newtheorem{definition}{Definition}[section]

\theoremstyle{oupplain}
\newtheorem{lemma}[definition]{Lemma}
\newtheorem{theorem}[definition]{Theorem}
\newtheorem{proposition}[definition]{Proposition}
\newtheorem{corollary}[definition]{Corollary}

\newtheorem{manualtheoreminner}{Theorem}
\newenvironment{manualtheorem}[1]{%
  \renewcommand\themanualtheoreminner{#1}%
  \manualtheoreminner
}{\endmanualtheoreminner}

\theoremstyle{oupremark}
\newtheorem{remark}[definition]{Remark}
\newtheorem{example}[definition]{Example}
\newtheorem{assumption}[definition]{Assumption}
\newtheorem{notation}[definition]{Notation}
\newtheorem{terminology}[definition]{Terminology}

\theoremstyle{oupproof}
\newtheorem{proof}{Proof}

\newlist{conditions}{enumerate}{1}
\setlist[conditions]{label={(\Alph*)},ref={\Alph*}}
\crefname{conditionsi}{condition}{conditions}

\numberwithin{equation}{section}

\mathchardef\mh="2D

\newcommand{\define}[4]{\expandafter#1\csname#3#4\endcsname{#2{#4}}}

\forcsvlist{\define{\newcommand}{\mathbb}{B}}%
  {A,B,C,D,E,F,G,H,I,J,K,L,M,N,O,P,Q,R,S,T,U,V,W,X,Y,Z}

\forcsvlist{\define{\newcommand}{\mathbf}{Cat}}%
{A,B,C,D,E,F,G,H,I,J,K,L,M,N,O,P,Q,R,S,T,U,V,W,X,Y,Z}

\forcsvlist{\define{\newcommand}{\mathcal}{C}}%
  {A,B,C,D,E,F,G,H,I,J,K,L,M,N,O,P,Q,R,S,T,U,V,W,X,Y,Z}

\forcsvlist{\define{\newcommand}{\mathfrak}{F}}%
  {A,B,C,D,E,F,G,H,I,J,K,L,M,N,O,P,Q,R,S,T,U,V,W,X,Y,Z,a,b,c,d,e,f,g,h,i,j,k,l,m,n,o,p,q,r,s,t,u,v,w,x,y,z}

\newcommand{\GD}{\ensuremath{\Delta}}

\newcommand{\GO}{\ensuremath{\Omega}}

\newcommand{\GS}{\ensuremath{\Sigma}}

\newcommand{\Ga}{\ensuremath{\alpha}}

\newcommand{\Gf}{\ensuremath{\varphi}}

\newcommand{\Gd}{\ensuremath{\delta}}
\newcommand{\Ge}{\ensuremath{\varepsilon}}
\newcommand{\Gh}{\ensuremath{\eta}}
\newcommand{\Gi}{\ensuremath{\iota}}
\newcommand{\Gk}{\ensuremath{\kappa}}

\newcommand{\Gp}{\ensuremath{\pi}}
\newcommand{\Gps}{\ensuremath{\psi}}

\newcommand{\Gs}{\ensuremath{\sigma}}

\newcommand{\Gth}{\ensuremath{\theta}}
\newcommand{\Gw}{\ensuremath{\omega}}

\NewDocumentCommand\YoSymScaled{m}{
  \raisebox{#1em * \real{-.02}}{%
    \includegraphics[quiet=true,draft=false,height=#1em * \real{.67},keepaspectratio]{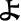}%
    \hspace{0.1em}%
  }
}

\NewDocumentCommand{\yosym}{}{
  \mathord{%
    \mathchoice{\YoSymScaled{1}}{\YoSymScaled{1}}{\YoSymScaled{\defaultscriptratio}}{\YoSymScaled{\defaultscriptscriptratio}}
  }%
}

\newcommand{\defeq}{\coloneqq}

\newcommand{\co}{\colon}

\DeclareMathOperator{\Ob}{\mathrm{Ob}}
\newcommand{\Hom}[3]{{#1}(#2,#3)}
\NewDocumentCommand{\id}{o}{\mathrm{id}\IfValueT{#1}{_{#1}}}
\NewDocumentCommand{\Id}{o}{\mathrm{Id}\IfValueT{#1}{_{#1}}}

\newcommand{\op}[1]{#1^{\mathrm{op}}}

\newcommand{\pair}[2]{\langle #1, #2 \rangle}
\newcommand{\projl}{\Gp_0}
\newcommand{\projr}{\Gp_1}

\newcommand{\copair}[2]{[ #1, #2 ]}
\newcommand{\inl}{\Gi_0}
\newcommand{\inr}{\Gi_1}

\newcommand{\Func}[2]{[#1,#2]}
\newcommand{\FPP}[2]{\Func{#1}{#2}_{\mathrm{fp}}}

\newcommand{\leib}[1]{\mathbin{\widehat{#1}}}
\newcommand{\lbtimes}{\leib{\times}}
\newcommand{\lbto}{\leib{\to}}
\newcommand{\lbotimes}{\leib{\otimes}}
\newcommand{\lboslash}{\leib{\oslash}}
\newcommand{\app}{\mathbin{@}}

\NewDocumentCommand{\Set}{o}{\mathbf{Set}\IfValueT{#1}{_{#1}}}
\newcommand{\FinSet}{\Set[\lfin]}

\newcommand{\Pos}{\mathbf{Pos}}
\newcommand{\Cat}{\mathbf{Cat}}

\newcommand{\catInit}{\mathbf{0}}
\newcommand{\catTerm}{\mathbf{1}}

\newcommand{\lfin}{\mathrm{fin}}
\newcommand{\linh}{\mathrm{inh}}
\newcommand{\ldec}{\mathrm{dec}}

\newcommand{\lcart}{\mathrm{cart}}
\newcommand{\CartArr}[1]{{#1}^\to_\lcart}

\newcommand{\Aut}[2][]{\mathrm{Aut}_{#1}(#2)}

\newcommand{\Slice}[2]{{#1}/{#2}}
\newcommand{\Coslice}[2]{{#2}/{#1}}

\newcommand{\autquo}[2]{{#1}/{#2}}

\newcommand{\Sieve}[2]{\mathrm{Sv}_{#1}(#2)}
\newcommand{\PSieve}[2]{\mathrm{PrSv}_{#1}(#2)}
\newcommand{\principal}[1]{\langle #1 \rangle}

\DeclareMathOperator*{\colim}{colim}
\DeclareMathOperator*{\im}{Im}

\NewDocumentCommand{\wcolim}{m g g}{\IfValueTF{#2}{{#2} \mathbin{\oast_{#1}} {#3}}{{\oast_{#1}}}}
\NewDocumentCommand{\lbwcolim}{m g g}{\IfValueTF{#2}{{#2} \mathbin{\leib{\oast}_{#1}} {#3}}{{\oast_{#1}}}}
\newcommand{\tens}{\mathbin{\ast}}

\DeclareMathOperator{\coend}{\int}

\NewDocumentCommand{\mono}{s}{\IfBooleanTF{#1}{\leftarrowtail}{\rightarrowtail}}
\NewDocumentCommand{\epi}{s}{\IfBooleanTF{#1}{\twoheadleftarrow}{\twoheadrightarrow}}

\newcommand{\core}[1]{\mathrm{Core}({#1})}

\makeatletter
\newcommand{\adjunction}[4]{%
  #1\colon #2%
  \mathrel{\vcenter{%
    \offinterlineskip\m@th
    \ialign{%
      \hfil$##$\hfil\cr
      \longrightarrow\cr
      \noalign{\kern-.3ex}
      \smallbot\cr
      \noalign{\kern.1ex}
      \longleftarrow\cr
    }%
  }}%
  #3 \noloc #4%
}
\newcommand{\noloc}{%
  \nobreak
  \mspace{6mu plus 1mu}
  {:}
  \nonscript\mkern-\thinmuskip
  \mathpunct{}
  \mspace{2mu}
}
\newcommand{\smallbot}{%
  \begingroup\setlength\unitlength{.25em}%
  \begin{picture}(1,1)
  \roundcap
  \polyline(0,0)(1,0)
  \polyline(0.5,0)(0.5,1)
  \end{picture}%
  \endgroup
}
\makeatother

\tikzset{
  epi/.style = {commutative diagrams/two heads},
  mono/.style = {commutative diagrams/tail},
  cof/.style = {{Triangle[reversed,scale=0.9]}->},
  fib/.style = {-{Triangle[open,scale=0.9]}},
  weq/.style = {"\sim"'{sloped,font=\tiny,#1}},
  tcof/.style = {cof,weq={#1}},
  tcof/.default = {above},
  tfib/.style = {fib,weq={#1}},
  tfib/.default = {above},
  lowering/.style = {"-"'{sloped,font=\tiny,#1}},
  raising/.style = {"+"'{,sloped,font=\tiny,#1}},
}

\tikzcdset{
  arrow style=tikz,
  every label/.append style = {font = \small}
}

\NewDocumentCommand{\pushout}{D<>{0} O{6ex}}{%
  \arrow[phantom,start anchor=center,to path={-- ++({#1+135}:#2) \tikztonodes}, "\rotatebox{#1}{$\ulcorner$}"]
}
\NewDocumentCommand{\pullback}{D<>{0} O{6ex}}{%
  \arrow[phantom,start anchor=center,to path={-- ++({#1-45}:#2) \tikztonodes}, "\rotatebox{#1}{$\lrcorner$}"]
}

\NewDocumentCommand{\PSh}{o m}{\mathrm{PSh}\IfValueT{#1}{_{#1}}(#2)}
\NewDocumentCommand{\yo}{g}{\yosym\IfValueT{#1}{#1}}

\newcommand{\subst}[1]{#1^*}
\newcommand{\lan}[1]{{#1}_!}
\newcommand{\ran}[1]{{#1}_*}

\newcommand{\nerve}[1]{N_{#1}}

\DeclareMathOperator{\catEl}{\mathbf{el}}

\newcommand{\Simp}{\GD}
\newcommand{\simp}[1]{\GD^{#1}} 
\newcommand{\simpbd}[1]{\partial \Delta^{#1}} 
\newcommand{\simphorn}[2]{\Lambda^{#1}_{#2}} 

\newcommand{\Cube}[1]{\square_{#1}}
\newcommand{\ival}{\BI}

\newcommand{\Aff}{\Cube{\mathrm{aff}}}
\newcommand{\Cart}{\Cube{{\times}}}
\newcommand{\Or}{\Cube{\lor}}
\newcommand{\Ded}{\Cube{\land\!\lor}}

\newcommand{\SLat}{\mathbf{SLat}}
\newcommand{\BSLat}{\mathbf{01}\SLat}
\newcommand{\LBSLat}{\mathbf{0}\SLat}
\newcommand{\SLatF}{\SLat_\lfin}
\newcommand{\SLatFNE}{\SLat^\linh_\lfin}

\newcommand{\insimp}{\blacktriangle}

\newcommand{\incube}{\mathchoice
  {\scalebox{0.80}{$\blacksquare$}}
  {\scalebox{0.80}{$\blacksquare$}}
  {\scalebox{0.50}{$\blacksquare$}}
  {\scalebox{0.25}{$\blacksquare$}}
}

\newcommand{\Simpaug}{\Simp_{\mathrm{a}}}
\newcommand{\IdCompOraug}{\overline{\square}_{\lor{\mathrm{a}}}}
\newcommand{\insimpaug}{\insimp_{\mathrm{a}}}

\newcommand{\Triang}{\mathrm{T}}
\newcommand{\triang}{{\boxslash}}

\newcommand{\lcof}{\mathrm{cof}}
\newcommand{\lfib}{\mathrm{fib}}

\NewDocumentCommand{\arrowofstyle}{m m m}{
  \NewDocumentCommand#1{o}{
    \mathrel{\hspace{-0.63ex}\tikz[baseline=-0.7ex, shorten <=3pt, shorten >=2pt]
      \draw[#2,line width=0.08ex](0,0) -- (1.6em,0)
      \IfValueT{##1}{node [midway,scale=0.75,yshift=0.5ex,xshift=#3] {$##1$}};
    \hspace{-0.41ex}}
  }
}

\arrowofstyle{\ordinary}{->}{-0.2ex}
\arrowofstyle{\cof}{cof}{0.0ex}
\arrowofstyle{\fib}{fib}{-0.3ex}
\newcommand{\weq}{\ordinary[\sim]}
\newcommand{\tcof}{\cof[\sim]}
\newcommand{\tfib}{\fib[\sim]}

\newcommand{\candweq}[2]{\CW(#1,#2)}

\NewDocumentCommand{\Cyl}{o g}{\ival \otimes\IfValueT{#1}{_{#1}} \IfValueTF{#2}{#2}{(-)}}
\NewDocumentCommand{\Cocyl}{o g}{\ival \oslash\IfValueT{#1}{_{#1}} \IfValueTF{#2}{#2}{(-)}}

\newcommand{\MCyl}[2]{\mathrm{M}_{#1}(#2)}

\newcommand{\KanSimp}{\widehat{\Simp}^{\mathrm{kq}}}
\newcommand{\SattlerDed}{\widehat{\square}_{\land\!\lor}^{\mathrm{ty}}}
\newcommand{\CMSOr}{\widehat{\square}_{\lor}^\mathrm{ty}}
\newcommand{\CMSIdComp}{\widehat{\overline{\square}}_{\lor}^{\raisebox{-4pt}{$\scriptstyle\mathrm{ty}$}}}
\newcommand{\ACCRS}{\widehat{\square}_{\mathrm{\times}}^{\textrm{eq}}}
\newcommand{\Test}[2][]{\widehat{#2}_{#1}^{\mathrm{test}}}
\newcommand{\TestOr}{\Test[\lor]{\square}}
\newcommand{\TestIdComp}{\widehat{\overline{\square}}_{\lor}^{\raisebox{-4pt}{$\scriptstyle\mathrm{test}$}}}

\DeclareMathOperator{\TFibStr}{TFib}
\DeclareMathOperator{\FibStr}{Fib}

\newcommand{\incat}[1]{i_{#1}}

\newcommand{\inftyGpd}{\mathbf{{\infty}\mh Gpd}}

\renewcommand{\deg}[1]{\lvert #1 \rvert}

\NewDocumentCommand{\lowering}{s}{\IfBooleanTF{#1}{\overset{-}\leftarrow}{\overset{-}\to}}
\NewDocumentCommand{\raising}{s}{\IfBooleanTF{#1}{\overset{+}\leftarrow}{\overset{+}\to}}

\newcommand{\degcore}[2]{{#2}[#1]}

\newcommand{\sk}[2]{\mathrm{sk}_{{<}#1}{#2}}

\newcommand{\bdyd}[2]{\partial_{#1}{#2}}
\newcommand{\bdyu}[2]{\partial^{#1}{#2}}
\newcommand{\inbdyd}[2]{\text{\textschwa}_{#1}{#2}}
\newcommand{\inbdyu}[2]{\text{\textschwa}^{#1}{#2}}

\newcommand{\latchob}[2]{L_{#1}{#2}}
\newcommand{\latchmap}[2]{\widehat{\ell}_{#1}{#2}}
\newcommand{\latchcod}[2]{{#2}_{#1}}

\newcommand{\elcore}[1]{{#1}^{\mathrm{ec}}}

\newcommand{\IdComp}[1]{\overline{#1}}
\newcommand{\IdCompOr}{\overline{\square}_\lor}
\newcommand{\IdCompDed}{\overline{\square}_{\land\!\lor}}

\NewDocumentCommand{\Alg}{m}{\mathrm{Alg}(#1)}
\newcommand{\TJoin}{\CatT_\lor}

\newcommand{\join}{\mathbin{\star}}

\newcommand{\tmsubst}[3]{{#1}[{#2}/{#3}]}
\newcommand{\tmface}[2]{{#1} \mapsto {#2}}
\newcommand{\tmcom}[5]{\mathsf{com}_{#1}^{{#2} \to {#3}}\,\left[#4\right]\,{#5}}

\newcommand{\Fence}{\mathfrak{F}}
\newcommand{\Crown}[1]{\mathfrak{C}_{#1}}
\newcommand{\crownproj}[1]{p_{#1}}
\newcommand{\wind}[1]{\mathrm{deg}(#1)}

\articletype{RESEARCH ARTICLE}
\jname{Canadian Mathematical Society}
\jyear{2020}

\begin{document}

\begin{Frontmatter}

\title[Relative Elegance and Cartesian Cubes with One Connection]{Relative elegance and cartesian cubes with one connection\thanks{The first author was supported by the Knut and Alice Wallenberg Foundation (KAW) under grants no.\ 2020.0266 and 2019.0116. The second author was supported by Swedish Research Council grant 2019-03765.}}

\author{Evan Cavallo}
\author{Christian Sattler}

\authormark{E. Cavallo and C. Sattler}

\address{\orgname{Department of Computer Science and Engineering, Chalmers University of Technology and University of Gothenburg}, \orgaddress{\city{Gothenburg}, \country{Sweden}}
\email{evan.cavallo@gu.se}}
\address{\orgname{Department of Computer Science and Engineering, Chalmers University of Technology and University of Gothenburg}, \orgaddress{\city{Gothenburg}, \country{Sweden}}
\email{sattler@chalmers.se}}



\date{\today}

\abstract{
  We establish a Quillen equivalence between the Kan--Quillen model structure and a model structure, derived
  from a cubical model of homotopy type theory, on the category of cartesian cubical sets with one
  connection. We thereby identify a second model structure which both constructively models homotopy type
  theory and presents $\infty$-groupoids, the first example being the equivariant cartesian model of
  Awodey--Cavallo--Coquand--Riehl--Sattler.
}

\end{Frontmatter}

\setcounter{tocdepth}{1}
\tableofcontents

\section{Introduction}

Homotopy type theory (HoTT)~\cite{hott-book} is said to be a language for reasoning in homotopical
settings. The conjecture (``Awodey's proposal'') goes that HoTT should have an interpretation in any
$(\infty,1)$-category belonging to some class of ``elementary $(\infty,1)$-topoi'', indeed, that models of
HoTT should be in correspondence with such $(\infty,1)$-categories. When one says that HoTT interprets in a
given $(\infty,1)$-category, one typically means more precisely that it admits a 1-categorical
\emph{presentation} interpreting HoTT in a 1-categorical sense. These presentations have historically come in
the form of \emph{Quillen model categories}. As an example, Voevodsky's interpretation of HoTT
\cite{kapulkin21} lands in the Kan--Quillen model structure on simplicial sets, which presents the
$(\infty,1)$-category $\inftyGpd$ of $(\infty,1)$-groupoids. Shulman \cite{shulman19} has now shown that every
Grothendieck $(\infty,1)$-topos can be presented by a model category that interprets HoTT.

The interests of type theorists have thus led to new questions in homotopy theory; one avenue is through
the search for \emph{constructive} interpretations of HoTT. The first constructive model to be discovered, due
to \citeauthor{bezem13} \cite{bezem13,bezem19}, interprets HoTT in a category of \emph{affine cubical sets},
presheaves over a certain \emph{affine cube category} $\Aff$ whose objects are symmetric monoidal products of
an interval object $I$. Subsequent constructions~\cite{cohen15,orton18,licata18,angiuli18,cavallo20,abcfhl}
use different cube categories to obtain better properties. With the exception of the BCH model, all employ
presheaves over a cube category with \emph{cartesian} products, \ie, including degeneracy, diagonal, and
permutation maps among its generators. While natural from a type-theoretic perspective, the presence of
diagonals---and to a lesser degree, permutations---is not typical in the homotopy-theoretic literature on
cubical structure.

Initially, none of these cubical models was shown to be compatible with a Quillen model structure; they were
models of HoTT (or of \emph{cubical type theories}) in the direct sense that they gave an interpretation of
the type-theoretic judgments, though they certainly made use of model-categorical intuitions. The connection
with model category theory is first made precise in~\cite{gambino17,sattler17}, where it is shown that
structure patterned on Cohen \etal's cubical set model~\cite{cohen15}---in particular, a functorial cylinder
with connections---gives rise to a Quillen model structure. These methods were adapted by Cavallo,
M\"{o}rtberg, and Swan~\cite{cavallo20} and Awodey~\cite{awodey23} to presheaves over cartesian cube
categories not necessarily supporting connections, producing model structures compatible with the type
theories and interpretations of Angiuli \etal~\cite{angiuli18,abcfhl}. Model structures in this lineage have
been called \emph{cubical-type model structures}.

It is now natural to ask which $(\infty,1)$-categories these model structures present. In particular, we would
like to know if any present $\inftyGpd$: such a presentation would be a constructive setting for standard
homotopy theory equipped with a constructive interpretation of HoTT, and could serve as a base case for
constructing further constructive models following Shulman~\cite{shulman19}. However, Buchholtz and Sattler
determined in 2018 \cite{coquand18c,sattler18} that almost all concrete cubical-type model structures
considered up to that point present $(\infty,1)$-categories \emph{in}equivalent to $\inftyGpd$. The exception
is the Sattler model structure $\SattlerDed$ on presheaves on the \emph{Dedekind cube category} $\Ded$, the
cube category with cartesian structure and both connections, whose status remains an open problem.

\subsubsection*{Cubes with one connection}

The difficulty in analyzing the Dedekind cube category $\Ded$ is that it is not a \emph{(generalized) Reedy
  category} \cite{berger11}, one in which each object is associated an ordinal degree and any morphism factors
as a degeneracy-like degree-lowering map followed by a face-like degree-raising map. Any presheaf over a Reedy
category can be built up inductively by attaching cells drawn from a set of generators, namely quotients of
representables by automorphism subgroups. In the subclasses of \emph{elegant} or \emph{Eilenberg-Zilber (EZ)}
categories, this cellular decomposition is moreover homotopically well-behaved with respect to any model
structure in which the cofibrations are the monomorphisms: it exhibits any presheaf as the \emph{homotopy}
colimit of basic cells. The problem in $\Ded$ is the combination of connections and diagonals, exemplified the
morphism $(x,y,z) \mapsto (x \lor y, y \lor z, x \land y)$ from the $3$-cube to itself. This map has no (split
epi, mono) factorization, a state of affairs forbidden in an elegant Reedy category.%
\footnote{A simpler map without a (split epi, mono) factorization in $\Ded$ is $(x,y) \mapsto (x, x \lor y)$,
  but this is an idempotent and so admits such a factorization in the idempotent completion $\IdCompDed$
  (characterized in \cite[Theorem 2.1]{sattler19}). The aforementioned 3-cube endomap does not: it does have
  an (epi, mono) factorization in $\IdCompDed$, but the left map does not split. It is the idempotent
  completion that counts when we consider whether elegant Reedy techniques apply.}

Thus, while \citeauthor{sattler19}~\cite{sattler19} and \citeauthor{streicher21}~\cite{streicher21} have
identified an adjoint triple of Quillen adjunctions relating $\SattlerDed$ and $\KanSimp$, it is not known
whether there is a Quillen equivalence. In particular, it is unclear how to prove that a round-trip composite
$\SattlerDed \to \KanSimp \to \SattlerDed$ is weakly equivalent to the identity in the absence of an elegant
Reedy structure on $\Ded$.

In this article we consider an overlooked cube category: the category $\Or$ of cubes with cartesian structure
and a \emph{single} connection. (We arbitrarily choose the ``max'' or ``negative'' connection, but this choice
plays no role.)  Presheaves on this category satisfy conditions sufficient to obtain a cubical-type model
structure $\CMSOr$ using existing techniques~\cite{cavallo20,awodey23}. Moreover, the arguments used in
\cite{sattler19,streicher21} adapt readily from $\Ded$ to $\Or$, providing a Quillen adjoint triple relating
$\CMSOr$ with $\KanSimp$.

Like the Dedekind cube category, $\Or$ is not Reedy. In this case, the archetypical problematic map is
$(x,y,z) \mapsto (x \lor y, y \lor z, z \lor x)$.\footnote{See \cref{sec:negative-disjunctive} for a proof
  that neither $\Or$ nor its idempotent completion is Reedy.} However, $\Or$ does \emph{embed} nicely in a
Reedy category, namely the category of finite inhabited join-semilattices: we have a functor
$i \co \Or \to \SLatFNE$ sending the $n$-cube to the $n$-fold product of the poset $\{0 < 1\}$.
While $\SLatFNE$ is not itself elegant, it satisfies a relativized form of elegance with respect to the
subcategory $\Or$. Whereas elegance would require the Yoneda embedding $\yo \co \SLatFNE \to \PSh{\SLatFNE}$
to preserve pushouts of spans of degeneracy maps, here it is the nerve
$\nerve{i} \defeq \subst{i}\yo \co \SLatFNE \to \PSh{\Or}$ that preserves such pushouts. We say that
$\SLatFNE$ is \emph{elegant relative to $i$}, or that $i$ is an \emph{elegant embedding}.

We find that the useful properties of elegant Reedy categories can be extended, in an appropriately
relativized form, to categories $\CatC$ with an elegant embedding $i \co \CatC \to \CatR$ in a Reedy
category. In particular, we show that any presheaf over $\CatC$ admits a homotopically well-behaved cellular
decomposition whose cells are automorphism quotients of objects in the image of $\nerve{i}$. With these tools
in hand, we are able to establish that the Quillen adjunctions relating $\CMSOr$ and $\KanSimp$ are Quillen
equivalences. We thus identify a cubical-type model structure presenting $\inftyGpd$, compatible with a
constructive interpretation of either HoTT or of cubical type theory with one connection.

\subsubsection*{Outline}

We begin in \cref{sec:background} with a brief review of model structures, Quillen equivalences, Reedy
categories, and the Kan--Quillen model structure on simplicial sets. In \cref{sec:model-structure-recognition},
we present an improvement on the first part of \cite{sattler17}: a series of increasingly specialized criteria
under which candidate (cofibration, trivial fibration) and (trivial cofibration, fibration) factorization
systems induce a model structure, culminating in a theorem tailored to models of type theory with universes.

In \cref{sec:disjunctive-cubical-sets}, we introduce the cube category $\Or$ and its basic properties,
construct the cubical-type model structure on $\PSh{\Or}$ using the results of the previous section, and define
a \emph{triangulation} adjunction $\adjunction{\Triang}{\PSh{\Or}}{\PSh{\Simp}}{\nerve{\triang}}$. We moreover
characterize the cube category's idempotent completion $\IdCompOr$. The categories of presheaves on $\Or$ and
$\IdCompOr$ are equivalent, but by working with the latter we can more easily compare with the simplex
category, following \cite{sattler19,streicher21}. In particular we have an embedding
$\insimp \co \Simp \to \IdCompOr$, thus an adjoint triple
$\lan{\insimp} \dashv \subst{\insimp\!} \dashv \ran{\insimp}$ relating $\PSh{\Simp}$ and $\PSh{\IdCompOr}$;
the triangulation adjunction corresponds to $\subst{\insimp\!} \dashv \ran{\insimp}$ along the equivalence
$\PSh{\Or} \simeq \PSh{\IdCompOr}$. In \cref{sec:simp-cub-quillen-adjunction} we show that both
$\lan{\insimp} \dashv \subst{\insimp\!}$ and $\subst{\insimp\!} \dashv \ran{\insimp}$ are Quillen adjunctions.

We focus on the adjunction $\lan{\insimp} \dashv \subst{\insimp\!}$. It is easy to see that its derived unit
is valued in weak equivalences, as $\insimp$ is fully faithful. To show its derived counit is valued in weak
equivalences, we spend \cref{sec:relatively-elegant} developing a theory of relative elegance. In
\cref{sec:algebra-elegant-core}, we show that the functor $i \co \Or \to \SLatFNE$ is relatively
elegant by way of a general analysis of Reedy categories of finite algebras.  In \cref{sec:equivalences} we
use this result to complete the Quillen equivalence between $\CMSOr$ and $\KanSimp$. We show first that
$\lan{\insimp} \dashv \subst{\insimp\!}$ is a Quillen equivalence, then deduce that
$\subst{\insimp\!} \dashv \ran{\insimp}$ is one as well, concluding with our main theorem as an immediate
corollary:

\begin{manualtheorem}{{\ref{triangulation-quillen-equivalence}}}
  The triangulation-nerve adjunction $\adjunction{\Triang}{\CMSOr}{\KanSimp}{\nerve{\triang}}$ is a
  Quillen equivalence.
\end{manualtheorem}

As a final corollary, we show in \cref{sec:test} that $\CMSOr$ coincides with Cisinski's \emph{test model
  structure} on $\PSh{\Or}$.

In \cref{sec:negative}, we give proofs of some negative results concerning Reedy structures on cartesian cube
categories with connections. First, we check that neither $\Or$ nor its idempotent completion supports a Reedy
structure, justifying our recourse to relative elegance. Second, we prove that $\Ded$ does not embed elegantly
in \emph{any} Reedy category, showing that our techniques cannot be applied in the two-connection case.

\subsection{Related work}

\subsubsection{Cartesian cubes}

This work's closest relative is the \emph{equivariant model structure} $\ACCRS$ on presheaves over the
cartesian cube category $\Cart$ constructed by Awodey, Cavallo, Coquand, Riehl, and Sattler (ACCRS)
\cite{accrs}, which also classically presents $\inftyGpd$.
The ACCRS construction is a modification of earlier models in presheaves on
$\Cart$~\cite{abcfhl,cavallo20,awodey23}. Briefly, where the definition of fibration involves lifting against maps
$1 \to \ival$ from the point to the interval, the definition of equivariant fibration involves lifting against
maps $1 \to \ival^n$ for all $n$ and requires lifts stable under permutations of $\ival^n$. Like our own model
structure, $\ACCRS$ is compatible with a constructive interpretation of HoTT.

In $\CMSOr$, equivariance does not appear explicitly but is still implicitly present: when the interval
supports a connection operator, ordinary and equivariant lifting become interderivable (see
\cref{equivariant-lifting-from-connection}). Our model structure may thus be seen as an instance of the
equivariant model structure construction applied in $\PSh{\Or}$, one which happens to admit a simpler
description.

\subsubsection{Test category theory}

Buchholtz and Morehouse~\cite{buchholtz17} catalogue a number of categories of cubical sets, specifically
investigating cube categories used in models of HoTT such as $\Cart$, $\Ded$, and the De Morgan cube
category. They observe that these categories are all \emph{test categories}, thus that each supports a
\emph{test model structure} equivalent to $\KanSimp$ \cite{cisinski06}. To our knowledge, however, none of
these model structures is known to be compatible with a model of HoTT with the exception of the test model
structure on $\Cart$, which coincides with $\ACCRS$ \cite[Theorem 6.3.6]{accrs}. As a corollary of
our Quillen equivalence, we check in \cref{sec:test} that $\CMSOr$ coincides with the test model structure on
$\Or$. Cisinski~\cite{cisinski14} does show that the test model structure on any elegant strict (that is,
non-generalized) Reedy category is compatible with a model of HoTT, but the strictness condition precludes
application to any cube category with permutations.

\subsubsection{Cubes with one connection}

To our knowledge, the category of cubes with cartesian structure and $\lor$-connections (or
$\land$-connections) has not been studied before, except in passing by
\citeauthor{buchholtz17}~\cite{buchholtz17}, though cartesian cube categories with both $\lor$- and
$\land$-connections have been used in interpretations of HoTT beginning with
\citeauthor{cohen15}~\cite{cohen15}.

On the other hand, subcategories without diagonals have seen use in classical homotopy theory. Indeed,
\citeauthor{brown81} use the cube category generated by faces, degeneracies, and $\lor$-connections in their
seminal article introducing connections for cubical sets~\cite{brown81}. Isaacson \cite{isaacson11} studies
the cube category with faces, degeneracies, symmetries, and $\land$-connections. Unlike $\Or$, these are
elegant Reedy categories \cites[Remarque 5.6]{maltsiniotis09}[Proposition 3.4]{isaacson11}: connections are
only problematic in combination with diagonals. They furthermore have useful properties compared to the
minimal cube category (generated by faces and degeneracies). For one, they are \emph{strict} test
categories~\cites{maltsiniotis09}[Theorem 3]{buchholtz17}, meaning that the localization functor from the test
model structures on these cubical sets to their homotopy categories preserves products.

It should be noted, however, that this particular distinction disappears in the cartesian cases: any cube
category with cartesian structure is a strict test category, regardless of the presence of connections
\cite[Corollary 2]{buchholtz17}. For us, the convenient properties of $\Or$ relative to $\Cart$ are
\begin{enumerate*}
\item the existence of an embedding from the simplex category into the idempotent completion of $\Or$, which
  facilitates the comparison between their presheaf categories, and
\item the existence of a contracting homotopy of each $n$-cube invariant under permutations, namely
  $(x_1,\ldots,x_n,t) \mapsto (x_1 \lor t, \ldots, x_n \lor t) \co [1]^n \times [1] \to [1]^n$.
\end{enumerate*}

\subsubsection{Constructive simplicial models}
\label{sec:constructive-simplicial}

Another line of work aims to reformulate the Kan--Quillen model structure and Voevodsky's simplicial model of
HoTT so that these can be obtained constructively. Bezem, Coquand, and Parmann
\cite{bezem15-kripke,bezem15-constructivity,parmann18} show that fibrations as usually defined\footnote{
  \cite{bezem15-kripke,parmann18} prove obstructions for a definition of fibration where lifting is treated as
  an operation, while \cite{bezem15-constructivity} considers fibrations requiring mere existence of a lift.
} %
in $\KanSimp$ do not provide a model of HoTT constructively; in particular, they are not closed under
pushforward along fibrations, which is necessary to interpret $\Pi$-types. These obstructions are avoided in
the cubical models by working with \emph{uniform fibrations}, which classically coincide with ordinary
fibrations but provide necessary extra structure in the constructive case. However, there are obstructions to
constructing a universe classifying uniform fibrations in simplicial sets \cites[Appendix
D]{van-den-berg22}[\S8.4.1]{swan22}.

\citeauthor{henry19} \cite{henry19} discovered that the Kan--Quillen model structure can be constructivized by
instead modifying the class of cofibrations, in particular taking a simplicial set to be cofibrant only when
degeneracy of its cells is decidable. Alternative constructions of the same model structure were later
presented by \citeauthor{gambino22b} \cite{gambino22b}. Gambino and Henry \cite{gambino22a} exhibit a
constructive form of Voevodsky's simplical model of HoTT using these ideas. The problem is not entirely
settled, however: the \emph{left adjoint splitting} coherence construction \cite{lumsdaine15}, applied to the
classical simplicial model to obtain a strict model of type theory, does not apply constructively in this case
\cite[Remark 8.5]{gambino22a}. There has since been progress on coherence theorems that do apply here
\cite{bocquet21,gambino21}, but the question is not to our knowledge fully resolved. Separately, van den Berg
and Faber \cite{van-den-berg22} have identified and developed a theory of \emph{effective fibrations} of
simplicial sets, which are both closed under pushforward and support a classifying universe, but have not yet
addressed the interpretation of univalence.

\subsubsection{Constructivity}

Though our interest in cubical-type model structures is motivated by constructive concerns, we work entirely
and incautiously within a classical metatheory in this article, our goal being an equivalence with a
classically defined model structure.  Given that $\CMSOr$ is constructively definable, however, it is natural
to wonder whether it is \emph{constructively} equivalent with the ACCRS or constructive simplicial model
structures. We leave this question for the future, referring to~\citeauthor{shulman23}~\cite{shulman23} for
further discussion of the constructive homotopy theory of spaces.

We note that the triangulation functor $\Triang \co \PSh{\Or} \to \PSh{\Simp}$ (\cref{def:triangulation}) is
definitely \emph{not} a left Quillen adjoint from $\CMSOr$ to~\citeauthor{henry19}'s simplicial model
structure constructively, as it does not preserve cofibrations unless the excluded middle holds.  The
(triangulation, nerve) adjunction exhibits $\PSh{\Simp}$ as a reflective subcategory of $\PSh{\Or}$, so every
simplicial set is the triangulation of some cubical set. But while every cubical set is cofibrant in $\CMSOr$,
not every simplicial set is cofibrant in~\citeauthor{henry19}'s model structure. For example, given a
subsingleton set $P$, the pushout of the span $\Delta^1 \leftarrow \Delta^1 \times P \to P$ is cofibrant if
and only if $P$ is decidable.

\subsubsection{Reedy, non-Reedy, and Reedy-like categories}

Campion \cite{campion23} studies the existence and non-existence of elegant Reedy structures on various cube
categories, among them $\Or$ (under the name $\square_{d,c^{\lor},s}$). A few observations are made
independently in that article and our own; in particular, \cite[Proposition 8.3]{campion23} is our
\cref{or-idempotent-completion}, while \cite[Theorem 8.12(2)]{campion23} follows from our
\cref{idcompor-not-reedy}.

Shulman's \emph{almost c-Reedy categories}~\cite[Definition 8.8]{shulman15} generalize beyond generalized
Reedy categories. These allow for non-isomorphisms that do not factor through a lower-degree object, so one
may wonder if the aforementioned pathological map $u \co [1]^3 \to [1]^3$ in $\Or$ (and $\IdCompOr$) defined
by $(x,y,z) \mapsto (x \lor y, y \lor z, z \lor x)$ can be accommodated in this way. However, the class of
degree-preserving maps not admitting a lower-degree factorization must be closed under
composition~\cite[Theorem 8.13(ii)]{shulman15}. While $u$ factors through no lower-dimensional object,
$uu$ factors through the 1-cube. As such, this generalization is unlikely to be helpful here.

\subsection{Acknowledgments}

We thank Steve Awodey, Thierry Coquand, and Emily Riehl, our collaboration with whom inspired this spin-off
project, for their suggestions and feedback. We also thank Emily Riehl for alerting us to errors in the first
preprint version of this article. The idea of embedding non-Reedy cube categories in larger Reedy categories
came to us via Matthew Weaver and Daniel Licata, who experimented with (but did not ultimately use) this
strategy in work on cubical models of directed type theory \cite{weaver20}. The first author thanks Brandon
Doherty, Anders M\"{o}rtberg, Axel Ljungström, and Matthew Weaver for helpful conversations. We credit an
observation of \citeauthor*{imrich14}\ \cite[Lemma 2]{imrich14} for inspiring the argument in
\cref{sec:dedekind-not-relative-elegant}.

\section{Background}
\label{sec:background}

\subsection{Preliminaries}

We begin by fixing a few notational conventions.

\begin{notation}
  We write $\Func{\CatE}{\CatF}$ for the category of functors from $\CatE$ and $\CatF$. We write
  $\PSh{\CatC} \defeq \Func{\op{\CatC}}{\Set}$ for the category of presheaves on a category $\CatC$ and
  $\yo \co \CatC \to \PSh{\CatC}$ for the Yoneda embedding.
\end{notation}

\begin{notation}
  When regarding a functor as a diagram, we use superscripts for covariant indexing and subscripts for
  contravariant indexing. Thus if $F \co \CatD \to \CatE$ then we have $F^d \in \CatE$ for $d \in \CatD$,
  while if $F \co \op{\CatC} \to \CatE$ then we have $F_c \in \CatE$ for $c \in \CatC$. We sometimes partially
  apply a multi-argument functor: given $F \co \op{\CatC} \times \CatD \to \CatE$ and $c \in \CatC$,
  $d \in \CatD$, we have $F_c \in \CatD \to \CatE$, $F^d \in \op{\CatC} \to \CatE$, and $F_c^d \in \CatE$.
\end{notation}

By a \emph{bifunctor} we mean a functor in two arguments. We make repeated use of the \emph{Leibniz
  construction} \cite[Definition 4.4]{riehl14b}, which transforms a bifunctor into an bifunctor on arrow
categories.

\begin{definition}
  Given a bifunctor ${\odot} \co \CatC \times \CatD \to \CatE$ into a category $\CatE$ with pushouts, the
  \emph{Leibniz construction} defines a bifunctor
  ${\leib{\odot}} \co \CatC^\to \times \CatD^\to \to \CatE^\to$, with $f \leib{\odot} g$ defined for
  $f \co A \to B$ and $g \co X \to Y$ as the following induced map:
  \[
    \begin{tikzcd}[column sep=large]
      A \odot X \ar{d}[left]{A \odot g} \ar{r}{f \odot X} & B \odot X \ar{d} \ar[bend left]{ddr}{B \odot g}\\
      A \odot Y \ar{r} \ar[bend right]{drr}[below left]{f \odot Y} & \bullet \pushout \ar[dashed]{dr} \\[-1em]
      & &[-4em] B \odot Y \rlap{.}
    \end{tikzcd}
  \]
\end{definition}

\begin{example}
  If $\CatE$ is a category with binary products and pushouts, applying the Leibniz construction to the binary
  product functor ${\times} \co \CatE \times \CatE \to \CatE$ produces the \emph{pushout product} bifunctor
  $\lbtimes \co \CatE^\to \times \CatE^\to \to \CatE^\to$.
\end{example}

\subsection{Model structures and Quillen equivalences}

In the abstract, the force of our result is that a certain model category presents the $(\infty,1)$-category
of $\infty$-groupoids. Concretely, we work entirely in model-categorical terms, exhibiting a Quillen
equivalence between this model category and another model category---simplicial sets---already known to
present $\inftyGpd$. We briefly fix the relevant basic definitions here but assume prior familiarity,
especially with factorization systems; standard references include~\cite{hovey99,dwyer04}.

\begin{definition}
  A \emph{model structure} on a category $\CatM$ is a triple $(\CC, \CW, \CF)$ of classes of morphisms in
  $\CatM$, called the \emph{cofibrations}, \emph{weak equivalences}, and \emph{fibrations} respectively, such
  that $(\CC, \CF \cap \CW)$ and $(\CC \cap \CW, \CF)$ are weak factorization systems and $\CW$ satisfies the
  2-out-of-3 property. A \emph{model category} is a finitely complete and cocomplete category equipped with a model
  structure. We use the arrow $\cof$ for cofibrations, $\weq$ for weak equivalences, and $\fib$ for
  fibrations. Maps in $\CC \cap \CW$ and $\CF \cap \CW$ are called \emph{trivial} cofibrations and fibrations
  respectively.

  We say that a model structure on $\CatM$ \emph{has monos as cofibrations} when its class of cofibrations is
  exactly the class of monomorphisms in $\CatM$.\footnote{Such a model structure which is also cofibrantly
    generated (see below) is called a \emph{Cisinski model structure}, these being the subject of
    \cite{cisinski06}.}
\end{definition}

\begin{definition}
  We say an object is \emph{cofibrant} when $0 \to A$ is a cofibration, dually \emph{fibrant} if $A \to 1$ is
  a fibration. The weak factorization system $(\CC, \CF \cap \CW)$ implies that for every object $A$, we have a
  diagram $0 \cof A^\lcof \tfib A$ obtained by factorizing $0 \to A$; we say such an $A^\lcof$ is a
  \emph{cofibrant replacement} of $A$. Likewise, an object $A^\lfib$ sitting in a diagram
  $A \tcof A^\lfib \fib 1$ is a \emph{fibrant replacement} of $A$.
\end{definition}

\begin{definition}
  We say an object $X$ in a model category is \emph{weakly contractible} when the map $X \to 1$ is a weak equivalence.
\end{definition}

Note that given any two of the classes $(\CC, \CW, \CF)$, we can reconstruct the third: $\CC$ is the class of
maps with left lifting against $\CF \cap \CW$, $\CF$ is the class of maps with right lifting against
$\CC \cap \CW$, and $\CW$ is the class of maps that can be factored as a map with left lifting against $\CF$
followed by a map with right lifting against $\CC$. We will thus frequently introduce a model category by
giving a description of two of its classes.

The two factorization systems are commonly generated by \emph{sets} of left maps:

\begin{definition}
  We say a weak factorization system $(\CL,\CR)$ on a category $\CatE$ is \emph{cofibrantly generated} by some
  set $S \subseteq \CL$ when $\CR$ is the class of maps with the right lifting property against all maps in
  $S$. A model structure is cofibrantly generated when its component weak factorization systems are.
\end{definition}

Now we come to relationships between model categories.

\begin{definition}
  A \emph{Quillen adjunction} between model categories $\CatM$ and $\CatN$ is a pair of adjoint functors
  $\adjunction{F}{\CatM}{\CatN}{G}$ such that $F$ preserves cofibrations and $G$ preserves fibrations.
\end{definition}

Note that $F$ preserves cofibrations if and only if $G$ preserves trivial fibrations, while $G$ preserves
fibrations if and only if $F$ preserves trivial cofibrations.

\begin{definition}
  A Quillen adjunction $\adjunction{F}{\CatM}{\CatN}{G}$ is a \emph{Quillen equivalence} when
  \begin{itemize}
  \item for every cofibrant $X \in \CatM$, the \emph{derived unit}
    $X \overset{\Gh_X}\to GFX \overset{Gm}\to G((FX)^{\lfib})$ is a weak equivalence for some
    fibrant replacement $m \co FX \tcof (FX)^{\lfib}$;
  \item for every fibrant $Y \in \CatN$, the \emph{derived counit}
    $F((GY)^{\lcof}) \overset{Fp} \to FGY \overset{\Ge_Y}\to Y$ is a weak equivalence for some
    cofibrant replacement $p \co (GY)^{\lcof} \tfib GY$.
  \end{itemize}
  Two model structures are \emph{Quillen equivalent} when there is a zigzag of Quillen equivalences connecting
  them.
\end{definition}

\subsection{Reedy categories and elegance}

The linchpin of our approach is Reedy category theory, the theory of diagrams over categories whose morphisms
factor into degeneracy-like and face-like components. As our base category of interest contains non-trivial
isomorphisms, we work more specifically with the \emph{generalized Reedy categories} introduced by Berger and
Moerdijk~\cite{berger11}.

\begin{definition}
  \label{def:reedy}
  A \emph{(generalized) Reedy structure} on a category $\CatR$ consists of an orthogonal factorization system
  $(\CatR^-,\CatR^+)$ on $\CatR$ together with a \emph{degree map} $\deg{-} \co \Ob{\CatR} \to \BN$,
  compatible in the following sense: given $f \co a \to b$ in $\CatR^-$ (\resp $\CatR^+$), we have
  $\deg{a} \ge \deg{b}$ (\resp $\deg{a} \le \deg{b}$), with $\deg{a} = \deg{b}$ only if $f$ is invertible.

  We refer to maps in $\CatR^-$ as \emph{lowering maps} and maps in $\CatR^+$ as \emph{raising maps}, and we
  use the annotated arrows $\lowering$ and $\raising$ to denote lowering and raising maps respectively. The
  \emph{degree} of a map is the degree of the intermediate object in its Reedy factorization. Note that this
  definition is self-dual: if $\CatR$ is a Reedy category, then $\op{\CatR}$ is a Reedy category with the same
  degree function but with lowering and raising maps swapped.
\end{definition}

\begin{terminology}
  We henceforth drop the qualifier \emph{generalized}, as we are almost always working with generalized Reedy
  categories. Instead, we say a Reedy category is \emph{strict} if any parallel isomorphisms are equal and it
  is skeletal, \ie, it is a Reedy category in the original sense.
\end{terminology}

The prototypical strict Reedy category is the simplex category $\Simp$: the degree of an $n$-simplex is $n$,
while the lowering and raising maps are the degeneracy and face maps respectively \cite[\S II.3.2]{gabriel67}.

A Reedy structure on a category $\CatR$ is essentially a tool for working with $\CatR$-shaped diagrams. For
example, a weak factorization system on any category $\CatE$ induces \emph{injective} and \emph{projective
  Reedy weak factorization systems} on the category $\Func{\CatR}{\CatE}$ of $\CatR$-shaped diagrams in
$\CatE$; likewise for model structures. Importantly for us, any diagram of shape $\CatR$ can be regarded as
built iteratively from ``partial'' diagrams specifying the elements at indices up to a given degree. We are
specifically interested in presheaves, \ie, $\op{\CatR}$-shaped diagrams in $\Set$. We refer to
\cites[\S22]{dwyer04}{berger11}{riehl14b}{shulman15} for overviews of Reedy categories and their applications.

\citeauthor{berger11}'s definition of generalized Reedy category \cite[Definition 1.1]{berger11} includes one
additional axiom. Following \citeauthor{riehl17b}~\cite{riehl17b}, we treat this as a property to be assumed
only where necessary:

\begin{definition}
  In a Reedy category $\CatR$, we say \emph{isos act freely on lowering maps} when for any
  $e \co r \lowering s$ and isomorphism $\Gth \co s \cong s$, if $\Gth e = e$ then $\Gth = \id$.
\end{definition}

Note that any Reedy category in which all lowering maps are epic satisfies this property. The main
results of this paper are restricted to \emph{pre-elegant} Reedy categories (\cref{def:pre-elegant}) for which
this is always the case (\cref{pre-elegant-lowering-epi}); nevertheless, we try to record where only the
weaker assumption is needed.

The following cancellation property will come in handy.

\begin{lemma} \label{reedy-cancellation}
  Let $f \co r \to s$, $g \co s \to t$ be maps in a Reedy category. If $gf$ is a lowering map, then so
  is $g$. Dually, if $gf$ is a raising map, then so is $f$.
\end{lemma}
\begin{proof}
  We prove the first statement; the second follows by duality.
  Suppose $gf$ is a lowering map.
  We take Reedy factorizations $f = me$, $g = m'e'$, and then $e'm = m''e''$:
  \[
    \begin{tikzcd}[row sep=small]
      r \ar[dashed,lowering=above]{dr}[below left]{e} \ar{rr}{f} && s \ar[dashed,lowering=above]{dr}[below left]{e'} \ar{rr}{g} && r \rlap{.} \\
      & t \ar[dashed,lowering=above]{dr}[below left]{e''} \ar[dashed,raising=above]{ur}[below right]{m} && t' \ar[dashed,raising=above]{ur}[below right]{m'} \\
      & & t'' \ar[dashed,raising=above]{ur}[below right]{m''}
    \end{tikzcd}
  \]
  This gives us a Reedy factorization $gf = (m'm'')(e''e)$.
  By uniqueness of factorizations, $m'm''$ must be an
  isomorphism; this implies $\deg{t''} = \deg{t'} = \deg{r}$, so $m'$ and $m''$ are also isomorphisms.  Thus
  $g \cong e'$ is a lowering map.
\end{proof}

\begin{corollary}
  \label{split-epi-is-lowering}
  Any split epimorphism in a Reedy category is a lowering map; dually, any split monomorphism is a raising
  map.
\qed
\end{corollary}

When studying $\Set$-valued presheaves over a Reedy category, it is useful to consider the narrower class of
\emph{elegant Reedy categories} \cite{berger11,bergner13}.

\begin{definition}
  \label{def:elegant-reedy}
  A Reedy structure on a category $\CatR$ is \emph{elegant} when
  \begin{enumerate}[label=(\alph*)]
  \item any span $s \overset{e}\leftarrow r \overset{e'}\to s'$ consisting of lowering maps $e,e$' has a
    pushout;
  \item the Yoneda embedding $\yo \co \CatR \to \PSh{\CatR}$ preserves these pushouts.
  \end{enumerate}
  We refer to spans consisting of lowering maps as \emph{lowering spans}, likewise pushouts of such spans as
  \emph{lowering pushouts}. Note that all the maps in a lowering pushout square are lowering maps, as the left
  class of any factorization system is closed under cobase change.
\end{definition}

Intuitively, an elegant Reedy category is one where any pair of ``degeneracies'' $s \lowering* r \lowering s'$
has a universal ``combination'' $r \lowering s \sqcup_r s'$, namely the diagonal of their pushout. The
condition on the Yoneda embedding asks that any $r$-cell in a presheaf is degenerate along (that is, factors
through) both $r \lowering s$ and $r \lowering s'$ if and only if it is degenerate along their
combination. Again, the simplex category is the prototypical elegant Reedy category
\cite[\S{II.3.2}]{gabriel67}.

\begin{remark}
  \label{eilenberg-zilber}
  This definition is one of a few equivalent formulations introduced by \citeauthor{bergner13}
  \cite[Definition 3.5, Proposition 3.8]{bergner13} for strict Reedy categories. For generalized Reedy
  categories, \citeauthor{berger11} \cite[Definition 6.7]{berger11} define \emph{Eilenberg-Zilber} (or
  \emph{EZ}) \emph{categories}, which additionally require that $\CatR^+$ and $\CatR^-$ are exactly the
  monomorphisms and split epimorphisms respectively. We make do without this restriction. It is always the
  case that the lowering maps in an elegant Reedy category are the split epis (see
  \cref{elegant-lowering-split-epi} below), but the raising maps need not be monic. For example \cite[Example
  4.3]{campion23}, any direct category (that is, any Reedy category with $\CatR^+ = \CatR^\to$) is elegant,
  but a direct category can contain non-monic arrows.
\end{remark}

A presheaf $X \in \PSh{\CatR}$ over any Reedy category can be written as the sequential colimit of a sequence
of \emph{$n$-skeleta} containing non-degenerate cells of $X$ only up to degree $n$, with the maps between
successive skeleta obtained as cobase changes of certain basic \emph{cell maps}. When $\CatR$ is elegant,
these cell maps are moreover monic. This property gives rise to a kind of induction principle: any
property closed under certain colimits can be verified for all presheaves on an elegant Reedy category by
checking that it holds on basic cells. This principle is conveniently encapsulated by the following
definition.

\begin{definition}[{{\thmcite[Definition 1.3.9]{cisinski19}}}]
  \label{def:saturated-by-monos}
  Let a category $\CatE$ be given. We say a replete class of objects $\CP \subseteq \CatE$ is \emph{saturated
    by monomorphisms} when
  \begin{enumerate}[label=(\alph*)]
  \item $\CP$ is closed under small coproducts;
  \item For every pushout square
    \[
      \begin{tikzcd}
        X \ar[tail]{d}[left]{m} \ar{r} & X' \ar[tail]{d}{m'} \\
        Y \ar{r} & Y' \pushout
      \end{tikzcd}
    \]
    such that $X,X',Y \in \CP$, we have $Y' \in \CP$;
  \item For every diagram $X \co \Gw \to \CatE$ such that each object $X^i$ is in $\CP$ and each morphism
    $X^i \to X^{i+1}$ is monic, we have $\colim_{i < \Gw} X^i \in \CP$.
  \end{enumerate}
\end{definition}

We note that when $\CatE$ is a model category with monos as cofibrations, these are all diagrams
whose colimits agree with their \emph{homotopy colimits}: we can compute their colimits in the
$(\infty,1)$-category presented by $\CatE$ by simply computing their 1-categorical colimits in $\CatE$, which
is hardly the case in general. This fact is another application of Reedy category theory; see for example
Dugger \cite[\S14]{dugger08}. As a result, these colimits have homotopical properties analogous to
1-categorical properties of colimits. For example, recall that given a natural transformation
$\Ga \co F \to G$ between left adjoint functors $F,G \co \CatE \to \CatF$, the class of $X \in \CatE$ such
that $\Ga_X$ is an isomorphism is closed under colimits. If $F,G$ are left \emph{Quillen} adjoints and
$\CatE,\CatF$ have monomorphisms as cofibrations, then the class of $X$ such that $\Ga_X$ is a \emph{weak
  equivalence} is saturated by monomorphisms. This particular fact will be key in
\cref{sec:equivalence-kan-quillen}.

For presheaves over an elegant Reedy category, the basic cells are the quotients of representables by
automorphism subgroups.

\begin{definition}
  Given an object $X$ of a category $\CatE$ and a subgroup $H \le \Aut[\CatE]{X}$, their \emph{quotient} is
  the colimit $\autquo{X}{H} \defeq \colim\parens{H \to \Aut[\CatE]{X} \to \CatE}$.
\end{definition}

\begin{proposition}
  \label{saturated-elegant}
  Let $\CatR$ be an elegant Reedy category. Let $\CP \subseteq \PSh{\CatR}$ be a class of objects such that
  \begin{itemize}
  \item for any $r \in \CatR$ and $H \le \Aut[\CatR]{r}$, we have $\autquo{\yo{r}}{H} \in \CP$;
  \item $\CP$ is saturated by monomorphisms.
  \end{itemize}
  Then $\CP$ contains all objects of $\PSh{\CatR}$.
\end{proposition}
\begin{proof}
  \cite[Corollary 1.3.10]{cisinski19} gives a proof for strict elegant Reedy categories; the proof for the
  generalized case is similar (and a special case of our \cref{saturated-elegant-relative}).
\end{proof}

As described above, we will be studying a category $\Or$ that is \emph{not} a Reedy category. Thus, we will
not use the previous proposition directly. Instead, our \cref{sec:relatively-elegant} establishes a
generalization to categories that only \emph{embed} in a Reedy category in a nice way.

\subsection{Simplicial sets}
\label{sec:simplicial-sets}

To show that a given model category presents $\inftyGpd$, it suffices to exhibit a Quillen equivalence to a
model category already known to present $\inftyGpd$. Here, our standard of comparison will be the classical
Kan--Quillen model structure on simplicial sets~\cite[\S II.3]{quillen67}.

\begin{definition}
  The \emph{simplex category} $\Simp$ is the full subcategory of the category $\Pos$ of posets and monotone
  maps consisting of the finite inhabited linear orders $[n] \defeq \{0 < \cdots < n\}$ for $n \in \BN$.
\end{definition}

This is a strict Reedy category, in fact an Eilenberg-Zilber category (see \cref{eilenberg-zilber}).
The raising and lowering maps are given by the \emph{face} and \emph{degeneracy maps}, defined as the injective and surjective maps of posets, respectively.

\begin{definition}
  We define the usual generating maps of the simplex category:
  \label{def:simplex-degeneracy-face}
  \begin{itemize}
  \item given $n \geq 0$ and $i \in [n]$, the \emph{generating degeneracy map} $s_i \co [n+1] \to [n]$ identifies the elements $i$ and $i + 1$ of $[n+1]$,
  \item given $n \geq 1$ and $i \in [n]$, the \emph{generating face map} $d_i \co [n-1] \to [n]$ skips over the element $i$ of $[n]$.
  \end{itemize}
  \end{definition}

\begin{definition}
  Write $\simp{n} \in \PSh{\Simp}$ for the representable $n$-simplex $\yo{[n]}$. We define the following sets
  of maps in simplicial sets:
\begin{itemize}
\item
For $n \ge 0$, the \emph{boundary inclusion} $\simpbd{n} \mono \simp{n}$ is the union of the subobjects $\simp{i} \mono \simp{n}$ given by a non-invertible face map $[i] \to [n]$.
\item
For $n \ge 1$ and $0 \le k \le n$, the \emph{$k$-horn} $\simphorn{n}{k} \mono \simp{n}$ is the union of the subobjects $\simp{i} \mono \simp{n}$ given by a face map $d \co [i] \to [n]$ whose pullback along $[n] - k \mono [n]$ is non-invertible.
\end{itemize}
\end{definition}

\begin{proposition}[Kan--Quillen model structure]
  \label{simp-model-structure}
  There is a model structure on $\PSh{\Simp}$ with the following weak factorization systems:
  \begin{itemize}
  \item the weak factorization system (cofibration, trivial fibration) is cofibrantly generated by the boundary inclusions;
  \item the weak factorization system (trivial cofibration, fibration) is cofibrantly generated by the horn inclusions.
  \end{itemize}
  We write $\KanSimp$ for this model category.
\end{proposition}

\begin{proof}
This is Theorem~3 and the following Proposition~2 in~\cite[\S II.3]{quillen67}.
\end{proof}

\begin{proposition}[{{\thmcite[\S IV.2]{gabriel67}}}]
  \label{simp-model-structure-alt-descriptionn}
  The weak factorization systems of $\KanSimp$ admit the following alternative descriptions:
  \begin{itemize}
  \item the cofibrations are the monomorphisms; the trivial fibrations are the maps right lifting against monomorphisms.
  \item the weak factorization system (trivial cofibration, fibration) is generated by pushout products $d_k \lbtimes m$ of an endpoint inclusion $d_k \co 1 \to \simp{1}$ with a monomorphism $m \co A \mono B$. \qed
  \end{itemize}
\end{proposition}

\section{Model structures from cubical models of type theory}
\label{sec:model-structure-recognition}

As the cube category $\Or$ is cartesian, we may obtain our cubical-type model structure on $\PSh{\Or}$
immediately by applying existing arguments~\cite{cavallo20,awodey23}, which build on a criterion for
recognizing model structures introduced in the first part of~\cite{sattler17}. We will instead take the
opportunity to present an improvement on the latter criterion, hoping to give an idea of the character of
these model structures along the way.

We begin in \cref{sec:model-structure-general} with a set of conditions necessary and sufficient to determine
when a \emph{premodel structure}---essentially, all the ingredients of a model structure except 2-out-of-3 for
weak equivalences---is in fact a model structure. In \cref{sec:model-structure-cylindrical}, we give a
simplified set of conditions for the case where the premodel structure is equipped with a compatible adjoint
functorial cylinder. Finally, in \cref{sec:model-structure-fep} we show that such a cylindrical premodel
structure satisfies these conditions when all its objects are cofibrant and it satisfies the \emph{fibration
  extension property}. We shall apply this result in \cref{sec:disjunctive-model-structure} to obtain our
model structure on $\PSh{\Or}$; a reader who would prefer to take the existence of the model structure for
granted may skip this section and read only \cref{or-model-structure} in
\cref{sec:disjunctive-model-structure}.

\subsection{Model structures from premodel structures}
\label{sec:model-structure-general}

\begin{definition}[{{\thmcite[Definition 2.1.23]{barton19}}}]
  A \emph{premodel structure} on a finitely complete and cocomplete category $\CatM$ consists of weak
  factorization systems $(\CC,\CF_t)$ (the \emph{cofibrations} and \emph{trivial fibrations}) and
  $(\CC_t,\CF)$ (the \emph{trivial cofibrations} and \emph{fibrations}) on $\CatM$ such that
  $\CC_t \subseteq \CC$ (or equivalently $\CF_t \subseteq \CF$).
\end{definition}

\begin{remark}[Stability under (co)slicing] \label{premodel-structure-slices}
Given an object $X \in \CatM$, any weak factorization system on $\CatM$ descends to weak factorization systems on the slice over $X$ and the coslice under $X$, with left and right classes created by the respective forgetful functor to $\CatM$.
In the same fashion, any premodel structure on $\CatM$ descends to slices and coslices of $\CatM$.
\end{remark}

As any two of the classes $(\CC,\CW,\CF)$ defining a model structure determines the third, any premodel
structure induces a candidate class of weak equivalences.

\begin{definition}
  We say that a morphism in a premodel structure is a \emph{weak equivalence} if it factors as a
  trivial cofibration followed by a trivial fibration; we write $\candweq{\CC}{\CF}$ for the class of such
  morphisms.
\end{definition}

\begin{remark}
  The above definition is only necessarily appropriate when examining when a premodel structure forms a model
  structure: there are premodel structures with a useful definition of weak equivalence not agreeing with
  $\candweq{\CC}{\CF}$. For example, there are various \emph{weak model structures} on semisimplicial sets in
  which not all trivial fibrations are weak equivalences \cite[Remark 5.5.7]{henry20}.
\end{remark}

For the remainder of this section, we fix a premodel category $\CatM$ with factorization systems $(\CC,\CF_t)$
and $(\CC_t,\CF)$. The following two propositions are standard.

\begin{proposition}
  $\CC_t = \CC \cap \candweq{\CC}{\CF}$ and $\CF_t = \CF \cap \candweq{\CC}{\CF}$.
\end{proposition}
\begin{proof}
  An immediate consequence of the retract argument \cite[Lemma 1.1.9]{hovey99}.
\end{proof}

In light of the above, we use the arrows $\tcof$ and $\tfib$ to denote trivial cofibrations and fibrations
also in a premodel structure.

\begin{corollary}
  $(\CC,\candweq{\CC}{\CF},\CF)$ forms a model structure if and only if $\candweq{\CC}{\CF}$ satisfies
  2-out-of-3.
\end{corollary}

We now reduce the problem of checking 2-out-of-3 for $\candweq{\CC}{\CF}$ to a reduced collection of special
cases of 2-out-of-3 where some or all maps belong to $\CC$ or $\CF$.

\begin{definition}
  \label{cancellation}
  Given a wide subcategory $\CA \subseteq \CatE$ of a category $\CatE$, we say \emph{$\CA$ has left
    cancellation in $\CatE$} (or \emph{among maps in $\CatE$}) when for every composable pair $g,f$ in
  $\CatE$, if $gf$ and $g$ are in $\CA$ then $f$ is in $\CA$. Dually, $\CA$ has \emph{right cancellation in
    $\CatE$} when for all $g,f$ with $gf,f \in \CA$, we have $g \in \CA$.
\end{definition}

\begin{theorem}
  \label{premodel-to-model}
  $\candweq{\CC}{\CF}$ satisfies 2-out-of-3 exactly if the following hold:
  \begin{conditions}
  \item \label{pre:other-cancellation} trivial cofibrations have left cancellation among cofibrations and
    trivial fibrations have right cancellation among fibrations.
  \item \label{pre:for-all-factors} any (cofibration, trivial fibration) factorization or (trivial
    cofibration, fibration) factorization of a weak equivalence is a (trivial cofibration, trivial fibration)
    factorization;
  \item \label{pre:exchange} any composite of a trivial fibration followed by a trivial cofibration is a weak
    equivalence.
  \end{conditions}
  Note that each of these conditions is self-dual.
\end{theorem}
\begin{proof}
  \Cref{pre:other-cancellation,pre:for-all-factors,pre:exchange} all follow by straightforward applications of
  2-out-of-3 for $\candweq{\CC}{\CF}$. Suppose conversely that we have
  \labelcref{pre:other-cancellation,pre:for-all-factors,pre:exchange} and let maps $g \co Y \to Z$ and
  $f \co X \to Y$ be given. Then using the two factorization systems and \cref{pre:exchange}, we have the
  following diagram:
  \[
    \begin{tikzcd}
      X \ar{dr}[below left]{f} \ar[dashed,cof]{r} & U \ar[dashed,tfib]{d} \ar[dashed,tcof]{r} & W \ar[dashed,tfib]{d} \\
      & Y \ar{dr}[below left]{g} \ar[dashed,tcof]{r} & V \ar[dashed,fib]{d} \\
      & & Z \rlap{.}
    \end{tikzcd}
  \]

  Suppose first that $g$ and $f$ are weak equivalences. Then we may choose the factorizations of $f$ and $g$
  such that the map $X \cof U$ is a trivial cofibration and the map $V \fib Z$ is a trivial fibration. Thus
  $gf$ factors as a trivial cofibration followed by a trivial fibration, \ie, is a weak equivalence.

  Now suppose that $f$ and $gf$ are weak equivalences. We may choose the factorization of $f$ such that the
  map $X \cof U$ is a trivial cofibration. The composite $X \cof W$ is then a trivial cofibration, so the
  composite $W \fib Z$ is a trivial fibration by \cref{pre:for-all-factors}. Then the map $V \fib Z$ is a
  trivial fibration by \cref{pre:other-cancellation}. Hence $g$ is a weak equivalence. By the dual argument,
  if $g$ and $gf$ are weak equivalences then so is $f$.
\end{proof}

\subsection{Cylindrical premodel structures}
\label{sec:model-structure-cylindrical}

Now we derive a simpler set of criteria for premodel structures equipped with a compatible \emph{adjoint
  functorial cylinder}.

\begin{definition}
  A \emph{functorial cylinder} on a category $\CatE$ is a functor $\Cyl \co \CatE \to \CatE$ equipped with
  \emph{endpoint} and \emph{contraction} transformations fitting in a diagram as shown:
  \[
    \begin{tikzcd}[column sep=huge, row sep=huge]
      \Id \ar{dr}[sloped,below]{=} \ar[dashed]{r}{\Gd_0 \otimes (-)} & \Cyl \ar[dashed]{d}[description, yshift=1ex]{\Ge \otimes (-)} & \Id \ar[dashed]{l}[above]{\Gd_1 \otimes (-)} \ar{dl}[sloped,below]{=} \rlap{.} \\
      & \Id
    \end{tikzcd}
  \]
  An \emph{adjoint functorial cylinder} is a cylinder such that $\Cyl$ is a left adjoint.
\end{definition}

\begin{notation}
  Given a functorial cylinder in a finitely cocomplete category, we have induced \emph{boundary} maps
  $\partial \otimes X \defeq [\Gd_0 \otimes X, \Gd_1 \otimes X] \co X \sqcup X \to \Cyl{X}$.
\end{notation}

There is a dual notion of \emph{functorial path object} consisting of a functor $\Cocyl$ and natural
transformations $\Gd_k \oslash (-) \co \Cyl \to \Id$ and $\Ge \oslash (-) \co \Id \to \Cyl$. By transposition,
each adjoint functorial cylinder corresponds to an adjoint functorial path object.

\begin{remark}[Stability under (co)slicing] \label{functorial-cylinder-slices}
Fix a functorial cylinder denoted as above and an object $X \in \CatE$.
Then $\Cyl$ lifts through the forgetful functor $\Slice{\CatE}{X} \to \CatE$ to a functorial cylinder $\Cyl[\Slice{\CatE}{X}]$ on the slice over $X$.
This crucially uses the contraction.
For example, the action of $\Cyl[\Slice{\CatE}{X}]$ on $f \co Y \to X$ is given by $(\Ge \otimes X)(\Cyl{f}) \co \Cyl{Y} \to X$.
Furthermore, $\Cyl$ lifts through the pushout functor $\CatE \to \Coslice{\CatE}{X}$ to a functorial cylinder $\Cyl[\Coslice{\CatE}{X}]$ on the coslice under $X$.
For example, the action of $\Cyl[\Coslice{\CatE}{X}]$ on $f \co X \to Y$ is given by the pushout of $\Cyl{f} \co \Cyl{X} \to \Cyl{Y}$ along $\Ge \otimes X$.
In both cases, adjointness is preserved, and the corresponding functorial path object is given by performing the dual construction.
\end{remark}

\begin{definition}
  Write ${\app} \co \Func{\CatE}{\CatF} \times \CatE \to \CatF$ for the application bifunctor defined by
  $F \app X \defeq F(X)$. Given a category $\CatE$ with a functorial cylinder and $f \in \CatE^\to$, we
  abbreviate $(\Gd_k \otimes (-)) \leib{\app} f \in \CatE^\to$ as $\Gd_k \lbotimes f$. We
  likewise write $\Ge \lbotimes f$ for Leibniz application of the contraction. We write $\Gd_k \lboslash (-)$
  and $\Ge \lboslash (-)$ for the dual operations associated to a functorial path object.
\end{definition}

\begin{definition}
  Given a finitely cocomplete category $\CatE$ with a functorial cylinder, a weak factorization system
  $(\CL,\CR)$ is \emph{cylindrical} when $\partial \lbotimes (-)$ preserves left maps.
\end{definition}

\begin{definition}
  \label{def:biased-mapping-cylinder}
  Given $f \co A \to B$ in a finitely cocomplete category with a functorial cylinder and $k \in \braces{0,1}$,
  we write $\MCyl{k}{f}$ for its \emph{$k$-sided mapping cylinder}, defined as the pushout
  \[
    \begin{tikzcd}
      A \ar{d}[left]{\Gd_k \otimes A} \ar{r}{f} & B \ar[dashed]{d}{\inr} \\
      \Cyl{A} \ar[dashed]{r}[below]{\inl} & \pushout \MCyl{k}{f}
    \end{tikzcd}
  \]
  The \emph{$k$-sided mapping cylinder factorization} of $f$ is the factorization
  \[
    \begin{tikzcd}[column sep=7em]
      A \ar{r}{\inl (\Gd_{1-k} \otimes A)} & \MCyl{k}{f} \ar{r}{\copair{f(\Ge \otimes A)}{\id}} & B \rlap{.}
    \end{tikzcd}
  \]
\end{definition}

\begin{definition}
  A \emph{cylindrical premodel structure} on a finitely complete and cocomplete category $\CatE$ consists of a
  premodel structure and adjoint functorial cylinder on $\CatE$ such that
  \begin{itemize}
  \item the (cofibration, trivial fibration) and (trivial cofibration, fibration) weak factorization systems
    are cylindrical;
  \item $\Gd_k \lbotimes (-)$ sends cofibrations to trivial cofibrations for $k \in \{0,1\}$.
  \end{itemize}
\end{definition}

\begin{remark}
  The above conditions are transpose to equivalent dual conditions on the corresponding adjoint functorial path object.
  Like its constituent components, the notion of cylindrical premodel structure is thus self-dual: a cylindrical premodel structure on $\CatE$ is the same as a cylindrical premodel structure on $\op{\CatE}$.
\end{remark}

\begin{remark}[Stability under (co)slicing] \label{cylindrical-premodel-structure-slices}
Continuing \cref{premodel-structure-slices,functorial-cylinder-slices}, a cylindrical premodel structure on $\CatE$ descends to cylindrical premodel structures on slices and coslices of $\CatE$.
We may exploit this to simplify arguments by for example working in a slice.
\end{remark}

Fix once more a premodel category $\CatM$ with factorization systems $(\CC,\CF_t)$ and $(\CC_t,\CF)$.  We show
that \cref{pre:exchange} is reducible to \cref{pre:other-cancellation} when $\CatM$ is cylindrical by relating
trivial fibrations with dual strong deformation retracts.

\begin{definition}
  In a category with a functorial cylinder, we say $f \co Y \to X$ is a \emph{dual strong $k$-oriented
    deformation retract} for some $k \in \{0,1\}$ when we have a map $s \co X \to Y$ such that $fs = \id$ and
  a homotopy $h \co \Cyl{Y} \to Y$ such that $h(\Gd_k \otimes Y) = sf$, $h(\Gd_{1-k} \otimes Y) = \id$, and
  $fh$ is a constant homotopy.  Equivalently (if the category is finitely cocomplete), $f$ is a dual strong
  $k$-oriented deformation retract when we have a diagonal filler
  \[
    \begin{tikzcd}[row sep=large, column sep=6em]
      Y \ar{d}[left]{\inl(\Gd_{1-k} \otimes Y)} \ar{r}{=} & Y \ar{d}{f} \\
      \MCyl{k}{f} \ar[dashed]{ur} \ar{r}[below]{\copair{f(\Ge \otimes Y)}{\id}} & X \rlap{.}
    \end{tikzcd}
  \]
\end{definition}

The following is a standard construction (see, \eg, \cite[Lemma I.5.1]{quillen67}).

\begin{lemma}
  \label{to-deformation-retract-objects}
  Let $(\CL,\CR)$ be a cylindrical weak factorization system on a finitely cocomplete category with a
  functorial cylinder.  Then any $\CR$-map between $\CL$-objects is a dual strong $k$-oriented deformation
  retract for any $k \in \{0,1\}$.
\end{lemma}
\begin{proof}
  Let $f \co Y \to X$ be an $\CR$-map between $\CL$-objects. We solve two lifting problems in turn:
  \begin{mathpar}
    \begin{tikzcd}[column sep=large, row sep=large]
      0 \ar{d}[left]{\CL \ni { }} \ar{r} & Y \ar{d}{f} \\
      X \ar[dashed]{ur}{s} \ar{r}[below]{=} & X
    \end{tikzcd}
    \and
    \begin{tikzcd}[column sep=large,row sep=large]
      Y \sqcup Y \ar{d}[left]{\CL \ni \partial \otimes Y} \ar{r}{\copair{sf}{\id}} & Y \ar{d}{f} \\
      \Cyl{Y} \ar[dashed]{ur}{h} \ar{r}[below]{f(\Ge \otimes Y)} & X \rlap{.}
    \end{tikzcd}
  \end{mathpar}
  The maps $s$ and $h$ exhibit $f$ as a dual strong $0$-oriented deformation retract; we may similarly
  construct a $1$-oriented equivalent.
\end{proof}

\begin{corollary}
  \label{to-deformation-retract}
  Let $(\CL,\CR)$ be a cylindrical weak factorization system on a category with a functorial cylinder. Then in
  any diagram of the form
  \[
    \begin{tikzcd}[row sep=small]
      & A \ar{dl}[above left]{m \in \CL} \ar{dr}{n \in \CL} \\
      Y \ar{rr}[below]{f \in \CR} & & X \rlap{,}
    \end{tikzcd}
  \]
  the horizontal map is a dual strong $k$-oriented deformation retract for any $k \in \{0,1\}$.
\end{corollary}
\begin{proof}
  By \cref{to-deformation-retract-objects}, applied in the coslice under $A$ via
  \cref{cylindrical-premodel-structure-slices}.
\end{proof}

\begin{lemma}
  \label{deformation-retract-to-trivial}
  If $\CatM$ is cylindrical, then any fibration $f \co Y \fib X$ that is a dual strong $k$-oriented
  deformation retract for some $k \in \{0,1\}$ is a trivial fibration.
\end{lemma}
\begin{proof}
  Let $s \co X \to Y$ and $h \co \Cyl{Y} \to Y$ be as in the definition of dual strong $k$-oriented
  deformation retract. Then the diagram
  \[
    \begin{tikzcd}[row sep=large, column sep=6em]
      Y \ar[fib]{d}[left]{f} \ar{r}{h^\dagger} & \Cocyl{Y} \ar[tfib]{d}[left]{\Gd_k \lboslash f} \ar{r}{\Gd_{1-k} \oslash Y} & Y \ar[fib]{d}{f} \\
      X \ar{r}[below]{\pair{\Ge \oslash X}{s}} & \Cocyl{X} \times_X Y \ar{r}[below]{(\Gd_{1-k} \oslash X)\projl} & X
    \end{tikzcd}
  \]
  exhibits $f$ as a retract of a trivial fibration.
\end{proof}

\begin{lemma}
  \label{cylindrical-weak-equivalence-factorization}
  Suppose $\CatM$ is cylindrical. If trivial fibrations have right cancellation among fibrations, then any
  (trivial cofibration, fibration) factorization of a weak equivalence is a (trivial cofibration, trivial
  fibration) factorization.

  Dually, if trivial cofibrations have left cancellation among cofibrations, then any (cofibration, trivial
  fibration) factorization of a weak equivalence is a (trivial cofibration, trivial fibration) factorization.
\end{lemma}
\begin{proof}
  Suppose we have a weak equivalence $X \to Y$ factoring as a trivial cofibration followed by a fibration,
  thus a diagram of the following form:
  \[
    \begin{tikzcd}[row sep=tiny]
      & U \ar[tfib]{dr} \\
      X \ar[tcof]{ur} \ar[tcof]{dr} & & Y \rlap{.} \\
      & V \ar[fib]{ur}
    \end{tikzcd}
  \]
  We first take a pullback and factorize the induced gap map as a trivial cofibration followed by a fibration.
  \[
    \begin{tikzcd}
      & & & U \ar[tfib]{dr} \\
      X \ar[tcof,bend left=20]{urrr} \ar[tcof,bend right=20]{drrr} \ar[tcof]{r} & Z \ar[fib]{r} & P \pullback<45>[7ex] \ar[dashed,fib]{ur} \ar[dashed,tfib=below]{dr} & & Y \\
      & & & V \ar[fib]{ur}
    \end{tikzcd}
  \]
  By \cref{to-deformation-retract}, the composites $Z \fib U$ and $Z \fib V$ are dual strong deformation
  retracts, thus trivial fibrations by \cref{deformation-retract-to-trivial}. Then the composite $Z \fib Y$ is
  a trivial fibration by composition, so $V \fib Y$ is a trivial fibration by right cancellation.
\end{proof}

\begin{theorem}
  \label{cylindrical-premodel-to-model}
  Suppose $\CatM$ is a cylindrical premodel structure. Then $\candweq{\CC}{\CF}$ satisfies 2-out-of-3 exactly
  if the following hold:
  \begin{conditions}[start=4]
  \item[(\labelcref{pre:other-cancellation})] trivial cofibrations have left cancellation among cofibrations and
    trivial fibrations have right cancellation among fibrations;
  \item[(\labelcref{pre:exchange})] any composite of a trivial fibration followed by a trivial cofibration is a weak equivalence.
  \end{conditions}
\end{theorem}
\begin{proof}
  \Cref{premodel-to-model} combined with \cref{cylindrical-weak-equivalence-factorization}.
\end{proof}

Finally, we prove for reference below that the cancellation properties opposite of
\cref{pre:other-cancellation} always hold in a cylindrical premodel structure, though we will not need this
fact.

\begin{lemma}
  \label{left-objects-mapping-cylinder-first-factor}
  Let $(\CL,\CR)$ be a cylindrical weak factorization system on a category with a functorial cylinder.  If $f$
  is a map between $\CL$-objects, then the first factor of its $k$-sided mapping cylinder
  factorization is an $\CL$-map.
\end{lemma}
\begin{proof}
  The first factor $A \to \MCyl{k}{f}$ in the factorization of $f \co A \to B$
  decomposes as the composite
  \[
    \begin{tikzcd}
      A \ar{r}{\inl} & A \sqcup B \ar{r}{\cong} & (A \sqcup A) \sqcup_A B \ar{r}{(\partial \otimes A) \sqcup_A B} &[5em] \Cyl{A} \sqcup_A B \rlap{.}
    \end{tikzcd}
  \]
  The first map is a cobase change of $0 \to B$, thus an $\CL$-map. The last map is a cobase change of
  $\partial \otimes A \cong \partial \lbotimes (0 \to A)$, thus also an $\CL$-map.
\end{proof}

\begin{lemma}
  \label{cylindrical-cancellation-objects}
  If $\CatM$ is cylindrical, then any cofibration between trivially cofibrant objects is a trivial
  cofibration. Dually, any fibration between trivially fibrant objects is a trivial fibration.
\end{lemma}
\begin{proof}
  Let $m \co A \cof B$ be a cofibration between trivially cofibrant objects. Consider the commutative square
  \[
    \begin{tikzcd}[column sep=huge]
      A \ar[cof]{d}[left]{m} \ar[tcof=below]{r}{\inl(\Gd_1 \otimes A)} & \MCyl{0}{m} \ar[tcof=below]{d}{\Gd_0 \lbotimes m} \\
      B \ar{r}[below]{\Gd_1 \otimes B} & \Cyl{B} \rlap{.}
    \end{tikzcd}
  \]
  The top horizontal map is a trivial cofibration by \cref{left-objects-mapping-cylinder-first-factor}, while
  the right vertical map is a trivial cofibration by cylindricality. The bottom map is split monic, so $m$ is
  a retract of a trivial cofibration and thus a trivial cofibration itself.
\end{proof}

\begin{corollary}[{{\thmcite[Lemma 4.5(iii)]{sattler17}}}]
  \label{cylindrical-cancellation}
  If $\CatM$ is cylindrical, then trivial cofibrations have right cancellation among cofibrations. Dually,
  trivial fibrations have left cancellation among fibrations.
\end{corollary}
\begin{proof}
  Given a diagram
   \[
     \begin{tikzcd}[row sep=small]
      & B \ar[cof]{dr}{f} \\
      A \ar[tcof]{ur} \ar[tcof=below]{rr} & & C \rlap{,}
    \end{tikzcd}
  \]
  we apply \cref{cylindrical-cancellation-objects} to $f$ in the coslice under $A$ via
  \cref{cylindrical-premodel-structure-slices}.
\end{proof}

\subsection{Model structures from the fibration extension property}
\label{sec:model-structure-fep}

We now narrow our attention to premodel structures satisfying properties common to cubical-type model
structures: first, that all objects are cofibrant, and second, that fibrations \emph{extend along trivial
  cofibrations}, the latter of which follows in particular from the existence of enough fibrant universes
classifying fibrations. Note that our conditions cease to be self-dual at this point; moreover, the result is
a criterion \emph{sufficient} but not \emph{necessary} to obtain a model structure.

\begin{lemma}
  \label{cancellation-iff-generation}
  Let $\CatM$ be a premodel category. Trivial fibrations have right cancellation in $\CatM$ if and only if the
  (cofibration, trivial fibration) factorization system is generated by cofibrations between cofibrant
  objects. Dually, trivial cofibrations have left cancellation in $\CatM$ if and only if the (trivial
  cofibration, fibration) factorization system is cogenerated by fibrations between fibrant objects.
\end{lemma}
\begin{proof}
  Suppose trivial fibrations have right cancellation in $\CatM$ and let $p \co Y \to X$ be a map lifting
  against cofibrations between cofibrant objects. We take a cofibrant replacement of $Y$, obtaining maps
  $0 \cof Y' \tfib Y$. By cancellation, it suffices to show the composite $p' \co Y' \to X$ is a trivial
  fibration. We appeal to the retract argument: $p'$ has the lifting property against the left part of its
  (cofibration, trivial fibration) factorization---this being a cofibration between cofibrant objects---so is
  a retract of the right part of its factorization. It is thus itself a trivial fibration.

  The converse is an elementary exercise in lifting. Suppose the (cofibration, trivial fibration)
  factorization system is generated by cofibrations between cofibrant objects, let $f \co Z \tfib Y$ and
  $g \co Y \to X$ be such that $gf$ is a trivial fibration. Given a cofibration $m \co A \cof B$ between
  cofibrant objects and a lifting problem
  \[
    \begin{tikzcd}
      A \ar[cof]{d}[left]{m} \ar{r} & Y \ar{d}{g} \\
      B \ar{r} & X \rlap{,}
    \end{tikzcd}
  \]
  we solve lifting problems first against $f$ and then against $gf$:
  \begin{mathpar}
    \begin{tikzcd}
      0 \ar[cof]{d}[left]{!} \ar{r}{!} & Z \ar[tfib=below]{d}{f} \\
      A \ar[dashed]{ur}[above left]{u} \ar{r} & Y
    \end{tikzcd}
    \and
    \begin{tikzcd}
      A \ar[cof]{d}[left]{m} \ar{r}{u} & Z \ar[weq]{d}{gf} \\
      B \ar[dashed]{ur}{v} \ar{r} & X \rlap{.}
    \end{tikzcd}
  \end{mathpar}
  The composite $fv$ is a lift for the original square.
\end{proof}

In particular, \cref{cancellation-iff-generation} tells us that trivial fibrations have right cancellation in
any premodel structure where all objects are cofibrant. If the premodel structure is additionally cylindrical,
then \cref{pre:exchange} is also always satisfied:

\begin{lemma}
  \label{all-cofibrant-to-exchange}
  Let $\CatM$ be cylindrical and suppose that all objects are cofibrant. Then any composite of a trivial
  fibration followed by a trivial cofibration is a weak equivalence.
\end{lemma}
\begin{proof}
  Suppose we have $p \co B \tfib A$ and $m \co A \tcof X$. We take their composite's (trivial cofibration,
  fibration) factorization:
  \[
    \begin{tikzcd}
      B \ar[tfib]{d}[left]{p} \ar[dashed,tcof=below]{r}{n} & Y \ar[dashed,fib]{d}{q} \\
      A \ar[tcof]{r}[below]{m} & X \rlap{.}
    \end{tikzcd}
  \]
  We intend to show $q$ is a trivial fibration. By \cref{to-deformation-retract} and the assumption that all
  objects are cofibrant, $p$ has the structure of a dual strong 0-oriented deformation retract. Thus we
  have a diagonal lift
  \[
    \begin{tikzcd}[row sep=large,column sep=6em]
      B \ar{d}[left]{\inl(\Gd_1 \otimes B)} \ar{r}{=} & B \ar[tfib=below]{d}{p} \\
      \MCyl{0}{p} \ar[dashed]{ur}{h} \ar{r}[below]{\copair{p(\Ge \otimes B)}{\id}} & A \rlap{.}
    \end{tikzcd}
  \]
  Using that $q$ is a fibration, we show that $q$ is a dual strong deformation retract by solving a lifting
  problem of the form
  \[
    \begin{tikzcd}[column sep=huge]
      Y \ar[bend right=80]{ddd}[left]{\inl(\Gd_1 \otimes Y)} \ar{d}{\inr} \ar[bend left]{dr}{\id} \\
      \MCyl{0}{p} \sqcup_B Y \ar{d}[sloped,below]{\cong} \ar{r}{\copair{nh}{\id}} & Y \ar[fib]{dd}{q} \\
      A \sqcup_B \MCyl{1}{n} \ar[tcof=below]{d} \\[0.5em]
      \MCyl{0}{q}  \ar[dashed]{uur} \ar{r}[below]{\copair{q(\Ge \otimes Y)}{\id}} & X \rlap{.}
    \end{tikzcd}
  \]
  The map $A \sqcup_B \MCyl{1}{n} \tcof \MCyl{0}{q}$ is the following composite:
  \[
    \begin{tikzcd}[column sep=small]
      A \sqcup_B \MCyl{1}{n} \ar[tcof]{r}[below=0.8em]{m \sqcup_B \MCyl{1}{n}}
      &[2em] X \sqcup_B \MCyl{1}{n} \ar{r}{\cong}
      & X \sqcup_Y (\Cyl{B} \sqcup_B (Y \sqcup Y)) \ar[tcof]{r}[below=0.8em]{X \sqcup_Y (\partial \lbotimes n)}
      &[2em] \MCyl{0}{q}
    \end{tikzcd}
  \]
  The first map is a cobase change of the trivial cofibration $m$, while the final map is a cobase change of
  the trivial cofibration $\partial \lbotimes n$; thus the composite is indeed a trivial cofibration. The
  diagonal lift exhibits $q$ as a dual strong deformation retract, thus a trivial fibration by
  \cref{deformation-retract-to-trivial}.
\end{proof}

Thus, in a cylindrical premodel structure where all objects are cofibrant, the only non-trivial property
necessary to apply \cref{cylindrical-premodel-to-model} is left cancellation for trivial cofibrations among
cofibrations. This we can further reduce to the following condition.

\begin{definition}[FEP]
  We say a premodel category $\CatM$ has the \emph{fibration extension property} when for any fibration
  $f \co Y \fib X$ and trivial cofibration $m \co X \tcof X'$, there exists a fibration $f' \co Y' \fib X'$
  whose base change along $m$ is $f$:
  \[
    \begin{tikzcd}
      Y \pullback \ar[fib]{d}[left]{f} \ar[dashed]{r} & Y' \ar[fib,dashed]{d}[right]{f'} \\
      X \ar[tcof]{r}[below]{m} & X' \rlap{.}
    \end{tikzcd}
  \]
\end{definition}

\begin{lemma}
  \label{fep-to-tcof-left-cancellation}
  Suppose $\CatM$ is a premodel category with the fibration extension property. Then trivial cofibrations have
  left cancellation in $\CatM$.
\end{lemma}
\begin{proof}
  By \cref{cancellation-iff-generation}, it suffices to show the (trivial cofibration, fibration)
  factorization system is cogenerated by fibrations between fibrant objects. Suppose $g \co A \to B$ is a map
  with the left lifting property against all fibrations between fibrant objects. Let $f \co Y \fib X$ be an
  arbitrary fibration. Its codomain $X$ has a fibrant replacement $m \co X \tcof X^\lfib$; by the fibration
  extension property there is some $f' \co Y' \fib X^\lfib$ whose pullback along $m$ is $f$. By assumption $g$
  lifts against $f'$, and this lift induces a lift for $g$ against $f$ via the usual argument that right maps
  of a weak factorization system are closed under base change.
\end{proof}

\begin{theorem}
  \label{all-cofibrant-fep-premodel-to-model}
  Let $\CatM$ be a cylindrical premodel category in which
  \begin{conditions}[start=4]
  \item \label{pre:cofibrant} all objects are cofibrant;
  \item the fibration extension property is satisfied.
  \end{conditions}
  Then the premodel structure on $\CatM$ defines a model structure.
\end{theorem}
\begin{proof}
  By \cref{cylindrical-premodel-to-model}. \Cref{pre:exchange} is satisfied by
  \cref{all-cofibrant-to-exchange}. Trivial cofibrations have left cancellation by
  \cref{fep-to-tcof-left-cancellation}, while trivial fibrations have right cancellation by
  \cref{cancellation-iff-generation}.
\end{proof}

The fibration extension property can in particular be obtained from the existence of fibrant classifiers for
fibrations, \ie, fibrant universes of fibrations. We do not generally expect to have a single classifier for
\emph{all} fibrations, only those below a certain size. Thus we now consider a setup where a premodel category
sits inside a larger category containing classifiers for its fibrations.

\begin{lemma}
  \label{universes-to-fep}
  Let $\CatE$ be a category, and let $\CatM$ be a subcategory of $\CatE$ equipped with a premodel
  structure. Say that a map in $\CatE$ is a fibration if it has the right lifting property against all
  trivial cofibrations in $\CatM$. Suppose we have a class $\CU \subseteq \CatE^\to$ of fibrations between
  fibrant objects that classifies fibrations in $\CatM$, in following sense:
  \begin{enumerate}[label=(\alph*)]
  \item every fibration in $\CatM$ is a pullback of some fibration in $\CU$;
  \item if $p \co E \fib U$ is a map in $\CU$ and $y \co X \to U$ is a map with $X \in \CatM$,
    then there exists a map in $\CatM$ which is the pullback of $p$ along $y$:
    \[
      \begin{tikzcd}
        \bullet \pullback \ar[dashed,fib]{d}[left]{\CatM \ni} \ar[dashed]{r} & E \ar[fib]{d}{p} \\
        X \ar{r}[below]{y} & U \rlap{.}
      \end{tikzcd}
    \]
  \end{enumerate}
  Then $\CatM$ has the fibration extension property.
\end{lemma}
\begin{proof}
  Let a fibration $f \co Y \fib X$ in $\CatM$ and trivial cofibration $m \co X \tcof X'$ in $\CatM$ be
  given. Then $f$ is the pullback of some fibration between fibrant objects $p \co E \fib U$ in
  $\CatE$ along some map $y \co X \to U$. As $U$ is fibrant, $y$ extends along $m$ to some
  $y' \co X' \to U$. By assumption, we can choose a pullback $f' \co Y' \fib X'$ of $p$ along $y'$
  belonging to $\CatM$. By the pasting law for pullbacks, $f$ is the pullback of $f'$ along $m$.
\end{proof}

\begin{corollary}
  \label{all-cofibrant-universe-to-model}
  Let $\CatE$ be a category, and let $\CatM$ be a subcategory of $\CatE$ equipped with a premodel
  structure. Suppose that $\CatM$ is cylindrical and the following
  conditions are satisfied:
  \begin{conditions}[start=6]
  \item[(\labelcref{pre:cofibrant})] all objects of $\CatM$ are cofibrant;
  \item there is a class of fibrations between fibrant objects in $\CatE$ that classifies fibrations in
    $\CatM$ in the sense of \cref{universes-to-fep}.
  \end{conditions}
  Then the premodel structure on $\CatM$ defines a model structure.
\end{corollary}
\begin{proof}
  By \cref{all-cofibrant-fep-premodel-to-model,universes-to-fep}.
\end{proof}

\begin{remark}
  \label{fep-without-universe}
  In applications, one usually starts with a set (or category, when working with algebraic weak factorization systems) of \emph{generating} trivial cofibrations that defines the class of fibrations via lifting.
  We can then consider an ``extension'' $\CatE$ of $\CatM$ large enough to build a classifier for fibrations in $\CatM$ (for example, by passing from presheaves to ``large'' presheaves as in \cref{sec:disjunctive-model-structure}).
  Fibrancy of the classifier is shown by extending fibrations along generating trivial cofibrations.

  In such settings, there is also an alternative approach that directly moves from fibration extension along generating trivial cofibrations to general fibration extension.
  For a set of generating trivial cofibrations with representable codomain, this is described in~\cite[\S7]{sattler17}.
  It involves exhibiting trivial cofibrations as codomain retracts of cell complexes of the generators using the small object argument; fibration extension along such a cell complex is then obtained inductively.
  In the model structure we construct in \cref{sec:disjunctive-model-structure}, we instead have a \emph{category} of generating trivial cofibrations with representable codomain (\cref{def:uniform-fibration}).
  However, the same technique still applies, using an analysis of the algebraic small object argument~\cite{sattler16}.
\end{remark}

\section{Semilattice cubical sets}
\label{sec:disjunctive-cubical-sets}

\subsection{The semilattice cube category}

We now introduce this article's main character: the (join-)semilattice cube category $\Or$ generated by an
interval object, finite cartesian products, and a binary connection operator. Like other cartesian cube
categories, it is a (single-sorted) \emph{Lawvere theory} \cite{lawvere63}: a finite product category in which
every object is a finite power of some distinguished object.

\begin{definition}
  \label{def:semilattice-theories}
  The theory of \emph{(join-)semilattices} consists of an associative and commutative binary operator $\lor$
  for which all elements are idempotent, which we call the \emph{join}. This means the following laws:
  \begin{mathpar}
    (x \lor y) \lor z = x \lor (y \lor z), \and
    x \lor y = y \lor x, \and
    x \lor x = x.
  \end{mathpar}
  The theory of \emph{01-bounded (join-)semilattices} consists, in addition to the above, of two constants
  $0$, $1$ and the following laws:
  \begin{mathpar}
    0 \lor x = x, \and 1 \lor x = 1.
  \end{mathpar}
  The \emph{(join-)semilattice cube category} $\Or$ is the Lawvere theory of 01-bounded
  semilattices. Concretely, the objects of $\Or$ are of the form $T^n$ for $n \in \BN$, and the morphisms
  $T^m \to T^n$ are $n$-ary tuples of expressions over $0,1,\lor$ in $m$ variables modulo the equations
  above. We write $\TJoin$ for the Lawvere theory of semilattices.
\end{definition}

\begin{remark}
\label{semilattice-theory-nice}
As a bicategory, $\TJoin$ can be identified with the subcategory of the bicategory of onto (decidable) relations between finite sets.
Equivalently, these are jointly injective spans in finite sets whose second leg is surjective.
This can be strictified to a 1-category by replacing relations with Boolean-valued matrices.
\end{remark}

Recall that the \emph{category of algebras} $\Alg{\CatT} \defeq \FPP{\CatT}{\Set}$ of a Lawvere theory $\CatT$
is the category of finite-product-preserving functors from $\CatT$ to $\Set$, which supports an ``underlying
set'' functor $U \co \FPP{\CatT}{\Set} \to \Set$ given by evaluation at the distinguished object $T^1$. This
functor has a left adjoint $F \co \Set \to \Alg{\CatT}$ which produces the \emph{free $\CatT$-algebra} on a
set, and the covariant Yoneda embedding restricts to an embedding $\op{\CatT} \to \Alg{\CatT}$ sending $T^n$
to the free algebra on $n$ elements. We write $\SLat$ and $\BSLat$ for the categories of algebras of $\TJoin$
and $\Or$ respectively. Concretely, these are the categories of sets equipped with the operations described in
\cref{def:semilattice-theories} and operation-preserving morphisms between them.

It can also be useful to take an order-theoretic perspective on $\SLat$ and $\BSLat$, identifying them as
subcategories of the category $\Pos$ of posets and monotone maps. Recall that the operator $\lor$ induces a
poset structure on any semilattice, with $x \le y$ when $x \lor y = y$.

\begin{proposition}
  $\SLat$ is equivalent to the subcategory of $\Pos$ consisting of posets with finite non-empty joins
  (that is, least upper bounds) and monotone maps that preserve said joins. $\BSLat$ is equivalent to the
  further (non-full) subcategory of posets that also have a minimum and maximum element and monotone maps that
  also preserve them. \qed
\end{proposition}

\begin{remark}
  \label{simplex-semilattice}
  Any finite linear order is a semilattice, and it is 01-bounded if it is inhabited.
  Moreover, any monotone map between linear orders preserves joins.
  Thus the inclusion $\Simp \to \Pos$ factors through a fully faithful inclusion $\Simp \to \SLat$.
\end{remark}

In particular, the interval $[1] \in \Pos$ is a 01-bounded semilattice.

\begin{proposition}
  The interval is a \emph{dualizing object} for a duality between the categories of finite semilattices and
  finite 01-bounded semilattices, which is to say that we have the following categorical equivalence:
  \[
    \begin{tikzcd}[column sep=8em]
      \op{\SLat_\lfin} \ar[bend left=15]{r}{\Hom{\SLat}{-}{[1]}} \ar[phantom]{r}{\rotatebox{270}{$\simeq$}} & \BSLat_\lfin \ar[bend left=15]{l}{\Hom{\BSLat}{-}{[1]}} \rlap{.}
    \end{tikzcd}
  \]
\end{proposition}
\begin{proof}
  By a slight variation on the argument that $\op{\LBSLat_\lfin} \simeq \LBSLat_\lfin$ indicated in \cite[\S
  VI3.6, \S VI.4.6(b)]{johnstone82}.
\end{proof}

Given a semilattice $A$, the 01-bounded semilattice structure on $\Hom{\SLat}{A}{[1]}$ is defined pointwise
from that on $[1]$; likewise $\Hom{\BSLat}{B}{[1]}$ has a pointwise semilattice structure for any
$B \in \BSLat$. This extends the duality between the augmented simplex category and the category of finite
intervals (\ie, finite bounded linear orders and bound-preserving monotone maps) observed by Joyal
\cites[\S1.1]{joyal97}{wraith93}.

By way of this duality, we have in particular an embedding of $\Or$ in the category of finite semilattices,
induced by the embedding of its opposite in its category of models:
\[
  \begin{tikzcd}
    \Or \ar{r}{\yo} & \op{\BSLat_\lfin} \ar{r}{\simeq} & \SLat_\lfin \rlap{.}
  \end{tikzcd}
\]
Here we use that the free semilattice on a finite set of generators is a finite semilattice. Unpacking, this
embedding sends $T^n$ to $\Hom{\BSLat}{F(n)}{[1]} \cong \Hom{\Set}{n}{U[1]} \cong [1]^n$.

\begin{notation}
  Henceforth we regard $\Or$ as a subcategory of $\SLat$, in particular writing $[1]^n$ rather than $T^n$ for
  its objects.
\end{notation}

We can also describe the cubes in $\SLat$ as free semilattices on posets. Given a poset $A$, write $1 \join A$
for the poset obtained by adjoining a minimum element $\bot$ to $A$. For any set $S$, we have a monotone map
$\Gh_n \co 1 \join S \to [1]^S$ sending $\bot$ to $\bot$ and $i \in S$ to the element of $[1]^S$ with $1$ at
its $i$th component and $0$ elsewhere.

\begin{proposition}
  \label{cube-free-on-poset}
  For any $S \in \FinSet$, the map $\Gh_S$ exhibits $[1]^S$ as the free semilattice on the poset $1 \join
  S$. That is, for any $A \in \SLat$ and monotone map $f \co 1 \join S \to A$, there is a unique semilattice
  morphism $f^\dagger \co [1]^S \to A$ such that $f = f^\dagger\Gh_S$. \qed
\end{proposition}

\subsection{Cubical-type model structure on semilattice cubical sets}
\label{sec:disjunctive-model-structure}

We now define our model structure on $\PSh{\Or}$ using \cref{all-cofibrant-universe-to-model}. That our case
satisfies the \lcnamecref{all-cofibrant-universe-to-model}'s hypotheses is essentially an application of
existing work, namely \cite{cavallo19c} or \cite{awodey23}, so we do not give many proofs, only enough of an
outline to guide an unfamiliar reader through the appropriate references. We point to
\cites{gambino17}{sattler17}[\S8]{awodey21-forcing} for further details on constructing model structures of
this kind and to \cite{licata18} for the definition of the universe in particular.

\begin{assumption}
  For simplicity, we work with a single universe: we assume a strongly inaccessible cardinal $\Gk$ and define
  a model structure on the category $\PSh[\Gk]{\Or}$ of $\Gk$-small presheaves. Outside of this section, we
  suppress the subscript $\Gk$. As described in \cref{fep-without-universe}, it is possible to eliminate the
  use of universes at the cost of some complication; alternatively, one can assume that every fibration
  belongs to some universe to obtain a model structure on all of $\PSh{\Or}$.
\end{assumption}

\begin{notation}
  We write $\ival \defeq \yo{[1]} \in \PSh{\Or}$ for the representable 1-cube. We write $\Gd_k \co 1 \to [1]$
  for the \emph{endpoint inclusion} picking out $k \in \{0,1\}$ and write $\Ge$ for the unique
  \emph{degeneracy} map $[1] \to 1$.
\end{notation}

\subsubsection{Factorization systems} As analyzed by Gambino and Sattler \cite{gambino17}, a key feature of
cubical-type model structures is that their fibrations are characterized by a \emph{uniform lifting
  property}. This characterization is used to obtain the model structure's factorization systems
constructively and to define fibrant universes of fibrations. We avoid formally introducing algebraic weak
factorization systems \cite{grandis06,garner09} for the sake of concision, but these form the
conceptual backbone of Gambino and Sattler's results.

\begin{definition}[Uniform lifting]
  Let $u \co \CatI \to \CatE^\to$ be a functor. A \emph{right $u$-map} is a map $f \co Y \to X$ in $\CatE$
  equipped with
  \begin{itemize}
  \item for each $i \in \CatI$ and filling problem
    \[
      \begin{tikzcd}
        A_i \ar{d}[left]{u_i} \ar{r}{h} & Y \ar{d}{f} \\
        B_i \ar{r}[below]{k} & X \rlap{,}
      \end{tikzcd}
    \]
    a diagonal filler $\Gf(i,h,k) \co B_i \to Y$;
  \item such that for each $\Ga \co j \to i$ and diagram
    \[
      \begin{tikzcd}
        A_j \ar[phantom]{dr}{u_\Ga} \ar{d}[left]{u_j} \ar{r}{a} & A_i \ar{d}{u_i} \ar{r}{h} & Y \ar{d}{f} \\
        B_j \ar{r}[below]{b} & B_i \ar{r}[below]{k} & X \rlap{,}
      \end{tikzcd}
    \]
    we have $\Gf(i,h,k)b = \Gf(j,ha, kb)$.
  \end{itemize}
  When $u$ is a subcategory inclusion, we may instead say that $f$ is a \emph{right $\CatI$-map}.
\end{definition}

\begin{notation}
  Given a category $\CatE$, write $\CartArr{\CatE} \subseteq \CatE^\to$ for the category of arrows in $\CatE$
  and cartesian squares between them.
\end{notation}

Write $\CM$ for the full subcategory of $\CartArr{\PSh[\Gk]{\Or}}$ consisting of monomorphisms.

\begin{definition}
  We say a map in $\CartArr{\PSh[\Gk]{\Or}}$ is a \emph{uniform trivial fibration} when it is a right
  $\CM$-map.
\end{definition}

\begin{remark}
  If working constructively, one must replace $\CM$ with the full subcategory $\CM_\ldec$ of levelwise
  decidable monomorphisms, \ie, those $m \co A \mono B$ such that $m_I$ is isomorphic to a coproduct inclusion
  for all $I \in \Or$. This restriction is used (see \eg, Orton and Pitts \cite[Theorem 8.4]{orton18}) in the
  proof of the \emph{realignment} property, which is important to the construction of fibrant universes.
\end{remark}

The following proposition lets us characterize the trivial fibrations (and later, the fibrations) as the maps
with uniform right lifting against a \emph{small} category.

\begin{proposition}[{{\thmcite[Proposition 5.16]{gambino17}}}]
  Let $\CatC$ be a small category and $\CatI$ be a full subcategory of $\CartArr{\PSh{\CatC}}$ closed under
  base change to representables, \ie, such that $\subst{x}f \in \CatI$ for any $f \co Y \to X$ in $\CatI$ and
  $x \co \yo{a} \to X$. Write $\CatI^{\yo}$ for the full subcategory of $\CatI$ consisting of maps with
  representable codomain. Then a map in $\PSh{\CatC}$ is a right $\CatI$-map if and only if it is a right
  $\CatI^{\yo}$-map.  \qed
\end{proposition}

\begin{proposition}[Uniform trivial fibrations]
  We have a weak factorization system $(\CM,\CF_t)$ where $\CF_t$ is the class of uniform trivial fibrations.
\end{proposition}
\begin{proof}
  By \cite[Theorem 9.1]{gambino17}, which goes through Garner's algebraic small object argument
  \cite{garner09}, we have a factorization system $(\CC,\CF_t)$ where $\CF_t$ is the class of uniform trivial
  fibrations. Here we need that the right $\CM$-maps coincide with the right $\CM^{\yo}$-maps and that
  $\CM^{\yo}$ is a small category. That the algebraic small object argument is constructive in this case is
  explained in \cite[Remark 9.4]{gambino17}; see also \cite[Appendix C]{henry20}.

  An alternative construction of the factorization using partial map classifiers is described in \cite[Remark
  9.5]{gambino17} and used by Awodey et al.~\cite{awodey21-forcing,awodey23}, while
  \citeauthor{swan18d}~\cite[\S6]{swan18d} describes a construction using W-types with reductions. The partial
  map classifier factorization factors any map as a mono followed by a trivial fibration. By the retract
  argument, any map in $\CC$ is then a retract of a mono and hence itself monic, so $\CC = \CM$.
\end{proof}

\begin{definition}
  \label{def:uniform-fibration}
  Define $u_\Gd \co \{0,1\} \times \CM^{\yo} \to \PSh[\Gk]{\Or}^\to$ by
  $u_\Gd(k,-) \defeq \Gd_k \lbtimes (-)$. A \emph{uniform fibration} is a right $u_\Gd$-map.
\end{definition}

\begin{proposition}[Uniform fibrations]
  There exists a weak factorization system $(\CC_t,\CF)$ such that $\CF$ is the class of uniform fibrations.
\end{proposition}
\begin{proof}
  By \cite[Theorem 7.5]{gambino17}, using the algebraic small object argument. Again, see \cite[Remark
  9.4]{gambino17} for discussion of constructivity.
\end{proof}

Though the algebraic/uniform description is important to constructively establish the \emph{existence} of
these weak factorization systems, we can also---still constructively---recognize that $\CF_t$ and $\CF$ are
classes of maps with lifting properties in the non-algebraic sense.

\begin{proposition}
  Let $f \co Y \to X$ in $\PSh[\Gk]{\Or}$. Then
  \begin{itemize}
  \item $f$ is a right $\CM$-map if and only if it has the right lifting property against all monomorphisms;
  \item $f$ is a right $u_\Gd$-map if and only if it has the right lifting property with respect to
    $\Gd_k \lbtimes m$ for all $k \in \{0,1\}$ and monomorphisms $m$.
  \end{itemize}
\end{proposition}
\begin{proof}
  By \cite[Theorem 9.9]{gambino17}.
\end{proof}

With the two factorization systems in hand, it is straightforward to verify the following.

\begin{proposition}
  $(\CC_t,\CF)$ and $(\CM,\CF_t)$, together with the adjoint functorial cylinder
  $\ival \times (-) \dashv (-)^\ival$, constitute a cylindrical premodel structure.
\qed
\end{proposition}

\subsubsection{Unbiased fibrations}

In order to apply \cref{all-cofibrant-universe-to-model}, we must check that we have a fibration between
fibrant objects in $\PSh{\Or}$ classifying fibrations in $\PSh[\Gk]{\Or}$. This follows from work on cubical
models of type theory, specifically the interpretation of universes. Our cube category falls within the ambit
of \cite{abcfhl}, which describes a universe $p_\lfib \co \widetilde U_\lfib \to U_\lfib$ with fibration
structures on $p_\lfib$ and $U_\lfib$ in type-theoretic terms; Awodey gives a construction of the same in
categorical language \cite[\S\S6--8]{awodey23}.

However, the fibrations used in these models are not \emph{a priori} the fibrations we defined in the previous
section: they are what Awodey \cite{awodey23} calls \emph{unbiased fibrations}, which lift not only against
(pushout products with) endpoint inclusions $\Gd_k \co 1 \to \ival$ but against generalized points on the
interval. To see that $\CMSOr$ is compatible with this model of type theory, we check here that biased (\ie,
ordinary) and unbiased fibrations coincide in the presence of a connection.

\begin{definition}
  Given $r \co B \to \ival$ and $f \co A \to B$, their \emph{unbiased mapping cylinder} is the following
  pushout:
  \[
    \begin{tikzcd}
      A \ar{d}[left]{\pair{rf}{\id_A}} \ar{r}[above]{f} & B \ar[dashed]{d}{c_r} \\
      \ival \times A \ar[dashed]{r}[below]{d_r} & \MCyl{r}{f} \pushout \rlap{.}
    \end{tikzcd}
  \]
  Note that $\MCyl{\Gd_k !_B}{f}$ is the ordinary $k$-sided mapping cylinder
  (\cref{def:biased-mapping-cylinder}).  We write $r \lbtimes_B m\co \MCyl{r}{m} \to \ival \times B$
  for the unique map fitting in the diagram
  \[
    \begin{tikzcd}
      A \ar{d}[left]{\pair{rf}{\id_A}} \ar{r}[above]{f} & B \ar{d}{c_r} \ar[bend left]{ddr}{\pair{r}{\id_B}} \\
      \ival \times A \ar[bend right]{drr}[below left]{\ival \times m} \ar{r}[below]{d_r} & \MCyl{r}{f} \pushout \ar[dashed]{dr}[left,yshift=-4pt]{r \lbtimes_B m} \\
      & & \ival \times B\rlap{.}
    \end{tikzcd}
  \]
  This is the pushout product in the slice over $B$ of $\pair{r}{\id_B} \co \id_B \to \Ge \times B$ and
  $m \co m \to \id_B$, hence the notation. Note that $(\Gd_k !_B) \lbtimes_B f$ is the ordinary pushout
  product $\Gd_k \lbtimes f$.
\end{definition}

\begin{definition}
  We say $f \co Y \to X$ is an \emph{unbiased fibration} when it has the right lifting property against
  $r \lbtimes_B m$ for all $r \co B \to \ival$ and $m \co A \cof B$.
\end{definition}

\begin{lemma}
  \label{unbiased-trivial-cof}
  $r \lbtimes_B m$ is a trivial cofibration for any $r \co B \to \ival$ and $m \co A \cof B$.
\end{lemma}
\begin{proof}
  Define $u(i,a) \defeq (i \lor r(m(a)), a) \co \ival \times A \to \ival \times A$.
  Take a pushout of $\Gd_0 \lbtimes m$:
  \[
    \begin{tikzcd}[column sep=large, row sep=large]
      \MCyl{0}{m} \ar[tcof]{d}[left]{\Gd_0 \lbtimes m} \ar{r}{u \sqcup \id} & \MCyl{r}{m} \ar[dashed,tcof=below]{d}{n} \\
      \ival \times B \ar[dashed]{r}[below]{b} & C \pushout \rlap{.}
    \end{tikzcd}
  \]
  Define a map $v \co \MCyl{1}{r \lbtimes_B m} \to C$ like so:
  \[
    \begin{tikzcd}
      \MCyl{r}{m} \ar{d}[left]{\Gd_1 \times \MCyl{r}{m}} \ar{r}{r \lbtimes_B m} & \ival \times B \ar{d}{c_1} \ar{r}{\Ge \times B} & B \ar{d}{\Gd_1 \times B} \\
      \ival \times \MCyl{r}{m} \ar{d}[sloped,below]{\cong} \ar{r}[below]{d_1} & \pushout \MCyl{1}{r \lbtimes_B m} \ar[dashed]{dr}[description]{v} & \ival \times B \ar{d}{b} \\
      \MCyl{r(\Ge \times B)}{\ival \times m} \ar{rr}[below]{\copair{nd_{r}({\lor} \times A)}{b}} & & C \rlap{.}
    \end{tikzcd}
  \]
  Take the pushout of $\Gd_1 \lbtimes (r \lbtimes_B m)$ by this map:
  \[
    \begin{tikzcd}[row sep=large]
      \MCyl{1}{r \lbtimes_B m} \ar[tcof]{d}[left]{\Gd_1 \lbtimes (r \lbtimes_B m)} \ar{r}{v} & C \ar[dashed,tcof=below]{d}{n'} \\
      \ival \times \ival \times B \ar[dashed]{r}[below]{b'} & \pushout D \rlap{.}
    \end{tikzcd}
  \]
  Then we can exhibit $r \lbtimes_B m$ as a retract of $n'n$:
  \[
    \begin{tikzcd}
      \MCyl{r}{m} \ar{dd}[left]{r \lbtimes_B m} \ar{dr}[description]{d_1(\Gd_0 \times \MCyl{r}{m})} \ar{rr}{\id} &[2em] &[-2em] \MCyl{r}{m} \ar[tcof=below]{d}{n} \ar{r}{\id} &[6.5em] \MCyl{r}{m} \ar{dd}{r \lbtimes_B m} \\
      & \MCyl{1}{r \lbtimes_B m} \ar{d} \ar{r}{v} & C \ar[tcof=below]{d}{n'} & \\
      \ival \times B \ar{r}[below]{\Gd_0 \times \ival \times B} & \ival \times \ival \times B \ar{r}[below]{b'} & \pushout D \ar{r}[below]{\copair{\Ge \times \ival \times B}{\copair{\id}{r \lbtimes_B m}}} & \ival \times B \rlap{.}
    \end{tikzcd}
  \]
  As a retract of a trivial cofibration, $r \lbtimes_B m$ is thus a trivial cofibration.
\end{proof}

\begin{corollary}
  \label{unbiased-fib}
  A map is a fibration in $\CMSOr$ if and only if it is an unbiased fibration.
\end{corollary}
\begin{proof}
  If $f \co Y \to X$ is an unbiased fibration, then lifting against any $\Gd_k \lbtimes m$ is obtained as
  lifting against $(\Gd_k{!}_B) \lbtimes_B m$. The converse is \cref{unbiased-trivial-cof}.
\end{proof}

\begin{remark}
  For the reader more comfortable with cubical type theories, we give the type-theoretic analogue to the proof
  of \cref{unbiased-fib}. The ABCHFL type theory equips types with a composition operator of the following
  form.
  \begin{mathpar}
    \inferrule
    {i \co \ival \vdash A\ \mathsf{type} \\
      \Gf\ \mathsf{cof} \\
      r,s \co \ival \\\\
      i \co \ival, \Gf \vdash M \co A \\
      M_0 \co \tmsubst{A}{r}{i} \\
      \Gf \vdash \tmsubst{M}{r}{i} = M_0 \co \tmsubst{A}{r}{i}}
    {\tmcom{i.A}{r}{s}{\tmface{\Gf}{i.M}}{M_0} \co \tmsubst{A}{s}{i}}
    \and
    \tmcom{i.A}{r}{r}{\tmface{\Gf}{i.M}}{M_0} = M_0 \co \tmsubst{A}{r}{i}
    \and
    \Gf \vdash \tmcom{i.A}{r}{s}{\tmface{\Gf}{i.M}}{M_0} = \tmsubst{M}{s}{i} \co \tmsubst{A}{s}{i}
  \end{mathpar}
  In the presence of a connection, we can derive a term satisfying the equations required of
  $\tmcom{i.A}{r}{s}{\tmface{\Gf}{i.M}}{M_0}$ using only composition $\Ge \to s$ where $\Ge \in \{0,1\}$,
  namely the term $Q$ below.
  \begin{align*}
    P(k) &\defeq \tmcom{j.\tmsubst{A}{r \lor j}{i}}{0}{k}{%
      \begin{array}{rcl}
        \Gf &\mapsto& j.\tmsubst{M}{r \lor j}{i}
      \end{array}
    }{M_0} \\
    Q &\defeq \tmcom{j.\tmsubst{A}{s \lor j}{i}}{1}{0}{%
      \begin{array}{rcl}
        \Gf &\mapsto& j.\tmsubst{M}{s \lor j}{i} \\
        r \equiv s &\mapsto& j.P(j)
      \end{array}
    }{P(1)}
  \end{align*}
\end{remark}

\begin{remark}
  \label{equivariant-lifting-from-connection}
  We can also use $\lor$ to show that any fibration is an \emph{equivariant} fibration in the sense of the
  ACCRS model structure \cite{accrs}. For simplicity, let us restrict attention to lifting along
  $\Gd^n_1 \co 1 \to \ival^n$, which is the simplest case; we leave it as an exercise to formulate and derive
  unbiased equivariant lifting by combining the proof of \cref{unbiased-trivial-cof} with the following
  sketch. A more complete proof (for simplicial sets rather than semilattice cubical sets, but with the same
  argument) is in \cite[Proposition 6.1.7]{accrs}.

  Write $\GS \defeq \core{\Or}$ for the wide subcategory of isomorphisms of $\Or$. We have a functor
  $\Gd \co \GS \to \PSh{\Or}^\to$ sending $[1]^n$ to $\Gd_1^n \co 1 \to \ival^n$ and
  $\Gs \co [1]^n \cong [1]^n$ to $(\id,\Gs) \co \Gd_1^n \to \Gd_1^n$. Take $u_{\Gd\GS}$ to be the composite
  \[
    \begin{tikzcd}[column sep=huge]
      \GS \times \CM^{\yo} \ar{r}{\Gd \times \CM^{\yo}} & \PSh{\Or}^\to \times \CM^{\yo}
   \ar{r}{\lbtimes} & \PSh{\Or}^\to \rlap{.}
    \end{tikzcd}
  \]
  A \emph{uniform equivariant 1-fibration} is a right $u_{\Gd\GS}$-map.

  Suppose $f \co Y \to X$ is a uniform fibration and let $m \co A \mono B$ and a lifting problem
  $(y,x) \co \Gd_1^n \lbtimes m \to f$ be given. We have a map
  ${\uparrow_n} \co {[1]} \times [1]^n \to [1]^n$ sending
  $(t, i_1, \ldots, i_n) \mapsto (t \lor i_1, \ldots, t \lor i_n)$ which we use to form a lifting problem
  against $\Gd_1 \lbtimes (\ival^n \times m)$:
  \[
    \begin{tikzcd}[row sep=large]
      (\ival \times \ival^n \times A) \sqcup_{\ival \times A} (\ival \times B) \ar[tcof]{d}[left]{\Gd_1 \lbtimes (\ival^n \times m)} \ar{r}{({\uparrow^n} \times A) \sqcup (\Ge \times B)} &[7em] (\ival^n \times A) \sqcup_A B \ar{d}{\Gd_1^n \lbtimes m} \ar{r}{y} & Y \ar[fib]{d}{f} \\
      \ival \times \ival^n \times B \ar{r}[below]{{\uparrow^n} \times B} \ar[dashed]{urr}[description]{j} & \ival^n \times B \ar{r}[below]{x} & X
    \end{tikzcd}
  \]
  The composite
  \begin{tikzcd}
    \ival^n \times B \ar{r}{\Gd_0 \times \ival \times B} &[3em] \ival \times \ival^n \times B \ar{r}{j} & Y
  \end{tikzcd}
  is our desired lift, while the uniformity conditions follow from those on $f$ and the fact that
  ${\uparrow_n} \circ (\ival \times \Gs) \cong \Gs \circ {\uparrow_n}$ for any $\Gs \co [1]^n \cong [1]^n$.
\end{remark}

\subsubsection{Universe}

To define a universe classifying fibrations, we use a theorem of \citeauthor*{licata18} \cite{licata18}. The
cardinal $\Gk$ provides a Grothendieck universe in $\Set$, from which Hofmann and Streicher's construction
produces a universe $p_U \co \widetilde U \to U$ in $\PSh{\Or}$ classifying $\Gk$-small maps
\cite{hofmann97,streicher05,awodey22}. Our classifier for $\Gk$-small fibrations shall be a subuniverse of
$p_U$. The key property of $\PSh{\Or}$ is that the cocylinder $(-)^\ival$ has a right adjoint, \ie, that
$\ival$ is \emph{internally tiny}: we have $(-)^\ival \cong \subst{((-) \times [1])}$ and therefore
$(-)^\ival \dashv \sqrt[\ival]{-} \defeq \ran{((-) \times [1])}$. This property is common to cube categories
but fails for example in simplicial sets. We refer to Swan \cite{swan22} for a deeper analysis.

Given a $\Gk$-small map $f \co Y \to X$ with characteristic map $A \co X \to U$, we define a family
$X^\ival \to U$ whose sections correspond to fibration structures on $A$. To do so, it is convenient to work
in the internal extensional type theory of the universe $p_U$ in the style of Orton and Pitts
\cite{orton18}.\footnote{We refer to \cite{awodey21-forcing} for a detailed translation between external and
  internal constructions in presheaf categories and to \cite[\S6]{awodey23} for a fully externalized
  argument.} Writing $\top \co 1 \to \GO$ for the subobject classifier in $\PSh{\Or}$, the maps
${!}_\GO \co \GO \to 1$ and $\top$ are both classified by $p_U$,\footnote{If working predicatively, one should
  replace $\GO$ with the classifier for levelwise decidable subobjects.} so appear as a closed type
$\cdot \vdash \GO : U$ and type family $\Gf \co \GO \vdash {[\Gf]} : U$ respectively. The interval likewise
appears as a closed type $\cdot \vdash \ival : U$ with inhabitants $\cdot \vdash 0,1 : \ival$.

\begin{definition}
  Given a type $A \co U$, define its type of trivial fibration structures $\TFibStr A \co U$ as follows:
  \begin{flalign*}
    \qquad \TFibStr A &\defeq \Pi \Gf \co \GO.\ \Pi v \co {[\Gf]} \to A.\ \Sigma a \co A.\ \Pi \Ga \co {[\Gf]}.\ v(\Ga) = a. &
  \end{flalign*}
\end{definition}

\begin{definition}
  Given $k \in \{0,1\}$ and $A \co X \to U$, define the pullback exponential
  $(\Gd_k \lbto A) : (\Sigma p \co X^\ival.\ A(p(k))) \to U$ internally as follows:
  \begin{flalign*}
    \qquad (\Gd_k \lbto A)(p,a) &\defeq \Sigma q \co (\Pi i \co \ival.\ A(p(i))).\ q(k) = a. &
  \end{flalign*}
\end{definition}

\begin{definition}
  Given $A \co X \to U$, define $\FibStr_kA \co X^\ival \to U$ for $k \in \{0,1\}$ and then
  $\FibStr A \co X^\ival \to U$ as follows:
    \begin{flalign*}
      \qquad (\FibStr_kA)(p) &\defeq \Pi a \co A(p(k)).\ \TFibStr ((\Gd_k \lbto A)(p,a)) & \\
      \qquad (\FibStr A)(p) &\defeq (\FibStr_0A)(p) \times (\FibStr_1A)(p). &
    \end{flalign*}
\end{definition}

\begin{proposition}
  \label{uniform-iff-str}
  Let $f \co Y \to X$ be given with classifying map $A \co X \to U$. Then $f$ is a uniform fibration if and
  only if the type $\Pi p \co X^\ival.\ (\FibStr A)(p)$ is inhabited.
\end{proposition}
\begin{proof}
  See \cite[Corollary 8.7]{awodey21-forcing}.
\end{proof}

Using the right adjoint to $(-)^\ival$, we carve out the subuniverse of $p_U$ corresponding to families
$A \co X \to U$ for which $\Pi p \co X^\ival.\ (\FibStr A)(p)$ is inhabited. For this step we return to
working externally, as $\sqrt[\ival]{-}$ does not straightforwardly internalize; \textcite{licata18} use a
global sections modality to axiomatize $\sqrt[\ival]{-}$ internally, while \textcite{riley24} has recently
proposed a type theory which directly represents $\sqrt[\ival]{-}$ as a modality. The following definition and
proposition constitute Theorem 5.2 of \cite{licata18}.

\begin{definition}
  Define $p_\lfib \co \widetilde U_\lfib \to U_\lfib$ by pullback as follows:
  \[
    \begin{tikzcd}
      \widetilde U_\lfib \pullback \ar[dashed]{d}[left]{p_\lfib} \ar[dashed]{r}{\widetilde\projr} & \widetilde U \ar{d}{p_U} \\
      U_\lfib \pullback \ar[dashed]{d}[left]{\projl} \ar[dashed]{r}{\projr} & U \ar{d}[right]{(\FibStr \id_U)^\dagger} \\
      \sqrt[\ival]{\widetilde U} \ar{r}[below]{\sqrt[\ival]{p_U}} & \sqrt[\ival]{U} \rlap{.}
    \end{tikzcd}
  \]
\end{definition}

\begin{proposition}[{{\thmcite[Theorem 5.2]{licata18}}}]
  If $f \co Y \to X$ is the pullback of $p_U$ along some $A \co X \to U$, then $f$ is a uniform fibration if
  and only if $A$ factors through $\projr \co U_\lfib \to U$.
\qed
\end{proposition}

\begin{corollary}
  \label{uproj-fibration}
  The map $p_\lfib$ is a uniform fibration.
\end{corollary}
\begin{proof}
  $p_\lfib$ is the pullback of $p_U$ along $\projr$, which of course factors through itself.
\end{proof}

Finally, we need a fibrancy structure on the universe $U$ itself. This is the most technically involved
argument; we defer to prior work.

\begin{proposition}
  \label{u-fibrant}
  The object $U_\lfib$ is uniform fibrant.
\end{proposition}
\begin{proof}
  A fibrancy structure on $U_\lfib$ is described in type-theoretic language in \cite[\S2.12]{abcfhl}, while
  Awodey \cite[\S8]{awodey23} gives an external categorical construction.
\end{proof}

\begin{theorem}[Cubical-type model structure on semilattice cubical sets]
  \label{or-model-structure}
  There is a model structure on $\PSh[\Gk]{\Or}$ in which
  \begin{itemize}
  \item the cofibrations are the monomorphisms;
  \item the fibrations are those maps with the right lifting property against all pushout products
    $\Gd_k \lbtimes m$ of an endpoint inclusion with a monomorphism.
  \end{itemize}
  We write $\CMSOr$ for this model category.
\end{theorem}
\begin{proof}
  By \cref{all-cofibrant-universe-to-model} applied with $\PSh[\Gk]{\Or}$ inside $\PSh{\Or}$ and the
  factorization systems $(\CM,\CF_t)$ and $(\CC_t,\CF)$ defined in this section. Clearly all objects are
  cofibrant, and every fibration in $\PSh[\Gk]{\Or}$ is classified by
  $p_\lfib \co \widetilde U_\lfib \to U_\lfib$, which is a fibration (\cref{uproj-fibration}) between fibrant
  objects (\cref{u-fibrant}).
\end{proof}

Our question now is whether $\CMSOr$ presents $\inftyGpd$. More narrowly, we can ask whether the
following comparison adjunction evinces a Quillen equivalence between $\CMSOr$ and $\KanSimp$.

\begin{definition}[Triangulation]
  \label{def:triangulation}
  Define $\triang \co \Or \to \PSh{\Simp}$ to be the functor sending the $n$-cube $[1]^n$ to the $n$-fold
  product $(\simp{1})^n$ of the $1$-simplex, with the evident functorial action. The \emph{triangulation}
  functor $\Triang \co \PSh{\Or} \to \PSh{\Simp}$ is the left Kan extension of $\triang$:
  \[
    \begin{tikzcd}
      {[1]^n} \ar[|->]{r}{\triang} & (\simp{1})^n \\[-1.5em]
      \Or \ar{d}[left]{\yo} \ar{r} & \PSh{\Simp} \rlap{.} \\
      \PSh{\Or} \ar[dashed]{ur}[below right]{\Triang}
    \end{tikzcd}
  \]
  Triangulation has a right adjoint, the nerve functor $\nerve{\triang} \co \PSh{\Simp} \to \PSh{\Or}$ defined
  by $\nerve{\triang}X \defeq \Hom{\PSh{\Simp}}{\triang-}{X}$.
\end{definition}

\subsection{Idempotent completion}
\label{sec:idempotent-completion}

Although the triangulation adjunction $\Triang \dashv \nerve{\triang}$ is the most immediate means of
comparing $\CMSOr$ and $\KanSimp$, it is not the most convenient. Ideally, we would like to have a comparison
on the level of the base categories, some functor $i \co \Simp \to \Or$ or vice versa, in which case we would
obtain an adjoint \emph{triple} $\lan{i} \dashv \subst{i} \dashv \ran{i}$ on their presheaf categories. This
is too much to hope for, but we \emph{can} define an embedding from $\Simp$ into the \emph{idempotent
  completion} of $\Or$, following the strategy used by \citeauthor{sattler19}~\cite{sattler19} and
\citeauthor{streicher21}~\cite{streicher21} to relate $\Simp$ and $\Ded$. The category of presheaves on any
category $\CatC$ is equivalent to the category of presheaves on its idempotent completion $\IdComp{\CatC}$,
the closure of $\CatC$ under splitting of idempotents \cite{borceux86}. We shall exhibit an embedding
$\insimp \co \Simp \to \IdCompOr$; by composing the triple
$\lan{\insimp} \dashv \subst{\insimp\!} \dashv \ran{\insimp}$ with the adjoint equivalence
$\adjunction{\subst{\incube}}{\PSh{\IdCompOr}}{\PSh{\Or}}{\lan{\incube}}$, we obtain a triple relating
$\PSh{\Simp}$ and $\PSh{\Or}$.

We then observe that $\Triang \cong \subst{\insimp\!}\lan{\incube}$ (\cref{middle-adjoint-is-triangulation});
thus the upshot of this detour is that $\Triang$ is also a right adjoint. It will, however, be easier to study
the adjunction $\lan{\insimp} \dashv \subst{\insimp\!}$ than $\Triang \dashv \nerve{\triang}$, in particular
because both $\lan{\insimp}$ and $\subst{\insimp\!}$ are left Quillen adjoints
(\cref{L-left-quillen,T-left-quillen}). We will first show in \cref{sec:equivalence-kan-quillen} that
$\lan{\insimp} \dashv \subst{\insimp\!}$ is a Quillen equivalence, then deduce formally that
$\subst{\insimp\!} \dashv \ran{\insimp}$ and $\Triang \dashv \nerve{\triang}$ are also Quillen equivalences.

\begin{definition}
  An \emph{idempotent} in a category $\CatC$ is a morphism $f \co A \to A$ such that $ff = f$.
  A \emph{splitting} for an idempotent is a section-retraction pair $(s,r)$ such that $f = sr$.
\end{definition}

The splitting of an idempotent is unique up to isomorphism if it exists: $s$ is the equalizer of the pair
$f,\id \co A \to A$, while $r$ is the coequalizer of the same.
We say that $\CatC$ is \emph{idempotent complete} if every idempotent splits.

\begin{definition}
  An \emph{idempotent completion} of a category $\CatC$ is a fully faithful functor
  $i \co \CatC \to \IdComp{\CatC}$ such that $\IdComp{\CatC}$ is idempotent complete
  and every object in $\IdComp{\CatC}$ is a retract of $iA$ for some $A \in \CatC$.
\end{definition}

Equivalently, an idempotent completion is a universal (in a bicategorical sense) fully faithful functor $\CatC \to \IdComp{\CatC}$ into an idempotent complete category.
We shall only need the following consequence of this characterization:

\begin{proposition}[{{essentially \thmcite[Theorem 1]{borceux86}}}]
  \label{idem-completion-equivalence}
  Given an idempotent completion $i \co \CatC \to \IdComp{\CatC}$, the induced substitution functor
  $i^* \co \PSh{\IdComp{\CatC}} \to \PSh{\CatC}$ is an equivalence of categories.
\qed
\end{proposition}

We can describe the idempotent completion of $\Or$ concretely as a full subcategory of $\SLat$.

\begin{definition}
  \label{def:idcompor}
  Write $\IdCompOr$ for the full subcategory of $\SLat$ consisting of finite inhabited distributive lattices.
  This subcategory contains all of $\Or$; we write $\incube \co \Or \to \IdCompOr$ for the inclusion.
\end{definition}

\begin{remark}
  Any finite inhabited lattice is bounded, with $\top$ and $\bot$ obtained as the join and meet of all elements
  respectively. Moreover, a finite lattice is distributive if and only if it is a Heyting algebra, \ie, supports
  an implication operator $\Rightarrow$. Note however that we do not require the morphisms of $\IdCompOr$ to
  preserve $\land$, $\bot$, $\top$, or $\Rightarrow$, only binary (\ie, non-empty finite) joins.
\end{remark}

We show that $\incube \co \Or \to \IdCompOr$ is an idempotent completion using the following observations of
\citeauthor{horn71}.

\begin{proposition}[{{\thmcite[Theorem 1.1]{horn71}}}]
  \label{semilattice-epi}
  A morphism in $\SLat$ is epic if and only it is surjective.
\qed
\end{proposition}

\begin{proposition}[{{\thmcite[Corollaries 2.9 and 5.4]{horn71}}}]
  \label{semilattice-injective-projective}
  Recall that an object in a category is \emph{injective} if maps into it extend along monomorphisms, and
  dually \emph{projective} if maps out of it lift along epimorphisms. A finite semilattice $A \in \SLat_\lfin$
  is
  \begin{itemize}
  \item injective if and only if $A$ is a distributive lattice;
  \item projective if and only if $1 \join A$ is a distributive lattice. \qed
  \end{itemize}
\end{proposition}

\begin{corollary}
  \label{idcompor-closed-under-retracts}
  $\IdCompOr$ is closed under retracts in $\SLat$.
\end{corollary}
\begin{proof}
  A retract of an inhabited finite semilattice is clearly inhabited and finite, and the class of injective
  objects is closed under retracts in any category.
\end{proof}

\begin{corollary}
\label{idcompor-closed-under-idempotent-splitting}
  $\IdCompOr$ is idempotent complete.
\end{corollary}
\begin{proof}
  Note that $\SLat$ is idempotent complete because it has limits.
  The claim follows from this using \cref{idcompor-closed-under-retracts}.
\end{proof}

\begin{lemma}
  \label{idcompor-retract-of-cube}
  Any $A \in \IdCompOr$ is a retract of $[1]^n$ for some $n \in \BN$.
\end{lemma}
\begin{proof}
  For any $A \in \IdCompOr$, we have a poset map $p \co 1 \join UA \to A$ sending $\bot$ to $\bot$ and
  $a \in UA$ to $a$. Per \cref{cube-free-on-poset}, this induces a surjective semilattice map
  $p^\dagger \co [1]^{UA} \to A$. This is epic by \cref{semilattice-epi}.
  As $A$ is distributive, so too is $1 \join A$, so $A$ is projective. Thus,
  the identity on $A$ factors through $p^\dagger$, exhibiting $A$ as a retract of $[1]^{UA}$.
\end{proof}

\begin{theorem}
  \label{or-idempotent-completion}
  $\incube \co \Or \to \IdCompOr$ is an idempotent completion.
\end{theorem}
\begin{proof}
  By \cref{idcompor-closed-under-idempotent-splitting,idcompor-retract-of-cube}.
\end{proof}

Recall from \cref{simplex-semilattice} that we have an embedding $\Simp \to \SLat$. The induced
lattice structure on a simplex is distributive, so this embedding factors through $\IdCompOr$.

\begin{notation}
  We write $\insimp \co \Simp \to \IdCompOr$ for the inclusion of the simplices among the finite
  inhabited distributive semilattices.
\end{notation}

We can now decompose the triangulation functor.

\begin{lemma}
  \label{middle-adjoint-is-triangulation}
  We have $\subst{\insimp\!}\lan{\incube} \cong \Triang \co \PSh{\Or} \to \PSh{\Simp}$.
\end{lemma}
\begin{proof}
  As both functors are left adjoints and thus cocontinuous, it suffices to exhibit a natural isomorphism
  between their restrictions to representables, i.e., show that $\subst{\insimp\!}{\yo}\incube \cong
  \triang$. Both $\subst{\insimp\!}{\yo}\incube$ and $\triang$ preserve products, and
  $\subst{\insimp\!} \yo \incube [1] \cong \simp{1} \cong \triang [1]$ by inspection.
\end{proof}

\subsection{Two Quillen adjunctions}
\label{sec:simp-cub-quillen-adjunction}

In light of the equivalence $\PSh{\IdCompOr} \simeq \PSh{\Or}$, it now suffices to compare $\KanSimp$ with the
induced model structure $\CMSIdComp$ on $\PSh{\IdCompOr}$, which again has monomorphisms for cofibrations and
fibrations generated by pushout products $\Gd_k \lbtimes m$. We begin by observing that both $\lan{\insimp}$
and $\subst{\insimp\!}$ are left Quillen adjoints.

\begin{lemma}
  \label{lan-insimp-preserves-monos}
  $\lan{\insimp}$ preserves monomorphisms.
\end{lemma}
\begin{proof}
  Write $\Simpaug$ for the augmented simplex category, the full subcategory of $\Pos$ consisting of the objects $[n]$ for $n \in \BN$ as well as $[-1] \defeq \emptyset$.
  Write $\IdCompOraug$ for the category of finite distributive semilattices, which similarly freely extends $\IdCompOr$ with an initial object (the empty semilattice).
  The inclusion $\insimp$ extends to an inclusion between the augmented categories:
  \[
    \begin{tikzcd}
      \Simp \ar{r}{\insimp} \ar{d}[left]{\iota} \pullback & \IdCompOr \ar{d}{\upsilon} \\
      \Simpaug \ar{r}[below]{\insimpaug} & \IdCompOraug \rlap{.}
    \end{tikzcd}
  \]
  The square is a pullback and the vertical maps are (discrete) Grothendieck opfibrations, so the square is exact in the sense that the canonical map $\subst{\upsilon} \lan{(\insimpaug)} \to \lan{\insimp} \subst{\iota}$ is invertible \cite[Proposition 5.2 and Corollary 3.1]{nlab:beck-chevalley_condition}.
  This is also straightforward to check directly: the functors are cocontinuous, so it suffices to check on representables, and $\subst{\iota}$ and $\subst{\upsilon}$ preserve all representables except for the initial representable, which they send to an initial object.
  Since $\iota$ is fully faithful, this gives $\lan{\insimp} \cong \lan{\insimp} \subst{\iota} \ran{\iota} \cong \subst{\upsilon} \lan{(\insimpaug)} \ran{\iota}$.
  Therefore, it suffices to prove that $\lan{(\insimpaug)}$ preserves monomorphisms.

  Just like in simplicial sets, the monomorphisms in augmented simplicial sets form the left class of a weak factorization system generated by boundary inclusions (of augmented simplices).
  As $\lan{(\insimpaug)}$ is a left adjoint, it therefore suffices to show that it sends boundary inclusions to monomorphisms.
  The boundary inclusion $\simpbd{n} \cof \simp{n}$ is the joint image of the non-identity face maps $\simp{k} \raising \simp{n}$.
  The joint image of a set of maps $(f_i : A_i \to B)_{i \in I}$ in any pretopos is computed as the coequalizer of
  \[
  \begin{tikzcd}
    \coprod_{i,j \in I} A_i \times_{B} A_j \ar[yshift=4pt]{r} \ar[yshift=-4pt]{r} & \coprod_{i \in I} A_i \rlap{.}
  \end{tikzcd}
  \]
  It therefore suffices to check that $\lan{(\insimpaug)}$ sends face maps to monomorphisms and preserves pullbacks of cospans whose legs are face maps.
  As face maps are monic, the latter condition implies the former.
  For the latter condition, as face maps go between representables and $\Simpaug$ has these pullbacks, it suffices to check that $\insimpaug$ preserves pullbacks of cospans whose legs are face maps.
  In fact $\insimpaug$ creates such pullbacks, as any subposet of a linear poset is again linear.
\end{proof}

The following statements can be phrased more generally at the level of cylinder objects in a model category.
They also have evident dual version in terms of path objects with fibrancy assumptions instead.

\begin{lemma}\label{homotopy-rel-weak-equivs}
In a cylindrical model category, let maps $f, g \co A \to X$ be related by a homotopy $h \co \Cyl{A} \to X$.
If $A$ is cofibrant, then $f$ is a weak equivalence exactly if $g$ is.
\end{lemma}

\begin{proof}
The top maps in the following diagram are trivial cofibrations because $A$ is cofibrant:
\[
  \begin{tikzcd}[column sep=huge]
    A \ar{dr}[below left]{f} \ar[tcof=below]{r}{\Gd_0 \otimes A} & \Cyl{A} \ar{d}{h} & \ar[tcof=below]{l}[above]{\Gd_1 \otimes A} A \ar{dl}{g} \rlap{.} \\
    & X
  \end{tikzcd}
\]
The claim follows using 2-out-of-3.
\end{proof}

In a cylindrical model category, a \emph{homotopy retract} is a pair of maps $s \co X \to Y$,
$r \co Y \to X$ equipped with a homotopy $h \co \Cyl{X} \to X$ from $rs$ to $\id_X$.

\begin{corollary}
  \label{homotopy-retracts-preserve-weakly-contractible}
  In a cylindrical model category, any cofibrant homotopy retract of a weakly contractible object is weakly
  contractible.
\end{corollary}
\begin{proof}
  Let a homotopy retract $s \co X \to Y$, $r \co Y \to X$, $h \co \Cyl{X} \to X$ from $rs$ to $\id_X$ be given
  with $X$ cofibrant and $Y$ weakly contractible.
  By \cref{homotopy-rel-weak-equivs}, $rs$ is a weak equivalence.
  Since $Y$ is weakly contractible, any endomorphism on $Y$ is a weak equivalence by 2-out-of-3.
  As the two binary sub-composites of the ternary composite $X \overset{s}\to Y \overset{r}\to X \overset{s}\to Y$ are weak equivalences, both $r$ and $s$ are weak equivalences by 2-out-of-6~\cite[Remark 2.1.3]{riehl14}.
\end{proof}

\begin{lemma}
  \label{left-quillen-from-kan-sset}
  Consider a model category $\CatM$ and a left adjoint $L \co \KanSimp \to \CatM$ that preserves cofibrations.
  Then $L$ is a left Quillen adjoint exactly if it sends representables to weakly contractible objects.
\end{lemma}

\begin{proof}
  For the non-trivial direction, assume that $L$ sends representables to weakly contractible objects.
  Given $n \ge 1$ and $I \subseteq [n]$, write $\simphorn{n}{I}$ for the union of the subobjects $d_i \co \simp{n-1} \mono \simp{n}$ over $i \in I$.
  We check by induction that $L$ sends $\simphorn{n}{I} \cof \simp{n}$ to a trivial cofibration for $n \in \BN$ and $\emptyset \subsetneq I \subsetneq [n]$.
  When $\deg{I} = 1$, $\simphorn{n}{I}$ is the representable $\simp{n-1}$, so the claim holds by assumption and 2-out-of-3.
  Otherwise, choose some $i \in I$.
  We have the following pushout square, which is preserved by $L$:
  \[
    \begin{tikzcd}
     \simphorn{n-1}{d^{-1}_i(I)} \ar[tail]{d} \ar{r} & \simphorn{n}{I\setminus\{i\}} \ar[tail]{d} \\
     \simp{n-1} \ar{r}[below]{d_i}  & \simphorn{n}{I} \pushout \rlap{.}
    \end{tikzcd}
  \]
  By induction hypothesis, $L$ sends the left vertical map to a trivial cofibration.
  As trivial cofibrations are closed under cobase change, $L$ then also sends the right vertical map to a trivial cofibration.
  By induction hypothesis, $L$ sends $\simphorn{n}{I\setminus\{i\}} \cof \simp{n}$ to a trivial cofibration.
  By 2-out-of-3, we conclude that $L$ sends $\simphorn{n}{I\setminus\{i\}} \cof \simp{n}$ to a trivial cofibration.
  For $I = [n]\setminus k$, we obtain that $L$ sends the horn inclusion $\simphorn{n}{k} \to \simp{n}$ to a trivial cofibration.
  This makes $L$ a left Quillen adjoint.
\end{proof}

  The combinatorics of the above proof have a conceptual explanation in terms of the pushout join in augmented simplicial sets, which produces boundary inclusions and horn inclusions starting from the maps $\emptyset \to 1$ and $\simp{-1} \to 1$.

\begin{corollary}[{{cf.\ \thmcite[Proposition 3.6]{sattler19}}}]
  \label{L-left-quillen}
  $\lan{\insimp}$ is a left Quillen adjoint $\KanSimp \to \CMSIdComp$.
\end{corollary}

\begin{proof}
  By \cref{lan-insimp-preserves-monos}, $\lan{\insimp}$ preserves monomorphisms.
  Using \cref{left-quillen-from-kan-sset}, it suffices to show that $\lan{\insimp}\simp{n} \cong \yo{[n]}$ is weakly contractible for $n \in \BN$.
  For this, we observe that $\yo{[n]}$ is a homotopy retract of 1 for each   $n \in \BN$ via the homotopy $(t,i) \mapsto (t \lor i) \co [1] \times [n] \to [n]$ and apply \cref{homotopy-retracts-preserve-weakly-contractible}.
\end{proof}

\begin{lemma}[{{cf.\ \thmcite[\S3.3]{sattler19}}}]
  \label{T-left-quillen}
  $\subst{\insimp\!}$ is a left Quillen adjoint $\CMSIdComp \to \KanSimp$.
\end{lemma}
\begin{proof}
  $\subst{\insimp\!}$ preserves monomorphisms because it is a right adjoint. As it is also a left adjoint, it
  also preserves pushout products, so
  $\subst{\insimp\!}(\Gd_k \lbtimes m) \cong \subst{\insimp\!}\Gd_k \lbtimes \subst{\insimp\!}m \cong d_{1-k}
  \lbtimes \subst{\insimp\!}m$ is a trivial cofibration for any $k \in \{0,1\}$ and
  $m \co A \rightarrowtail B$.
\end{proof}

We quickly see that $\lan{\insimp} \dashv \subst{\insimp\!}$ is a Quillen coreflection in the following
sense:

\begin{lemma}
  \label{unit-weq}
  The derived unit
  $X \overset{\Gh_X}\to \subst{\insimp\!}\lan{\insimp}X \to \subst{\insimp\!}((\lan{\insimp}X)^{\lfib})$ is valued
  in weak equivalences.
\end{lemma}
\begin{proof}
  It is equivalent to prove the unit $\Gh$ is valued in weak equivalences: any fibrant replacement map
  $\lan{\insimp}X \tcof (\lan{\insimp}X)^{\lfib}$ is a trivial cofibration, so is mapped to a trivial
  cofibration by the left Quillen adjoint $\subst{\insimp\!}$. But $\insimp$ is fully faithful, so the unit
  is valued in isomorphisms.
\end{proof}

\section{Relatively elegant Reedy categories}
\label{sec:relatively-elegant}

To show that the adjunction $\lan{\insimp} \dashv \subst{\insimp\!}$ defined in
\cref{sec:simp-cub-quillen-adjunction} is a Quillen equivalence, it remains to check that its counit is valued
in weak equivalences, that is, that $\Ge_X \co \lan{\insimp}\subst{\insimp\!}X \to X$ is a weak equivalence
for every fibrant $X \in \PSh{\IdCompOr}$. We noted earlier (\cref{saturated-elegant}) that for an elegant
Reedy category $\CatR$, we have a convenient set of objects---the automorphism quotients of
representables---that generate the whole of $\PSh{\CatR}$ upon saturation by monomorphisms. We will see later
on (\cref{saturated-nat-trans-weq}) that the class of $X \in \PSh{\IdCompOr}$ for which $\Ge_X$ is a weak
equivalence is saturated by monomorphisms, so if $\IdCompOr$ were an elegant Reedy category we would have a
line of attack. Unfortunately, this is not the case. Indeed, $\IdCompOr$ is not a Reedy category at all
(\cref{idcompor-not-reedy}).

We therefore require a generalization of elegant Reedy theory. We consider categories $\CatC$ equipped with a
fully faithful functor $i \co \CatC \to \CatR$ into a Reedy category $\CatR$ that has pushouts of lowering
spans and is \emph{elegant relative to $i$}: such that
$\nerve{i} \defeq \subst{i}\yo \co \CatR \to \PSh{\CatC}$ preserves lowering pushouts. In this case, the
objects of $\PSh{\CatC}$ are generated upon saturation by monos from the set of automorphism quotients of
objects in the image of $\nerve{i}$. When $i = \id$, we recover the original theorem for elegant Reedy
categories. In \cref{sec:algebra-elegant-core}, we shall see that $\IdCompOr$ embeds elegantly in the category
of inhabited finite semilattices.

In \cref{sec:cellular-presentation}, we review Reedy monomorphisms and the construction of cellular
presentations for maps between presheaves over a Reedy category. In \cref{sec:pre-elegance}, we narrow our
focus to what we call \emph{pre-elegant} Reedy categories, those having pushouts of lowering spans. The
Reedy monic presheaves are in this case characterized as those sending lowering pushouts to pullbacks.
This sets the stage for \cref{sec:relative-elegance}, where we define and study elegance relative to an
embedding $i \co \CatC \to \CatR$.

\subsection{Cellular presentations and Reedy monomorphisms}
\label{sec:cellular-presentation}

For the theory of cellular presentations of diagrams over Reedy categories, we follow
\citeauthor{riehl14b}~\cite{riehl14b,riehl17b}. Almost none of the content in this section is novel. For
simplicity, we restrict our attention throughout to presheaves, though much of the theory generalizes to
functors from a Reedy category into any category.

\subsubsection{Weighted colimits}

\citeauthor{riehl14b} observe that many arguments in Reedy category theory are naturally phrased in terms of
\emph{weighted (co)limits}. While more fundamental to enriched category theory, these can have a clarifying
role even in ordinary (\ie, $\Set$-enriched) category theory.

\begin{definition}
  Let $\CatE$ be a category. Let a functor $W \co \op{\CatC} \to \Set$ (the \emph{weight}) and a diagram
  $F \co \CatC \to \CatE$ be given. A \emph{weighted colimit} for this data is an object
  $\wcolim{\CatC}{W}{F} \in \CatE$, equipped with a natural transformation
  $W \to \Hom{\CatE}{F-}{\wcolim{\CatC}{W}{F}}$, such that for any $X \in \CatE$ the induced map
  \[
    \Hom{\CatE}{\textstyle\wcolim{\CatC}{W}{F}}{X} \to \Hom{\Func{\op{\CatC}}{\Set}}{W}{\Hom{\CatE}{F -}{X}}
  \]
  of sets is an isomorphism.
\end{definition}

Informally, the weight $W$ specifies how many ``copies'' of each object in the diagram $F$ to include in the
weighted colimit $\wcolim{\CatC}{W}{F}$.

\begin{example}
  \label{weighted-colim-and-ordinary-colim}
  The ordinary colimit of a diagram $F \co \CatC \to \CatE$ can be described as $\wcolim{\CatC}{1}{F}$, a
  colimit weighted by the terminal presheaf $1 \in \PSh{\CatC}$. Conversely, any weighted colimit
  $\wcolim{\CatC}{W}{F}$ admits a characterization as an ordinary colimit over the category of elements of
  $W$:
  \begin{align*}
    \textstyle\wcolim{\CatC}{W}{F} &\cong \colim \parens*{\catEl W \overset{\Gp}\longrightarrow \CatC \overset{F}\longrightarrow \CatE}.
  \end{align*}
  In particular, any cocomplete category has weighted colimits.
\end{example}

\begin{example}
  \label{weighted-colim-and-coend}
  Recall that a \emph{tensor} of a set $S \in \Set$ and object $X \in \CatE$ is an object $S \tens X$ such
  that morphisms $S \tens X \to Y$ correspond to objects $\Hom{\Set}{S}{\Hom{\CatE}{X}{Y}}$, \ie, families
  of morphisms $f_s \co X \to Y$ for $s \in S$. In ordinary category theory, this is simply the $S$-ary
  coproduct $\coprod_{s \in S} X$, so can be expressed as the weighted colimit $\wcolim{S}{1}{\GD X}$ of the
  constant diagram $\GD X \co S \to \CatE$. Alternatively, we can encode the tensor as the $S$-weighted
  colimit $\wcolim{\catTerm}{S}{X}$ of the diagram $X \co \catTerm \to \CatE$ over the terminal category. We
  can characterize any weighted colimit $\wcolim{\CatC}{W}{F}$ as a coend of tensors:
  \begin{align*}
    \wcolim{\CatC}{W}{F} &\cong \coend^{c \in \CatC} W_c \tens F^c \rlap{.}
  \end{align*}
\end{example}

We will always be working in cocomplete categories. For a given $\CatC$,
weighted colimits over $\CatC$ are then functorial in both the weight and the diagram, giving a bifunctor
$\wcolim{\CatC} \co \Func{\op{\CatC}}{\Set} \times \Func{\CatC}{\CatE} \to \CatE$.  This functoriality will be
an essential tool. In particular, we will often take a family of weighted colimits over a family of weights:

\begin{notation}
  Given a family of weights $W \co \CatD \times \op{\CatC} \to \Set$ and $F \co \CatC \to \CatE$, we write
  $\wcolim{\CatC}{W}{F} \co \CatD \to \CatE$ for the result of calculating the weighted colimit pointwise,
  that is $(\wcolim{\CatC}{W}{F})^d \defeq \wcolim{\CatC}{W^d}{F}$.
\end{notation}

\begin{remark}
  From the characterization in terms of ordinary colimits, it follows that weighted colimits in presheaf
  categories are computed pointwise. Thus for $W \co \op{\CatC} \to \Set$ and
  $F \co \CatC \times \op{\CatD} \to \Set$, we have
  $(\wcolim{\CatC}{W}{F})_d \cong \wcolim{\CatC}{W}{F_d}$, where on the left we regard $F$ as a functor
  $\CatC \to \PSh{\CatD}$.
\end{remark}

It follows quickly from the universal property defining weighted colimits that the bifunctor $\wcolim{\CatC}$
preserves colimits in both arguments. It is therefore determined by its behavior on representable weights,
which is simply characterized:

\begin{proposition}
  \label{weighted-colimit-representable}
  Naturally in $c \in \CatC$ and $X \co \CatC \to \CatE$, we have $\wcolim{\CatC}{\yo{c}}{X} \cong X^c$.
\qed
\end{proposition}

\begin{corollary}
  \label{weighted-colim-assoc}
  Naturally in $W \co \op{\CatD} \to \Set$, $V \co \CatD \times \op{\CatC} \to \Set$, and $F \co \CatC \to \CatE$,
  we have $\wcolim{\CatC}{(\wcolim{\CatD}{W}{V})}{F} \cong \wcolim{\CatD}{W}{(\wcolim{\CatC}{V}{F})}$.
\end{corollary}
\begin{proof}
  By cocontinuity, it suffices to check the case where $W$ is representable.
\end{proof}

\begin{notation}
  In this section, we use the notation $\yo{\CatC} \co \op{\CatC} \times \CatC \to \Set$ for the hom-bifunctor
  $\Hom{\CatC}{-}{-}$. Thus the representable functor for $c \in \CatC$, written $\yo{c}$ in our usual
  notation, may now be written as $\yo{\CatC}^c$, while we also have the co-representable
  $\yo{\CatC}_c \co \CatC \to \Set$. With our notation for parameterized weighted colimits,
  \cref{weighted-colimit-representable} then tells us that $\wcolim{\CatC}{\yo{\CatC}}{X} \cong X$ for any
  $X \in \PSh{\CatC}$. We have an analogous equation in the second argument:
  $\wcolim{\CatC}{X}{\yo{\CatC}} \cong X$.
\end{notation}

\subsubsection{Cellular presentations of presheaves}

A central theorem of Reedy theory is the existence of cellular presentations: when $\CatR$ is a Reedy
category, any $\CatR$-indexed diagram is a sequential colimit of maps that successively attach
cells of increasing degree. Likewise, any natural transformation between $\CatR$-indexed diagrams decomposes
as a transfinite composite of such maps. In the Riehl--Verity style, the intermediate objects and maps are
obtained by taking (Leibniz) weighted colimits of the input diagram. As
$X \cong \wcolim{\op{\CatR}}{\yo{\CatR}}{X}$ for any diagram $X$, one can exhibit a cellular presentation for $X$
by constructing a cellular presentation for $\yo{\CatR}$ and then applying the cocontinuous functor
$\wcolim{\op{\CatR}}{(-)}{X}$.

For the remainder of this section, we fix a Reedy category $\CatR$.

\begin{definition}
  For each $n \in \BN$, define $\inbdyd{}{\CatR} \co \sk{n}{\CatR} \mono \yo{\CatR}$ to be the subfunctor of arrows of
  degree less than $n$.
\end{definition}

\begin{definition}
  For any $n \in \BN$, write $\degcore{n}{\CatR}$ for the subcategory of $\CatR$ consisting of objects of
  degree $n$ and isomorphisms between them. We introduce the following notation for restrictions of
  $\yo{\CatR}$ where one argument or the other is required to have a given degree:
  \[
    \begin{tikzcd}[row sep=large]
      \degcore{n}{\CatR} \times \CatR \ar[dashed]{dr}[below left]{\yo_n\CatR} \ar[tail]{r} & \op{\CatR} \times \CatR \ar{d}[description]{\yo{\CatR}} & \op{\CatR} \times \degcore{n}{\CatR} \ar[tail]{l} \ar[dashed]{dl}{\yo^n\CatR} \\
      & \Set
    \end{tikzcd}
  \]
  We
  similarly introduce notation for the corresponding restrictions of the skeleton bifunctor
  $\sk{n}{\CatR} \co \op{\CatR} \times \CatR \to \Set$:
  \[
    \begin{tikzcd}[row sep=large]
      \degcore{n}{\CatR} \times \CatR \ar[dashed]{dr}[below left]{\bdyd{n}{\CatR}} \ar[tail]{r} & \op{\CatR} \times \CatR \ar{d}[description]{\sk{n}{\CatR}} & \op{\CatR} \times \degcore{n}{\CatR} \ar[tail]{l} \ar[dashed]{dl}{\bdyu{n}{\CatR}} \\
      & \Set
    \end{tikzcd}
  \]
  Finally, we write $\inbdyd{n}{\CatR} \co \bdyd{n}{\CatR} \mono \yo_n\CatR$ and
  $\inbdyu{n}{\CatR} \co \bdyu{n}{\CatR} \mono \yo^n\CatR$ for the restrictions of the inclusion
  $\inbdyd{}{\CatR} \co \sk{n}{\CatR} \mono \yo\CatR$.
\end{definition}

\begin{notation}
  For $r \in \CatR$ of degree $n$, we abbreviate $\bdyd{r}{\CatR} \defeq (\sk{n}{\CatR})_r \co \CatR \to \Set$
  and $\bdyu{r}{\CatR} \defeq (\sk{n}{\CatR})^r \co \op{\CatR} \to \Set$. Likewise, we write
  $\inbdyd{r}{\CatR} \defeq (\inbdyd{}{\CatR})_r \co \bdyd{r}{\CatR} \mono \yo{\CatR}_r$ and $\inbdyu{r}{\CatR} \defeq (\inbdyu{}{\CatR})^r \co \bdyu{r}{\CatR} \mono \yo{\CatR}^r$.
\end{notation}

\begin{definition}
  For any $f \co X \to Y$ in $\PSh{\CatR}$ and $n \in \BN$, the \emph{${<} n$-skeleton map} for $f$ is the Leibniz weighted colimit
  \[
    \lbwcolim{\op{\CatR}}{(\sk{n}{\CatR} \mono \yo\CatR)}{f}.
  \]
  We write $\sk{n}{f} \in \PSh{\CatR}$ for the domain of this map, which we call the \emph{${<}n$-skeleton} of
  $f$; its codomain is $Y$.
  For $Y \in \PSh{\CatR}$, we write $\sk{n}{Y}$ for the $n$-skeleton of the map $0 \mono Y$.
\end{definition}

Note that the ${<}0$-skeleton map is
$\lbwcolim{\op{\CatR}}{(0 \mono \yo\CatR)}{f} \cong \wcolim{\op{\CatR}}{\yo\CatR}{f} \cong f$.  For each
$m \le n \in \BN$, the inclusion $\sk{m}{\CatR} \mono \sk{n}{\CatR}$ induces a morphism
$\sk{m}{f} \to \sk{n}{f}$ by functoriality of weighted colimits, and the fact that $\yo{\CatR}$ is the union
of the subfunctors $\sk{n}{\CatR}$ implies that $Y \cong \colim_{n \in \BN} \sk{n}{f}$. Thus we have a natural
decomposition of $f$ as the transfinite composite $\sk{0}{f} \to \sk{1}{f} \to \sk{2}{f} \to \cdots$ where we
may compute $\sk{n}{f} \cong X \sqcup_{\sk{n}{X}} \sk{n}{Y}$. The chain of skeleta may be further decomposed in terms
of \emph{latching maps}:

\begin{definition}
  Given $f \co X \to Y$ in $\PSh{\CatR}$ and $r \in \CatR$, define the \emph{latching map
    $\latchmap{r}{f} \in \Set^\to$ for $f$ at $r$} by the Leibniz weighted colimit
  \[
    \latchmap{r}{f} \defeq \lbwcolim{\op{\CatR}}{\inbdyd{r}{\CatR}}{f} \rlap{.}
  \]
  The codomain of this map is $Y_r$; we write $\latchob{r}{f}$ for its domain and call this the \emph{latching
    object for $f$ at $r$}.
\end{definition}

We write $\latchmap{r}{Y}$ and $\latchob{r}{Y}$ for the latching map and object of $0 \mono Y$ at $r$. For
general $f \co X \to Y$, we can calculate that
$\latchob{r}{f} \cong X_r \sqcup_{\latchob{r}{X}} \latchob{r}{Y}$ and
$\latchmap{r}{f} \cong \copair{f_r}{\latchob{r}{f}}$. It is convenient to have notation for the collected
$\degcore{n}{\CatR}$-sets of latching maps at a given degree:

\begin{definition}
  Given $f \co X \to Y$ and $n \in \BN$, we define the \emph{$n$th latching map of $f$} by
  $\latchmap{n}{f} \defeq \lbwcolim{\op{\CatR}}{\inbdyd{n}{\CatR}}{f}$. We write
  $\latchob{n}{f} \in \PSh{\degcore{n}{\CatR}}$ for its domain and
  $\latchcod{n}{f} \in \PSh{\degcore{n}{\CatR}}$ for its codomain.
\end{definition}

These maps are assembled from the latching maps at the individual objects of degree $n$: we have
$(\latchmap{n}{f})_r \cong \latchmap{r}{f}$ for each $r \in \degcore{n}{\CatR}$.

We can now exhibit the maps between successive ${<}n$-skeleta as pushouts of Leibniz weighted colimits of
boundary inclusions and latching maps. The induced decomposition of a map $f$ into a sequential colimit of pushouts of basic maps is what we
mean by a \emph{cellular presentation} of $f$:

\begin{proposition}[{{\thmcite[Corollary 4.21]{riehl17b}}}]
  \label{cell-pushout-presheaf}
  For any $f \co X \to Y$ and $n \in \BN$, we have a pushout square of the following form:
  \[
    \begin{tikzcd}[column sep=huge]
      \bullet \ar{d} \ar{r}{\lbwcolim{\op{\degcore{n}{\CatR}}}{\inbdyu{n}{\CatR}}{\latchmap{n}{f}}} & \bullet \ar{d} \\
      \sk{n}f \ar{r} & \pushout \sk{n+1}f \rlap{.}
    \end{tikzcd}
  \]
  We refer to the maps $\lbwcolim{\op{\degcore{n}{\CatR}}}{\inbdyu{n}{\CatR}}{\latchmap{n}{f}}$ as \emph{cell
    maps}.
\end{proposition}
\begin{proof}
  By applying $\lbwcolim{\op{\CatR}}{(-)}{f}$ to a pushout square in $\op{\CatR} \times \CatR \to \Set$; see
  \cite[Theorem 4.15]{riehl17b}.
\end{proof}

\begin{corollary}
  \label{reedy-cell-complex}
  Every $f \co X \to Y$ in $\PSh{\CatR}$ has a cellular presentation by maps of the form
  $\lbwcolim{\op{\degcore{n}{\CatR}}}{\inbdyu{n}{\CatR}}{\latchmap{n}{f}}$.
\end{corollary}

For our purposes, namely working with properties saturated by monomorphisms, it is important to know when the
cell maps are monic.

\begin{definition}
  A map $f \co X \to Y$ in $\PSh{\CatR}$ is a \emph{Reedy monomorphism} when $\latchmap{r}{f}$ is
  monic in $\Set$ for all $r \in \CatR$.
\end{definition}

Here and in the following, we are specializing the theory of \emph{Reedy cofibrations} to the (mono, epi) weak
factorization system on $\Set$. To see when Reedy monomorphisms have monic cell maps, we use the
following lemma. Recall that a map is \emph{epi-projective} if it
has the left lifting property against all epimorphisms.

\begin{proposition}
  \label{weighted-colim-mono}
  Let $\CatC$ be a small category, $f \in \Func{\op{\CatC}}{\Set}^\to$, and $g \in
  \Func{\CatC}{\Set}^\to$.
  If $f$ is epi-projective and $g$ is monic, then $\lbwcolim{\CatC}{f}{g}$ is monic.
\end{proposition}
\begin{proof}
  By \cite[Lemma 3.13 and Corollary 3.17]{riehl17b} applied to the (mono, epi) weak factorization system on
  $\Set$.
\end{proof}

\begin{lemma}
  \label{boundary-projective-mono}
  If isos act freely on lowering maps in $\CatR$, then $\inbdyu{n}{\CatR}_r$ is epi-projective in $\degcore{n}{\CatR} \to \Set$.
\end{lemma}
\begin{proof}
A given morphism from $r$ to an object of degree $n$ is either a lowering map or has degree less than $n$.
This induces the following coproduct decomposition in $\degcore{n}{\CatR} \to \Set$:
\[
\begin{tikzcd}
  0
  \ar[r]
  \ar[d]
&
  \bdyu{n}{\CatR}_r
  \ar{d}{\inbdyu{n}{\CatR}_r}
\\
  \CatR^-(r, -)
  \ar[r]
&
  \yo^n\CatR_r
  \pushout
\rlap{.}
\end{tikzcd}
\]
Since epi-projective is the left class in a weak factorization system, it is stable under cobase change.
It thus suffices to show that $\CatR^-(r, -)$ is epi-projective.
Since isos act freely on $\CatR^-(r, -)$, it is the left Kan extension along some functor $A \to \degcore{n}{\CatR}$ of some $F \co A \to \Set$ with $A$ a set.
Recall that epimorphisms are characterized levelwise in $\Set$-valued functors.
By adjoint transposition, it thus suffices to show that $F$ is epi-projective.
Since $A$ is a set, this just means that $F$ is levelwise epi-projective.
And in $\Set$, every object is epi-projective.
\end{proof}

\begin{corollary}
  \label{reedy-mono-to-cell-mono}
  Suppose that isos act freely on lowering maps in $\CatR$. Given a Reedy monic $f \in \PSh{\CatR}^\to$, the map
  $\lbwcolim{\op{\degcore{n}{\CatR}}}{\inbdyu{n}{\CatR}}{\latchmap{n}{f}}$ is monic for all $n \in \BN$.
\end{corollary}
\begin{proof}
  We have
  $(\lbwcolim{\op{\degcore{n}{\CatR}}}{\inbdyu{n}{\CatR}}{\latchmap{n}{f}})_r =
  \lbwcolim{\op{\degcore{n}{\CatR}}}{\inbdyu{n}{\CatR}_r}{\latchmap{n}{f}}$ for every $r \in \CatR$. We know
  $\inbdyu{n}{\CatR}_r$ is epi-projective by \cref{boundary-projective-mono}, and $\latchmap{n}{f}$ is monic
  by assumption, so their Leibniz weighted colimit is monic by \cref{weighted-colim-mono}.
\end{proof}

\subsubsection{Eilenberg-Zilber decompositions}

The Reedy monomorphisms with initial domain can be characterized more simply: an object $X$ is Reedy monic exactly if every element of $X$ writes uniquely up to isomorphism as a degeneracy of a non-degenerate element of $X$.
We are not aware of a proof of this precise statement (\cref{cofibrant-iff-unique-non-degenerate}) in the literature,
though we would be surprised if it were unknown. We use Cisinski's term ``Eilenberg-Zilber decomposition''
\cite[Proposition 8.1.13]{cisinski06} for what Berger and Moerdijk call \emph{standard decompositions}.

\begin{definition}
  Let $X \in \PSh{\CatR}$. We say that $x \in X_r$ is \emph{non-degenerate} when every lowering map
  $e \co r \lowering s$ admitting an $x' \in X_s$ with $x'e = x$ is an isomorphism. An \emph{Eilenberg-Zilber
    (EZ) decomposition} of $x \in X_r$ is a pair $(e,x')$ where $x' \in X_s$ is non-degenerate,
  $e \co r \to s$ is a lowering map, and $x = x'e$. We regard two EZ decompositions $(e_0,x_0')$ and
  $(e_1,x_1')$ of $x$ as \emph{isomorphic} when there exists an isomorphism $\Gth \co s_0 \cong s_1$ in
  $\CatR$ such that $x_0'\Gth = x_1'$ and $e_0 = e_1\Gth$. We say $X$ has \emph{unique EZ decompositions} when
  any two EZ decompositions of any element of $X$ are isomorphic.
\end{definition}

\begin{remark}
  Every element of a presheaf admits at least one EZ decomposition: for any $x \in X_r$ there
  exists a minimal $n \in \BN$ such that $x$ factors though a lowering map to an object of degree $n$, and any
  such factorization is an EZ decomposition.
\end{remark}

\begin{proposition}[{{\thmcite[Observation 3.23]{riehl14b}}}]
  \label{latching-object-only-lowering-maps}
  Given $X \in \PSh{\CatR}$ and $r \in \CatR$, we have an isomorphism
  \[
    \begin{tikzcd}[row sep=tiny]
      \latchob{r}{X_-} \ar[dashed]{dd}[sloped]{\cong} \ar{dr}[pos=.2]{\latchmap{r}{X_-}} \\
      & X_r \rlap{,} \\
      \latchob{r}{X} \ar{ur}[below right,pos=.2]{\latchmap{r}{X}}
    \end{tikzcd}
  \]
  where $X_- \in \PSh{\CatR^-}$ is the restriction of $X$ along the Reedy category inclusion
  $\CatR^- \to \CatR$.
\qed
\end{proposition}

\begin{lemma}
  \label{cofibrant-iff-unique-non-degenerate}
  A presheaf $X \in \PSh{\CatR}$ is Reedy monic if and only if it has unique
  EZ decompositions.
\end{lemma}
\begin{proof}
  Suppose that $X$ is Reedy monic. We show that any two EZ decompositions of any $x \in X_r$ are
  isomorphic by induction on $\deg{r}$. Let two such factorizations $(e_0,x_0)$, $(e_1,x_1)$ be given. If
  either of $e_0$ or $e_1$ is an isomorphism, then the other must be as well, in which case the factorizations
  are trivially isomorphic; thus we can assume that each $e_k$ strictly decreases degree.  Then $(e_0,x_0)$
  and $(e_1,x_1)$ belong to $\latchob{r}{X_-}$; because $X$ is Reedy monic, they are moreover equal
  therein. By the concrete characterization of colimits in $\Set$, we have a finite sequence of lowering spans
  $s_i \overset{f_i}\leftarrow t_i \overset{f'_i} \rightarrow s_{i+1}$ for $0 \le i < n$, always with
  $\deg{s_i},\deg{t_i} < \deg{r}$, together with elements $y_i \co \yo{s_i} \to X$ for each $i \le n$, such
  that $y_0 = x_0$, $y_n = x_1$, and $y_if_i = y_{i+1}f'_i$:
  \[
    \begin{tikzcd}
      & & & \yo{r} \ar[lowering,to path={-| (\tikztotarget) [near start]\tikztonodes}, rounded corners]{ddlll}[above]{e_0} \ar[lowering]{dll} \ar[phantom]{d}{\vdots} \ar[lowering]{drr} \ar[lowering,to path={-| (\tikztotarget) [near start]\tikztonodes}, rounded corners]{ddrrr}{e_1} \\[-.2em]
      & \yo{t_0} \ar[lowering]{dl}[above left]{f_0} \ar[lowering]{dr}[above right]{f'_0} & & \ar[lowering]{dl}[above left]{f_1} \cdots \ar[lowering]{dr}[above right]{f'_{n-2}} & & \yo{t_{n-1}} \ar[lowering]{dl}[above left]{f_{n-1}} \ar[lowering]{dr}[above right]{f'_{n-1}} \\[-1.2em]
      \yo{s_0} \ar[to path={|- (\tikztotarget) [near end]\tikztonodes}, rounded corners]{drrr}[above]{x_0} & & \yo{s_1} \ar{dr}[below left]{y_1} & \vdots & \yo{s_{n-1}} \ar{dl}[below right]{y_{n-1}} & & \yo{s_n} \ar[to path={|- (\tikztotarget) [near end]\tikztonodes}, rounded corners]{dlll}[above]{x_1} \rlap{.} \\[-.2em]
      & & & X
    \end{tikzcd}
  \]
  By taking an EZ decomposition of each $y_i$ and absorbing the lowering map into $f'_i,f_{i+1} $,
  we can assume without loss of generality that each $y_i$ is non-degenerate. Then for each $i$, the equation
  $y_if_i = y_{i+1}f'_i$ makes $(y_i,f_i)$ and $(y_{i+1},f'_i)$ EZ decompositions of the
  same element of $X_{t_i}$. As $\deg{t_i} < \deg{r}$, it follows by induction hypothesis that they are
  isomorphic. Chaining these isomorphisms, we conclude that $(e_0,x_0)$ and $(e_1,x_1)$ are isomorphic.

  Now suppose conversely that $X$ has unique EZ decompositions. By \cref{latching-object-only-lowering-maps},
  it suffices to show the map $\latchob{r}{X_-} \to X_r$ is monic.  The elements of
  $\latchob{r}{X_-}$ are pairs $(e \co r \lowering s, x \in X_s)$ where $e$ is a strictly lowering map,
  quotiented by the relation $(fe,x) = (e,xf)$ for any $f \in \CatR^-$; the latching map sends $(e,x)$ to
  $xe \in X_r$. Let $(e_0,x_0),(e_1,x_1) \in L_rX_-$ be given such that $x_0e_0 = x_1e_1$. Without loss of
  generality, we may assume that these are EZ decompositions, in which case they are isomorphic and thus equal
  as elements of $\latchob{r}{X_-}$.
\end{proof}

\subsubsection{Saturation by monomorphisms}

Now we check that the class of Reedy monic presheaves is contained in the saturation by monos of the set
of automorphism quotients of representables, assuming isos act freely on lowering maps in $\CatR$.

\begin{lemma}
  \label{cell-is-groupoidal-colimit}
  For any $X \in \PSh{\degcore{n}{\CatR}}$, the presheaf $\wcolim{\op{\degcore{n}{\CatR}}}{\yo^n\CatR}{X}$ is a
  coproduct of automorphism quotients of representables.
\end{lemma}
\begin{proof}
  Write $\degcore{n}{\CatR}$ as a coproduct of groups $\degcore{n}{\CatR} \cong \coprod_{i} G_i$. Using
  the characterization of orbits as quotients by stabilizer groups, we may decompose $X$ as a coproduct of
  orbits $X \cong \coprod_{i,j} \autquo{\yo r_i}{H_{ij}}$, where $r_i \in \CatR$ is the
  point of $G_i$. By cocontinuity of $\wcolim{\op{\degcore{n}{\CatR}}}{\yo^n\CatR}{(-)}$, we then have
  \begin{align*}
    \wcolim{\op{\degcore{n}{\CatR}}}{\yo^n\CatR}{X}
    \cong \textstyle\coprod_{i,j} \autquo{(\wcolim{\op{\degcore{n}{\CatR}}}{\yo^n\CatR}{\yo{r_i}})}{H_{ij}}
    \cong  \textstyle\coprod_{i,j} \autquo{\yo{r_i}}{H_{ij}}
  \end{align*}
  as desired.
\end{proof}

\begin{lemma}
  \label{groupoidal-colim-reedy-monomorphic}
  Any colimit of a groupoid of representables in $\PSh{\CatR}$ is Reedy monic.
\end{lemma}
\begin{proof}
  Let a groupoid $\CatG$ and $d \co \CatG \to \CatR$ be given. Set $C \defeq \colim_{i \in \CatG} \yo{d^i}$.
  We show that $C$ has unique EZ decompositions. Let two EZ decompositions $(e_0,x_0)$ and $(e_1,x_1)$ of the
  same element of $C$ be given. As colimits are computed pointwise, each $x_k$ factors as $x_k= \Gi_km_k$
  through some leg $\Gi_k \co \yo{d^{i_k}} \to C$ of the coproduct and we have an arrow $g \co i_0 \cong i_1$
  in $\CatG$ making the following diagram commute:
  \[
    \begin{tikzcd}[column sep=huge, row sep=tiny]
      & \yo{s_0} \ar[bend left=25]{drr}[above right]{x_0} \ar[raising=above,dashed]{r}[below]{m_0} & \yo{d^{i_0}} \ar[dashed]{dd}[left]{d^g}[sloped]{\cong} \ar[dashed]{dr}[description, near start]{\Gi_0} \\
      \yo{r} \ar[lowering]{ur}{e_0} \ar[lowering=above]{dr}[below left]{e_1} & & &[2em] C \rlap{.} \\
      & \yo{s_1} \ar[bend right=25]{urr}[below right]{x_1} \ar[raising,dashed]{r}{m_1} & \yo{d^{i_1}} \ar[dashed]{ur}[description, near start]{\Gi_1}
    \end{tikzcd}
  \]
  Each $m_k$ must be a raising map because $x_k$ is non-degenerate. By uniqueness of Reedy factorizations, we
  have an isomorphism $\Gth \co s_0 \cong s_1$ fitting in the diagram above.
\end{proof}

\begin{theorem}
  \label{saturated-reedy-cofibrant}
  Let $\CatR$ be a Reedy category in which isos act freely on lowering maps. Let $\CP \subseteq \PSh{\CatR}$
  be a class of objects such that
  \begin{itemize}
  \item for any $r \in \CatR$ and $H \le \Aut[\CatR]{r}$, we have $\autquo{\yo{r}}{H} \in \CP$;
  \item $\CP$ is saturated by monomorphisms.
  \end{itemize}
  Then $\CP$ contains every Reedy monic presheaf.
\end{theorem}
\begin{proof}
  First we show by induction on $n$ that $\sk{n}X \in \CP$ for any Reedy monic presheaf $X$.  It then
  follows that $X \cong \colim_{n \in \BN} \sk{n}X \in \CP$ by saturation.

  In the base case, $\sk{0}X$ is the empty coproduct and thus belongs to $\CP$ by saturation.  For any
  $n \in \BN$, we have the following pushout square by \cref{cell-pushout-presheaf}:
  \[
    \begin{tikzcd}[column sep=huge]
      \bullet \ar{d} \ar[tail]{r}{\lbwcolim{\op{\degcore{n}{\CatR}}}{\inbdyu{n}{\CatR}}{\latchmap{n}{X}}} & \bullet \ar{d} \\
      \sk{n}X \ar[tail]{r} & \pushout \sk{n+1}X \rlap{.}
    \end{tikzcd}
  \]
  The upper horizontal map is monic by \cref{reedy-mono-to-cell-mono}, the lower by closure of monos in
  $\PSh{\CatR}$ under cobase change. We have $\sk{n}X \in \CP$ by induction hypothesis. The upper-right corner is
  $\wcolim{\op{\degcore{n}{\CatR}}}{\yo^n\CatR}{\latchcod{n}{X}}$, which belongs to $\CP$ by
  \cref{cell-is-groupoidal-colimit}. Finally, the upper-left corner is by definition the following pushout
  object:
  \[
    \begin{tikzcd}
      \wcolim{\op{\degcore{n}{\CatR}}}{\bdyu{n}{\CatR}}{\latchob{n}{X}} \ar{d} \ar[tail]{r} & \wcolim{\op{\degcore{n}{\CatR}}}{\yo^n\CatR}{\latchob{n}{X}} \ar{d} \\
      \wcolim{\op{\degcore{n}{\CatR}}}{\bdyu{n}{\CatR}}{\latchcod{n}{X}} \ar[tail]{r} & \pushout \bullet \rlap{~.}
    \end{tikzcd}
  \]
  The upper horizontal map is monic by \cref{weighted-colim-mono,boundary-projective-mono}, as we can write
  it as the pushout product
  $\lbwcolim{\op{\degcore{n}{\CatR}}}{\inbdyu{n}{\CatR}}{(\emptyset \mono \latchob{n}{X})}$.  The object
  $\wcolim{\op{\degcore{n}{\CatR}}}{\yo^n\CatR}{\latchob{n}{X}}$ is in $\CP$ by
  \cref{cell-is-groupoidal-colimit}. Using \cref{weighted-colim-assoc}, we have
  \begin{align*}
    \wcolim{\op{\degcore{n}{\CatR}}}{\bdyu{n}{\CatR}}{F}
    \cong \wcolim{\op{\degcore{n}{\CatR}}}{(\wcolim{\op{\CatR}}{\sk{n}{\CatR}}{\yo^n\CatR})}{F}
    \cong \wcolim{\op{\CatR}}{\sk{n}\CatR}{(\wcolim{\op{\degcore{n}{\CatR}}}{\yo^n\CatR}{F})}
  \end{align*}
  for any $F$. The objects $\wcolim{\op{\degcore{n}{\CatR}}}{\bdyu{n}{\CatR}}{\latchob{n}{X}}$ and
  $\wcolim{\op{\degcore{n}{\CatR}}}{\bdyu{n}{\CatR}}{\latchcod{n}{X}}$ thus belong to $\CP$ by
  \cref{cell-is-groupoidal-colimit,groupoidal-colim-reedy-monomorphic} and the induction hypothesis. By
  saturation, the upper-left corner of our original pushout diagram now belongs to $\CP$. For the same reason,
  we conclude that $\sk{n+1}X$ belongs to $\CP$.
\end{proof}

\subsection{Pre-elegant Reedy categories}
\label{sec:pre-elegance}

We next consider the subclass of Reedy categories in which any span of lowering maps has a pushout. This
restriction has some simplifying consequences (\eg, that all lowering maps are epic), and we can
characterize the Reedy monic presheaves over such categories as those preserving lowering pushouts.

\begin{definition}
  \label{def:pre-elegant}
  A Reedy category is \emph{pre-elegant} when it has pushouts of lowering spans.
\end{definition}

Intuitively, this means that any pair of lowering maps from the same object has a universal combination, the
diagonal of their pushout. Of course, any elegant Reedy category is pre-elegant, so $\Simp$ is one
example. Our motivating example is the (surjective, mono) Reedy structure on the category of finite inhabited
semilattices, which is pre-elegant but not elegant. In \cref{sec:algebra-elegant-core}, we see this is an
instance of a general class of examples: the (surjective, mono) Reedy structure on the category
$\Alg{\CatT}_\lfin$ of finite algebras for a Lawvere theory $\CatT$ is always pre-elegant, but not necessarily
elegant.

The following lemma generalizes the fact that any lowering map in an elegant Reedy category is split epic,
with essentially the same proof as Bergner and Rezk's Proposition 3.8(3) \cite{bergner13}.

\begin{lemma}
  \label{pre-elegant-lowering-epi}
  Let $\CatR$ be a pre-elegant Reedy category.
  Then any lowering map is epic.
\end{lemma}
\begin{proof}
  Consider a lowering map $e \co r \lowering s$. We take the pushout of $e$ with itself, then
  use its universal property to see that the legs of the pushout are split monic:
  \[
    \begin{tikzcd}
      r \ar[lowering=above]{d}[left]{e} \ar[lowering]{r}{e} & s \ar[dashed,lowering={near start}]{d}{f_1} \ar[bend left]{ddr}{\id} \\
      s \ar[bend right]{drr}[below left]{\id} \ar[dashed,lowering={above}]{r}[below]{f_0} & t \pushout \ar[dashed]{dr} \\
      & & s \rlap{.}
    \end{tikzcd}
  \]
  Any split mono is a raising map (\cref{split-epi-is-lowering}), so $f_0,f_1$ are isomorphisms.
  Thus $e$ is epic.
\end{proof}

\begin{corollary}
  If $\CatR$ is a pre-elegant Reedy category, then isos act freely on lowering maps in $\CatR$.
\qed
\end{corollary}

\begin{lemma}
  \label{pre-elegant-unique-non-degenerate-to-pushout}
  Let $\CatR$ be a Reedy category in which isos act freely on lowering maps. If $X \in \PSh{\CatR}$ is Reedy
  monic, then $X$ sends pushouts of lowering spans (should they exist) to pullbacks.
\end{lemma}
\begin{proof}
  Let a pushout square of lowering maps be given like so:
  \[
    \begin{tikzcd}
      r \ar[lowering=above]{d}[left]{e_0} \ar[lowering]{r}{e_1} & s_1 \ar[lowering={near start}]{d}{f_1} \\
      s_0 \ar[lowering=above]{r}[below]{f_0} & \pushout t \rlap{.}
    \end{tikzcd}
  \]
  Suppose we have $x_0 \in X_{s_0}$ and $x_1 \in X_{s_1}$ such that $x_0e_0 = x_1e_1$; we show this data
  determines a unique element of $X_t$ restricting to $x_0$ and $x_1$. For each $k \in \{0,1\}$, take an EZ
  decomposition $(g_k,y_k)$ of $x_k$. Then $(g_0e_0,y_0)$ and $(g_1e_1,y_1)$ are EZ decompositions of the same
  map, so by \cref{cofibrant-iff-unique-non-degenerate} they are isomorphic via some $\Gth \co u_0 \cong
  u_1$. The universal property of the pushout in $\CatR$ then provides a map $h_1 \co t \to u_1$ like so:
  \[
    \begin{tikzcd}
      & \yo{s_1} \ar[lowering]{dr}{f_1} \ar[lowering]{rr}{g_1} & & \yo{u_1} \ar{dr}{y_1} \\
      \yo{r} \ar[lowering=above]{dr}[below left]{e_0} \ar[lowering]{ur}{e_1} & & \yo{t} \ar[dashed]{ur}{h_1} & & X \\
      & \yo{s_0} \ar[lowering=above]{ur}[below right]{f_0} \ar[lowering=above]{rr}[below]{g_0} & & \yo{u_0} \ar{uu}[sloped]{\cong}[right]{\Gth} \ar{ur}[below right]{y_0}
    \end{tikzcd}
  \]
  This gives our desired element $y_1h_1 \in X_t$ restricting to $x_k$ along each $f_k$. Note that $h_1$ is a
  lowering map by \cref{reedy-cancellation}.

  To see that this element is unique, suppose we have $x \in X_t$ such that $xf_k = x_k$ for $k \in
  \{0,1\}$. Take an EZ decomposition $(h,y)$ of $X$, say through $u \in \CatR$. By uniqueness of EZ
  decompositions, we have isomorphisms $\Gps_k$ as shown:
  \[
    \begin{tikzcd}
      & \yo{s_1} \ar[lowering]{dr}{f_1} \ar[lowering]{rr}{g_1} & & \yo{u_1} \ar{dr}{y_1} \\
      \yo{r} \ar[lowering=above]{dr}[below left]{e_0} \ar[lowering]{ur}{e_1} & & \yo{t} \ar[lowering]{r}{h} & \yo{u}  \ar[dashed,<-]{u}[sloped]{\cong}[right]{\Gps_1} \ar{r}[description]{y} & X \\
      & \yo{s_0} \ar[lowering=above]{ur}[below right]{f_0} \ar[lowering=above]{rr}[below]{g_0} & & \yo{u_0} \ar[dashed]{u}[sloped]{\cong}[right]{\Gps_0} \ar{ur}[below right]{y_0}
    \end{tikzcd}
  \]
  Because isos act freely on lowering maps, we have $\Gps_1^{-1}\Gps_0 = \Gth$. It follows from the universal
  property of the pushout in $\CatR$ that $\Gps_1h = h_1$, thus that $yh = y_1h_1$ as desired.
\end{proof}

\begin{theorem}
  \label{pre-elegant-cofibrant}
  If $\CatR$ is a pre-elegant Reedy category, then $X \in \PSh{\CatR}$ is Reedy monic if and only if it
  sends pushouts of lowering spans to pullbacks.
\end{theorem}
\begin{proof}
  One direction is \cref{pre-elegant-unique-non-degenerate-to-pushout}. For the other, suppose $X$ sends
  pushouts of lowering spans to pullbacks. By \cref{cofibrant-iff-unique-non-degenerate}, it suffices to show
  $X$ has unique EZ decompositions. Let $(e_0,x_0)$ and $(e_1,x_1)$ be EZ decompositions of the same
  element. We have an induced element as shown:
  \[
    \begin{tikzcd}
      \yo{r} \ar[lowering=above]{d}[left]{e_0} \ar[lowering]{r}{e_1} & \yo{s_1} \ar[lowering]{d}{\inr} \ar[bend left]{ddr}{x_1} \\
      \yo{s_0} \ar[bend right=20]{drr}[below left,near start]{x_0} \ar[lowering=above]{r}[below]{\inl} & \yo{(s_0 \sqcup_r s_1)} \ar[dashed]{dr} \\
      & & X \rlap{.}
    \end{tikzcd}
  \]
  By non-degeneracy of $x_0$ and $x_1$, the maps $\inl$ and $\inr$ must be isomorphisms, so $(e_0,x_0)$ and
  $(e_1,x_1)$ are isomorphic.
\end{proof}

\begin{remark}
  A corollary of the previous theorem is that a pre-elegant Reedy category $\CatR$ is elegant if and only if
  all presheaves on $\CatR$ are Reedy monic. \citeauthor*{bergner13}~\cite[Proposition 3.8]{bergner13}
  show that this bi-implication actually holds for \emph{any} Reedy category. That is, if all presheaves on
  $\CatR$ are Reedy monic, then $\CatR$ is necessarily pre-elegant (and thus elegant).
\end{remark}

\subsection{Relative elegance}
\label{sec:relative-elegance}

Now we come to our central definition, elegance of a category \emph{relative to} a full subcategory.

\begin{definition}
  We say that a pre-elegant Reedy category $\CatR$ is \emph{elegant relative to} a fully faithful functor
  $i \co \CatC \to \CatR$ if the nerve $\nerve{i} \defeq \subst{i}\yo \co \CatR \to \PSh{\CatC}$ preserves
  pushouts of lowering spans. We also say that $i$ \emph{is relatively elegant} with the same meaning.
\end{definition}

\begin{remark}
  \label{relative-elegance-pointwise}
  As pushouts in $\PSh{\CatC}$ are computed pointwise, $i$ is relatively elegant if and only if
  $\Hom{\CatR}{ia}{-} \co \CatR \to \Set$ preserves lowering pushouts for all $a \in \CatC$.
\end{remark}

\begin{remark}
  A Reedy category is elegant if and only if it is elegant relative to the identity functor, in which case the
  nerve is simply the Yoneda embedding. At the other extreme, any pre-elegant Reedy category is elegant
  relative to the unique functor $\catInit \to \CatR$.
\end{remark}

\begin{lemma}
  \label{elegant-nerve-lowering-to-epi}
  If $\CatR$ is elegant relative to $i \co \CatC \to \CatR$, then $\nerve{i} \co \CatR \to \PSh{\CatC}$ sends
  lowering maps to epimorphisms.
\end{lemma}
\begin{proof}
  By \cref{pre-elegant-lowering-epi}, any $e \in \CatR^-$ fits in the pushout square
  \[
    \begin{tikzcd}
      r \ar[lowering=above]{d}[left]{e} \ar[lowering]{r}{e} & s \ar{d}{\id} \\
      s \ar{r}[below]{\id} & \pushout s \rlap{.}
    \end{tikzcd}
  \]
  which is then preserved by $\nerve{i}$.
\end{proof}

\begin{corollary}
  \label{relative-elegant-projective}
  If $\CatR$ is elegant relative to $i \co \CatC \to \CatR$, then objects in the image of $i$ are
  $\CatR^-$-projective: given a lowering map $e \co r \lowering s$ and $f \co ia \to s$, there exists a lift
  as below.
  \[
    \begin{tikzcd}
      & r \ar[lowering]{d}{e} \\
      ia \ar[dashed]{ur} \ar{r}[below]{f} & s
    \end{tikzcd}
  \]
\end{corollary}
\begin{proof}
  By \cref{elegant-nerve-lowering-to-epi}, $\nerve{i} e \co \nerve{i}r \to \nerve{i}s$ is epic; this
  means exactly post-composition with $e$ is a surjective map $\Hom{\CatR}{ia}{r} \to \Hom{\CatR}{ia}{s}$.
\end{proof}

\begin{remark}
  \label{elegant-lowering-split-epi}
  As a special case of the corollary above, we recover the fact that lowering maps in elegant Reedy categories
  are split epimorphisms \cite[Proposition 3.8]{bergner13}. Split epis are lowering maps in any Reedy category
  (\cref{split-epi-is-lowering}), so in the elegant case they coincide. It is not generally the case that
  the lowering maps in a Reedy category $\CatR$ elegant relative to some $i$ are exactly those sent to
  epimorphisms by $\nerve{i}$: consider that $\CatR$ is always elegant relative to $\mathbf{0} \to \CatR$.
\end{remark}

On the basis of \cref{relative-elegance-pointwise}, we can identify the maximal subcategory relative to which
a pre-elegant Reedy category $\CatR$ is elegant.

\begin{definition}
  Let $\CatR$ be a pre-elegant Reedy category. We define its \emph{elegant core}
  $\elcore{\CatR}$ to be the full subcategory of $\CatR$ consisting of objects $r$ such
  that $\Hom{\CatR}{r}{-}$ preserves lowering pushouts.
\end{definition}

\begin{proposition}
  An fully faithful functor $i \co \CatC \to \CatR$ into a pre-elegant Reedy category is relatively elegant
  exactly if it factors through the inclusion $\elcore{\CatR} \to \CatR$.  \qed
\end{proposition}

We can give another characterization of relative elegance in terms of the right Kan extension
$\ran{i} \co \PSh{\CatC} \to \PSh{\CatR}$:

\begin{lemma}
  \label{ran-elegant}
  Let $\CatR$ be a pre-elegant Reedy category. Then $i \co \CatC \to \CatR$ is relatively elegant if and only
  if $\ran{i} X \in \PSh{\CatR}$ is Reedy monic for every $X \in \PSh{\CatC}$.
\end{lemma}
\begin{proof}
  By definition, $i \co \CatC \to \CatR$ is relatively elegant exactly if $\nerve{i} = \subst{i} \yo{}$ preserves lowering pushouts.
  Testing pushouts by mapping out of them, this holds exactly if $\Hom{\PSh{\CatC}}{\subst{i} \yo{-}}{X}$ sends lowering pushouts to pullbacks for every $X \in \PSh{\CatC}$.
  Using the natural isomorphism
    \[\Hom{\PSh{\CatC}}{\subst{i}\yo{-}}{X} \cong \Hom{\PSh{\CatR}}{\yo{-}}{\ran{i}{X}} \cong \ran{i}X \rlap{,}\]
  this rewrites to $\ran{i}X$ sending lowering pushouts to pullbacks.
\end{proof}

This property of presheaves extends to morphisms as follows.

\begin{definition}\label{reflect-degeneracy}
A map $m \co X \to Y$ in $\PSh{\CatR}$ \emph{reflects degeneracy} if has the right lifting property against lowering maps $e \co \yo{r} \lowering \yo{s}$.
\end{definition}

This means that for any $x \in X_r$, if $m_r(x)$ factors through some $e \co \yo r \lowering \yo s$, then $x$ also factors through $e$.

\begin{lemma}
  \label{reflects-degeneracy-to-reedy-monomorphic}
  Let $\CatR$ be a Reedy category, let $Y \in \PSh{\CatR}$ be Reedy monic, and let $m \co X \mono
  Y$ be a degeneracy-reflecting monomorphism. Then $m$ is Reedy monic.
\end{lemma}
\begin{proof}
By \cref{latching-object-only-lowering-maps}, it suffices to show, for any $r \in \CatR$, that the pushout gap map in the naturality square
\[
\begin{tikzcd}[column sep=small]
\latchob{r}{X_-}
\ar[r]
\ar[d]
&
X_r
\ar[d]
\\
\latchob{r}{Y_-}
\ar[r]
&
Y_r
\end{tikzcd}
\]
is monic.
The bottom and right maps are monic by assumption.
Because $m$ reflects degeneracy, the square is a weak pullback, \ie, the pullback gap map is surjective.
This means that the pushout gap map, seen as an object over $Y_r$, is the union of the subobjects given by the bottom and right maps.
\end{proof}

\begin{corollary}
  \label{ran-elegant-morphisms}
  If $i \co \CatC \to \CatR$ is relatively elegant, then $\ran{i} m$ is Reedy monic for every
  $m \co X \mono Y$ in $\PSh{\CatC}$.
\end{corollary}
\begin{proof}
By \cref{reflects-degeneracy-to-reedy-monomorphic}, it suffices to show that $\ran{i}m$ reflects degeneracy.
For any $e \co r \lowering s$,
$\nerve{i} e$ is epic by \cref{elegant-nerve-lowering-to-epi}, so has
left lifting against monos.  By transposition, $e$ has left lifting against $\ran{i}m$.
\end{proof}

In any presheaf category, all monomorphisms can be presented as cell complexes (transfinite composites of
cobase changes of coproducts) of monomorphisms whose codomains are quotients of representables
\cite[Proposition 1.2.27]{cisinski06}. With \cref{ran-elegant-morphisms}, we can give an alternative---not
necessarily comparable---set of generators in terms of the boundary inclusions in $\CatR$.

\begin{theorem}
  If $i \co \CatC \to \CatR$ is relatively elegant, then every monomorphism in $\PSh{\CatC}$ is a cell complex
  of maps of the form $\subst{i}(\wcolim{\op{\degcore{n}{\CatR}}}{\inbdyu{n}{\CatR}}{(\autquo{\yo{r}}{H})})$
  where $r \in \CatR$ and $H \le \Aut[\CatR]{r}$.
\end{theorem}
\begin{proof}
  Let $m \co X \mono Y$ in $\PSh{\CatC}$. By \cref{reedy-cell-complex}, $\ran{i} m$ has a cellular
  presentation by maps of the form
  $\lbwcolim{\op{\degcore{n}{\CatR}}}{\inbdyu{n}{\CatR}}{\latchmap{n}{(\ran{i}m)}}$; by
  \cref{ran-elegant-morphisms}, each $\latchmap{n}{(\ran{i}m)}$ is monic in
  $\PSh{\degcore{n}{\CatR}}$.  In $\PSh{\degcore{n}{\CatR}}$, any monomorphism is a cell complex of maps of
  the form $0 \mono \autquo{\yo{r}}{H}$ for some $r \in \degcore{n}{\CatR}$ and $H \le \Aut[\CatR]{r}$,
  because $\PSh{\degcore{n}{\CatR}}$ is Boolean and any $\degcore{n}{\CatR}$-set decomposes as a coproduct of
  orbits. By \cite[Lemma 5.7]{riehl14b}, it follows that $\ran{i}m$ is a cell complex of maps
  $\lbwcolim{\op{\degcore{n}{\CatR}}}{\inbdyu{n}{\CatR}}{(0 \mono \autquo{\yo{r}}{H})}$. Finally, $\subst{i}$ preserves colimits and thus cell complexes.
\end{proof}

Finally, we exploit the fact that $\subst{i}$ preserves the operations of saturation by monomorphisms to
transfer the induction principle on the Reedy monic presheaves of $\PSh{\CatR}$ given by
\cref{saturated-reedy-cofibrant} to $\PSh{\CatC}$.

\begin{theorem}
  \label{saturated-elegant-relative}
  Let $\CatR$ be elegant relative to $i \co \CatC \to \CatR$. Let $\CP \subseteq \PSh{\CatC}$ be a class of
  objects such that
  \begin{itemize}
  \item for any $r \in \CatR$ and $H \le \Aut[\CatR]{r}$, we have $\autquo{\nerve{i}{r}}{\nerve{i}{H}} \in \CP$;
  \item $\CP$ is saturated by monomorphisms.
  \end{itemize}
  Then $\CP$ contains every presheaf in $\PSh{\CatC}$.
\end{theorem}
\begin{proof}
  As a left and right adjoint, $\subst{i}$ preserves colimits and monomorphisms. The class
  $(\subst{i})^{-1}\CP$ of $X \in \PSh{\CatR}$ such that $\subst{i}X \in \CP$ is thus saturated by
  monomorphisms. By our first assumption and the fact that $\subst{i}$ preserves colimits, we have
  $\autquo{\yo{r}}{H} \in (\subst{i})^{-1}\CP$ for every $r \in \CatR$ and $H \le \Aut[\CatR]{r}$. By
  \cref{saturated-reedy-cofibrant,ran-elegant}, we thus have $\ran{i}X \in (\subst{i})^{-1}\CP$ for all
  $X \in \PSh{\CatC}$. Hence $X \cong \subst{i}\ran{i}X \in \CP$ for all $X \in \PSh{\CatC}$.
\end{proof}

\section{Reedy structures on categories of finite algebras}
\label{sec:algebra-elegant-core}

\subsection{Finite algebras}

Per \cref{sec:disjunctive-cubical-sets}, $\Or$ and its idempotent completion can be regarded as full
subcategories of the category $\SLatF$ of finite semilattices. Any category of finite algebras of a Lawvere theory carries a
natural Reedy structure: the degree of an object is its cardinality, and the lowering and raising maps are
given by the (surjective, mono) factorization system. Here we observe that this Reedy structure is
pre-elegant and characterize its elegant core in the case where free finitely-generated algebras are
finite. As a corollary, the embedding $\Or \to \SLatF$ and its restriction
$\Or \to \SLatFNE$ to inhabited algebras are relatively elegant.

For this section, we fix a Lawvere theory $\CatT$. We recall a few basic properties of its category of
algebras.

\begin{proposition}[{{\thmcite[Corollary 3.5]{adamek10}}}]
  A morphism $f$ in $\Alg{\CatT}$ is regular epic if and only $Uf$ is surjective.
\qed
\end{proposition}

\begin{proposition}[{{\thmcite[Corollary 3.7]{adamek10}}}]
  \label{algebra-factorization}
  Any morphism in $\Alg{\CatT}$ factors as a regular epi followed by a mono.
\qed
\end{proposition}

Write $\Alg{\CatT}_\lfin \to $ and $\Alg{\CatT}^\linh_\lfin$ for the full subcategories of $\Alg{\CatT}$
consisting of algebras with finite and finite inhabited underlying sets respectively.  When we write
$\Alg{\CatT}^{(\linh)}_\lfin$ below, the relevant statement or proof applies to both of these.

\begin{corollary}
  The (surjective, mono) factorization system restricts to a Reedy structure on $\Alg{\CatT}^{(\linh)}_\lfin$ with
  degree map given by cardinality.
\qed
\end{corollary}

As any category of algebras has limits and colimits~\cite[Proposition 1.21, Theorem
4.5]{adamek10}, $\Alg{\CatT}$ has in particular pushouts of spans of surjections.

\begin{corollary}
  The Reedy structure on $\Alg{\CatT}^{(\linh)}_\lfin$ is pre-elegant.
\end{corollary}
\begin{proof}
  The pushout of a span of surjections has cardinality bounded by those of the objects in the span, as
  surjections are left maps and thus closed under cobase change.
\end{proof}

Recall that the forgetful functor $U$ preserves limits. While $U$ does not generally preserve colimits, we can
show that it preserves pushouts of surjective spans using the technology of sifted colimits.

\begin{definition}
  A small category $\CatD$ is
  \begin{itemize}
  \item \emph{filtered} if $\colim_\CatD \co \Func{\CatD}{\Set} \to \Set$ commutes with finite limits;
  \item \emph{sifted} if $\colim_\CatD \co \Func{\CatD}{\Set} \to \Set$ commutes with finite products.
  \end{itemize}
  A \emph{filtered (sifted) colimit} is a colimit over a filtered (sifted) category.
\end{definition}

Recall that a \emph{reflexive coequalizer} is a coequalizer of maps $f_0,f_1 \co A \to B$ with a mutual
section, that is, some $d \co B \to A$ such that $f_0d = f_1d = \id$. Reflexive coequalizers are sifted (but
not filtered) colimits \cite[Remark 3.2]{adamek10}.

\begin{lemma}
  \label{lex-and-preserves-sifted-to-preserves-regular-epi-pushout}
  Let $F \co \CatC \to \CatD$ be a functor between regular categories preserving finite limits and sifted
  colimits.
  Then $F$ preserves pushouts of regular epi spans.
\end{lemma}
\begin{proof}
  Let a span $B_0 \overset{e_0}{\epi*} A \overset{e_1}{\epi} B_1$ in $\CatC$ be given. We compute the following
  reflexive coequalizer:
  \[
    \begin{tikzcd}
      A \times_{B_0} A \times_{B_1} A \ar[yshift=0.5ex]{r}{\Gp_0} \ar[yshift=-0.5ex]{r}[below]{\Gp_2} & A \ar[bend right=60]{l}[above]{\angles{\id,\id,\id}} \ar[two heads, dashed]{r}{e} & B
    \end{tikzcd}
  \]
  It is straightforward to check, using the characterizations of $e_0$, $e_1$ as the coequalizers of their
  kernel pairs, that we have induced maps $B_0 \epi B \epi* B_1$ exhibiting $B$ as the pushout of our span. As
  $F$ preserves the diagram above, it preserves this pushout.
\end{proof}

\begin{corollary}
  \label{forgetful-preserves-surjective-pushout}
  $U \co \Alg{\CatT} \to \Set$ preserves pushouts of surjective spans.
\end{corollary}
\begin{proof}
  $U$ preserves limits and sifted colimits \cite[Proposition 2.5]{adamek10}.
\end{proof}

We now assume that any $\CatT$-algebra free on a finite set has a finite underlying set. In this case, the
elegant core coincides with the class of \emph{perfectly presentable} (also called \emph{strongly finitely
  presentable}) algebras.

\begin{definition}[{{\thmcite[Definition 5.3]{adamek10}}}]
  An object $A$ of a category $\CatC$ is
  \begin{itemize}
  \item \emph{finitely presentable} if $\Hom{\CatC}{A}{-} \co \CatC \to \Set$ preserves filtered colimits;
  \item \emph{perfectly presentable} if $\Hom{\CatC}{A}{-} \co \CatC \to \Set$ preserves sifted colimits.
  \end{itemize}
\end{definition}

\begin{proposition}[{{\thmcite[Corollary 5.16 and Proposition 11.28]{adamek10}}}]
  \label{perfectly-presentable-characterization}
  Let $A \in \Alg{\CatT}$. The following are equivalent:
  \begin{itemize}
  \item $A$ is perfectly presentable;
  \item $A$ is finitely presentable and regular projective;
  \item $A$ is a retract of a finitely-generated free algebra. \qed
  \end{itemize}
\end{proposition}

\begin{theorem}
  \label{algebra-elegant-core}
  Suppose that every finitely-generated free algebra in $\Alg{\CatT}$ has a finite underlying set. Then the
  elegant core of $\Alg{\CatT}^{(\linh)}_\lfin$ is the subcategory of objects perfectly presentable in
  $\Alg{\CatT}$.
\end{theorem}
\begin{proof}
  Suppose $A \in \Alg{\CatT}^{(\linh)}_\lfin$ is in the elegant core of the Reedy structure. By assumption, the free
  algebra $FUA$ belongs to $\Alg{\CatT}^{(\linh)}_\lfin$, and the counit $\Ge_A \co FUA \to A$ is clearly surjective.
  Then by \cref{relative-elegant-projective}, we have a lift
  \[
    \begin{tikzcd}
      & FUA \ar{d}{\Ge_A} \\
      A \ar[dashed]{ur} \ar{r}[below]{\id} & A
    \end{tikzcd}
  \]
  exhibiting $A$ a retract of a free algebra. Thus $A$ is perfectly presentable. Conversely, if $A$ is
  perfectly presentable, then $\Hom{\Alg{\CatT}}{A}{-} \co \Alg{\CatT} \to \Set$ preserves
  finite limits and sifted colimits, so preserves pushouts of lowering spans by
  \cref{lex-and-preserves-sifted-to-preserves-regular-epi-pushout}.
\end{proof}

\subsection{Semilattice cubes}
\label{sec:idcomp-elegant-embedding}

Applying the preceding results, we have a (surjective, mono) Reedy structure on $\SLatFNE$. We can give a
concrete description of its elegant core.

\begin{lemma}
  A semilattice $A \in \SLatFNE$ is in the elegant core of the (surjective, mono) Reedy structure if and only
  if $1 \join A$ is a distributive lattice.
\end{lemma}
\begin{proof}
  By \cref{algebra-elegant-core}, the elegant core consists of the perfectly presentable objects in $\SLat$.
  By \cref{perfectly-presentable-characterization}, these are the finite regular projectives in $\SLat$.
  These are characterized as above by \cref{semilattice-epi,semilattice-injective-projective}.
\end{proof}

\begin{theorem}
  \label{alg-bot-elegant}
  The inclusion $i \co \IdCompOr \to\SLatFNE$ is relatively elegant.
\end{theorem}
\begin{proof}
  If $A \in \SLatFNE$ is a distributive lattice, then $1 \join A$ is a distributive lattice as well, so $A$ is
  in the elegant core of $\SLatFNE$.
\end{proof}

\begin{remark}
  The subcategory $\SLatF^\bot$ of $\SLatFNE$ consisting of finite semilattices with a minimum element is
  closed under Reedy factorizations and lowering pushouts, so $\IdCompOr \to \SLatF^\bot$ is also relatively
  elegant. This embedding gives a more parsimonious set of generators, but $\SLatFNE$ suffices for our
  purposes.
\end{remark}

\section{Equivalences and equalities}
\label{sec:equivalences}

\subsection{Equivalence with the Kan--Quillen model structure}
\label{sec:equivalence-kan-quillen}

Returning to the candidate Quillen equivalence $\lan{\insimp} \dashv \subst{\insimp\!}$, it remains to show
that its counit is valued in weak equivalences. We first note that the collection of those
$X \in \PSh{\IdCompOr}$ for which $\Ge_X \co \lan{\insimp}\subst{\insimp\!}X \to X$ is a weak equivalence is
saturated by monomorphisms.

\begin{proposition}[{{\thmcite[Remarque 1.1.13]{cisinski06}}}]
  \label{saturated-preimage}
  Let $F \co \CatE \to \CatF$ be a mono- and colimit-preserving functor between cocomplete
  categories. If $\CP \subseteq \CatF$ is saturated by monos, then the class $F^{-1}(\CP)$ of objects
  whose image by $F$ is in $\CP$ is saturated by monos.
\qed
\end{proposition}

\begin{proposition}
  \label{saturated-weq}
  If $\CatM$ has monos as cofibrations, then its class of weak equivalences is saturated by monos
  as a class of objects of $\CatM^\to$.
\end{proposition}
\begin{proof}
  This is proven by Cisinski \cite[Remarque 1.4.16]{cisinski06} for \emph{localizers}~\cite[D\'efinition
  1.4.1]{cisinski06}; the class of weak equivalences in a model category with monos as cofibrations is
  always a localizer.
\end{proof}

\begin{corollary}
  \label{saturated-nat-trans-weq}
  Let $\CatE$ be a cocomplete category, $\CatN$ be a model category with monos as cofibrations, and
  $F,G \co \CatE \to \CatN$ be mono- and colimit-preserving functors. For any natural transformation
  $h \co F \to G$, the class of objects $X \in \CatE$ such that $h_X \co FX \to GX$ is a weak equivalence is
  saturated by monos.
\end{corollary}
\begin{proof}
  By \cref{saturated-preimage,saturated-weq}, regarding $h$ as a functor $\CatE \to \CatN^\to$.
\end{proof}

In particular, any natural transformation $h \co F \to G$ of left Quillen adjoints $F,G \co \CatM \to \CatN$
between model categories with monos as cofibrations satisfies the hypotheses of
\cref{saturated-nat-trans-weq}. In light of this, we only need to check that $\Ge$ is a weak equivalence at
generating presheaves.

\begin{lemma}
  \label{alg-contractible}
  Let $A \in \SLatFNE$ and $H \le \Aut[\SLatFNE]{A}$ be given. Then $\autquo{\nerve{i}{A}}{\nerve{i}{H}}$
  is weakly contractible.
\end{lemma}
\begin{proof}
  Per \cref{homotopy-retracts-preserve-weakly-contractible}, it suffices to show that this object is a
  homotopy retract of $1$.  We have a semilattice morphism ${\uparrow} \co [1] \times A \to A$ sending
  $(0,a) \mapsto a$ and $(1,a) \mapsto \top$. Any automorphism $g \in H$ preserves maximum elements, so we
  have a diagram like so:
  \[
    \begin{tikzcd}
      {[1]} \times A \ar{d}[left]{[1] \times g} \ar{r}{\uparrow} & A \ar{d}{g} \\
      {[1]} \times A \ar{r}[below]{\uparrow} & A
    \end{tikzcd}
  \]
  We thus obtain a contracting homotopy
  $\ival \times (\autquo{\nerve{i}{A}}{\nerve{i}{H}}) \to (\autquo{\nerve{i}{A}}{\nerve{i}{H}})$, using that
  $\nerve{i}([1] \times A) \cong \ival \times \nerve{i}A$ and that $\ival \times (-)$ commutes with colimits.
\end{proof}

\begin{lemma}
  \label{counit-weq}
  The counit map $\Ge_X \co \lan{\insimp}\subst{\insimp\!}X \to X$ is a weak equivalence for every
  $X \in \PSh{\IdCompOr}$.
\end{lemma}
\begin{proof}
  Recall that both $\lan{\insimp}$ and $\subst{\insimp\!}$ are left Quillen
  (\cref{L-left-quillen,T-left-quillen}). By \cref{saturated-elegant-relative,saturated-nat-trans-weq}, it
  suffices to show that $\Ge_X \co \lan{\insimp}\subst{\insimp\!}X \to X$ is a weak equivalence whenever $X$
  is an automorphism quotient of an object in the image of $\nerve{i}$. In this case $X$ is weakly
  contractible by \cref{alg-contractible}. As $\lan{\insimp}\subst{\insimp\!}$ preserves the terminal object,
  it preserves weak contractibility by Ken Brown's lemma; thus $\lan{\insimp}\subst{\insimp\!}X$ is weakly
  contractible and so $\Ge_X$ is a weak equivalence by 2-out-of-3.
\end{proof}

\begin{theorem}
  \label{cms-left-quillen-equivalence}
  $\adjunction{\lan{\insimp}}{\KanSimp}{\CMSIdComp}{\subst{\insimp\!}}$ is a Quillen equivalence.
\end{theorem}
\begin{proof}
  By \cref{L-left-quillen,unit-weq,counit-weq}.
\end{proof}

\begin{corollary}
  \label{cms-right-quillen-equivalence}
  $\adjunction{\subst{\insimp\!}}{\CMSIdComp}{\KanSimp}{\ran{\insimp}}$ is a Quillen
  equivalence.
\end{corollary}
\begin{proof}
  Write $\Gh'$ and $\Ge'$ for the unit and counit of this adjunction. The counit is an isomorphism, so
  trivially valued in weak equivalences. To check the derived unit, let $X \in \PSh{\Or}$ and let
  $m \co \subst{\insimp\!}X \tcof (\subst{\insimp\!}X)^\lfib$ be a fibrant replacement. We have the following
  naturality square:
  \[
    \begin{tikzcd}[column sep=5em]
      \lan{\insimp}\subst{\insimp\!}X \ar[weq=above]{d}[left]{\Ge_X} \ar{r}{\lan{\insimp}\subst{\insimp\!}\Gh'_X}[below]{\cong} & \lan{\insimp}\subst{\insimp\!}\ran{\insimp}\subst{\insimp\!}X  \ar[tcof=below]{r}{\lan{\insimp}\subst{\insimp\!}\ran{\insimp}m} & \lan{\insimp}\subst{\insimp\!}\ran{\insimp}(\subst{\insimp\!}X)^\lfib \ar[weq=below]{d}{\Ge_{\ran{\insimp}(\subst{\insimp\!} X)^\lfib}} \\
      X \ar{r}[below]{\Gh'_X} & \ran{\insimp}\subst{\insimp\!}X \ar{r}[below]{\ran{\insimp}m} & \ran{\insimp}(\subst{\insimp\!} X)^\lfib \rlap{.}
    \end{tikzcd}
  \]
  It follows by 2-out-of-3 that the bottom composite is a weak equivalence.
\end{proof}

\begin{theorem}
  \label{triangulation-quillen-equivalence}
  $\adjunction{\Triang}{\CMSOr}{\KanSimp}{\nerve{\triang}}$ is a Quillen equivalence.
\end{theorem}
\begin{proof}
  By the decomposition $\Triang \cong \subst{\insimp\!}\lan{\incube}$ (\cref{middle-adjoint-is-triangulation}).
\end{proof}

In particular, both $\CMSIdComp$ and $\CMSOr$ present $\inftyGpd$.

\subsection{Equality with the test model structure}
\label{sec:test}

It is worth remarking that there is a model structure on $\PSh{\Or}$ already known to present $\inftyGpd$,
namely its \emph{test model structure}. Constructed by Cisinski~\cite{cisinski06} based on Grothendieck's
theory of test categories~\cite{grothendieck83}, a test model structure exists on the category of presheaves
$\PSh{\CatC}$ over any \emph{local test category} $\CatC$. If $\CatC$ is moreover a \emph{test category}, then
this model structure is Quillen equivalent to $\KanSimp$.

Buchholtz and Morehouse observe that $\Or$, among various other cube categories, is a test category
\cite[Corollary 3]{buchholtz17}.%
\footnote{Maltsiniotis~\cite{maltsiniotis09} also observed that a cube category with one connection is a
  strict test category, but a different one: the subcategory of $\Or$ generated by faces, degeneracies, and
  connections, \ie, not including diagonals and permutations.}
Thus it supports a model structure presenting $\inftyGpd$. However, it has not been established whether this
model structure is constructive or compatible with a model of homotopy type theory. Cisinski~\cite{cisinski14}
has shown that the test model structure on an elegant strict Reedy local test category is type-theoretic in
the sense of Shulman~\cite[Definition 6.1]{shulman19}, but the strictness requirement prevents application of
this result to any cube category with permutations (or any non-Reedy category).

By virtue of the Quillen equivalences to $\KanSimp$ already established, we know that $\CMSOr$ and
$\CMSIdComp$ are Quillen equivalent to the test model structures on their respective base categories. Here we
check that they are in fact identical, adapting an argument of Streicher and Weinberger
\cite[\S5]{streicher21}.

We must begin by recalling the main definitions of test category theory. For more detail, we refer the reader
to Maltsiniotis~\cite{maltsiniotis05}, Cisinski \cite{cisinski06}, or Jardine~\cite{jardine06}. The foundation
of test category theory that we can relate presheaves on an arbitrary base category $\CatC$ with simplicial
sets by way of the category of small categories, $\Cat$. We write $\nerve{\Simp} \co \Cat \to \PSh{\Simp}$ for
the nerve of the inclusion $\Simp \to \Cat$.

\begin{definition}
  Given a small category $\CatC$, write $\incat{\CatC} \co \CatC \to \Cat$ for the slice category functor
  $a \mapsto \Slice{\CatC}{a}$. We have an induced nerve functor
  $\subst{\incat{\CatC}} \co \Cat \to \PSh{\CatC}$. As $\Cat$ is cocomplete, this functor has a left adjoint
  $\PSh{\CatC} \to \Cat$, for which we also write $\incat{\CatC}$.
\end{definition}

The composite $\nerve{\Simp}\incat{\CatC} \co \PSh{\CatC} \to \PSh{\Simp}$ is the means by which we can
inherit a model structure on $\PSh{\CatC}$ from $\KanSimp$ under appropriate conditions.

\begin{remark}
  The definitions and results of Cisinski that we cite below are typically parameterized by an arbitrary
  \emph{basic localizer}~\cite[D\'{e}finition 3.3.2]{cisinski06}, a class of functors to be regarded as the
  weak equivalences in $\Cat$. We always instantiate with the \emph{minimal} basic localizer $\CW_\infty$: the
  class of functors $f \co \CatC \to \CatD$ such that
  $\nerve{\Simp} f \co \nerve{\Simp}\CatC \to \nerve{\Simp}\CatD$ is a weak equivalence of
  $\KanSimp$~\cite[Corollaire 4.2.19]{cisinski06}.
\end{remark}

\begin{definition}[{{\thmcite[\S3.3.3 and D\'{e}finition 4.1.3]{cisinski06}}}]
  We say $X \in \PSh{\CatC}$ is \emph{aspheric} if $\nerve{\Simp}\incat{\CatC}X \in \PSh{\Simp}$ is weakly
  contractible in $\KanSimp$.
\end{definition}

\begin{definition}[{{\thmcite[D\'efinitions 4.1.8 and 4.1.12]{cisinski06}}}]
  A small category $\CatC$ is
  \begin{itemize}
  \item a \emph{weak test category} if $\subst{\incat{\CatC}}\CatD$ is aspheric for every $\CatD$
    with a terminal object;
  \item a \emph{local test category} if $\Slice{\CatC}{a}$ is a weak test category for all $a \in \CatC$;
  \item a \emph{test category} if it is both a weak and local test category.
  \end{itemize}
\end{definition}

\begin{proposition}[{{\thmcite[Corollaire 4.2.18]{cisinski06}}}]
  \label{test-model-structure}
  Let $\CatC$ be a local test category. There is a model structure on $\PSh{\CatC}$ in which
  \begin{itemize}
  \item the cofibrations are the monomorphisms;
  \item the weak equivalences are the maps sent by $\nerve{\Simp}\incat{\CatC}$ to a weak equivalence of
    $\KanSimp$.
  \end{itemize}
  We write $\Test{\CatC}$ for this model category.
\qed
\end{proposition}

\begin{remark}
  \label{test-simplicial-is-kan-quillen}
  The test model structure $\Test{\Simp}$ coincides with $\KanSimp$. A proof is contained in the proof of
  \cite[Corollaire 4.2.19]{cisinski06}: the class of weak equivalences of $\Test{\Simp}$ is by definition the
  preimage $\nerve{\Simp}^{-1}\CW_\infty$, which is the minimal test $\Simp$-localizer by Th\'{e}or\`eme
  4.2.15, and said localizer is the class of weak equivalences of $\KanSimp$ by Corollaire 2.1.21 and
  Proposition 3.4.25.
\end{remark}

Note that whereas cubical-type model structures come with explicit characterizations of their cofibrations and
fibrations (or rather generating trivial cofibrations), the test model structure comes with explicit
descriptions of its cofibrations and weak equivalences. In general, $\Test{\CatC}$ is Quillen equivalent to a
slice of $\KanSimp$, namely $\KanSimp/\nerve{\Simp}\CatC$ \cite[Corollaire 4.4.20]{cisinski06}. When $\CatC$
is a test category, $\nerve{\Simp}\CatC$ is weakly contractible, and so we have an equivalence to $\KanSimp$
itself.

We recall the argument used by Buchholtz and Morehouse~\cite[Theorem 1]{buchholtz17} to show that $\Or$ is a
test category---actually a \emph{strict} test category.

\begin{definition}[{{\thmcite[\S4.3.1, Proposition 4.3.2, \S4.3.3]{cisinski06}}}]  We say a category $\CatC$ is \emph{totally aspheric} if it is non-empty and $\yo{a} \times \yo{b}$ is
  aspheric for every $a,b \in \CatC$. A test category that is totally aspheric is called a \emph{strict test
    category}.
\end{definition}

Any representable is aspheric: the category $\incat{\CatC}(\yo{a})$ has a terminal object, thus a natural
transformation from its identity functor to a constant functor, and this induces a contracting homotopy on
$\nerve{\Simp}\incat{\CatC}(\yo{a})$. Thus, any category with binary products is totally aspheric.

The following result originates in~\cite[Section 44(c)]{grothendieck83} and is invoked in~\cite{buchholtz17}
for a broad class of cube categories.

\begin{proposition}[{{\thmcite[Proposition 4.3.4]{cisinski06}}}]
  Let $\CatC$ be a totally aspheric category. If $\PSh{\CatC}$ contains an aspheric presheaf $I$ with disjoint
  maps $e_0,e_1 \co 1 \to I$, then $\CatC$ is a strict test category.
\qed
\end{proposition}

In particular, both $\Or$ and $\IdCompOr$ are strict test categories. To relate their test model structures to
$\KanSimp$, we recall the notion of aspheric functor.

\begin{definition}[{{\thmcite[\S3.3.3, Proposition 4.2.23(a$\Leftrightarrow$b$''$)]{cisinski06}}}]
  A functor $u \co \CatC \to \CatD$ is \emph{aspheric} if for every $d \in \CatD$,
  the presheaf $\subst{u}(\yo{d})$ is aspheric.  \qed
\end{definition}

An aspheric functor $u \co \CatC \to \CatD$ between test categories induces a Quillen equivalence
$\subst{u} \dashv \ran{u}$ between their test model structures~\cite[Proposition 4.2.24]{cisinski06}. For our
purposes, the more relevant property is the following immediate consequence.

\begin{proposition}[{{\thmcite[Proposition 4.2.23(d)]{cisinski06}}}]
  \label{aspheric-functor-weqs}
  Let $u \co \CatC \to \CatD$ be an aspheric functor between two test categories. Then a map $f$ in
  $\PSh{\CatD}$ is a weak equivalence in $\Test{\CatD}$ if and only if $\subst{u}f$ is a weak equivalence in
  $\Test{\CatC}$.
\end{proposition}

\begin{lemma}
  \label{idempotent-completion-aspheric}
  Any idempotent completion $i \co \CatC \to \IdComp{\CatC}$ is aspheric.
\end{lemma}
\begin{proof}
  Any $A \in \IdComp{\CatC}$ is a retract of $ia$ for some $a \in \CatC$. Then $\subst{i}\yo{A}$ is likewise a
  retract of $\subst{i}\yo{(ia)} \cong \yo{a}$, thus aspheric by
  \cref{homotopy-retracts-preserve-weakly-contractible}.
\end{proof}

\begin{lemma}
  \label{insimp-aspheric}
  $\insimp \co \Simp \to \IdCompOr$ is aspheric.
\end{lemma}
\begin{proof}
  For any $[1]^n \in \IdCompOr$, we have $\subst{\insimp\!}\yo{[1]^n} \cong (\simp{1})^n$. As $\Simp$ is a
  strict test category~\cite[Proposition 1.6.14]{maltsiniotis05}, any finite product of representables in
  $\PSh{\Simp}$ is aspheric~\cite[Proposition 4.3.2(b)]{cisinski06}.
\end{proof}
\begin{lemma}
  \label{cms-idcompor-weqs}
  A map $f$ in $\PSh{\IdCompOr}$ is a weak equivalence in $\CMSIdComp$ if and only if $\subst{\insimp\!}f$ is a
  weak equivalence in $\KanSimp$.
\end{lemma}
\begin{proof}
  Any left Quillen equivalence both preserves (Ken Brown's lemma) and reflects \cite[Corollary
  1.3.16]{hovey99} weak equivalences between cofibrant objects, so this follows from
  \cref{cms-right-quillen-equivalence}.
\end{proof}

\begin{theorem}
  The model structures $\TestIdComp$ and $\CMSIdComp$ are identical.
\end{theorem}
\begin{proof}
  As they have the same cofibrations, it suffices to show they have the same weak equivalences. This follows
  from \cref{aspheric-functor-weqs} and \cref{insimp-aspheric} (together with
  \cref{test-simplicial-is-kan-quillen}) and \cref{cms-idcompor-weqs}.
\end{proof}

\begin{corollary}
  The model structures $\TestOr$ and $\CMSOr$ are identical.
\end{corollary}
\begin{proof}
  Again, it suffices to show they have the same weak equivalences. By
  \cref{aspheric-functor-weqs,idempotent-completion-aspheric}, a map $f$ is a weak equivalence in $\TestOr$ if
  and only if $\lan{\incube}f$ is a weak equivalence in $\TestIdComp$. Likewise, $f$ is a weak equivalence in
  $\CMSOr$ if and only if $\lan{\incube}f$ is a weak equivalence in $\CMSIdComp$.
\end{proof}

These results can also be read as characterizations of the fibrations in the test model structures:

\begin{corollary}
  The fibrations in $\TestIdComp$ and $\TestOr$ are those maps lifting against $\Gd_k \lbtimes m$ for all
  $k \in \{0,1\}$ and $m \co A \mono B$.  \qed
\end{corollary}

\appendix

\section{Negative results}
\label{sec:negative}

Here we collect a pair of negative results concerning the existence of (relative) Reedy structures on
(idempotent completions of) cube categories. In \cref{sec:negative-disjunctive}, we check that $\Or$ and
$\IdCompOr$ are not Reedy categories, motivating this paper's approach. \cref{sec:negative-dedekind} concerns
the limits of relative elegance: we show that the Dedekind cube category does not embed elegantly in any Reedy
category.

\subsection{Semilattice cubes}
\label{sec:negative-disjunctive}

The non-existence of a Reedy structure on $\Or$ is easily verified: every Reedy category is idempotent
complete \cite[Proposition 6.5.9]{borceux94}, but we have seen in \cref{sec:idempotent-completion} that $\Or$
is not. The map $(x,y) \mapsto (x, x \lor y) \co [1]^2 \to [1]^2$ is a simple example of an idempotent with no
splitting in $\Or$.

It is therefore more appropriate to ask if the cube category's idempotent completion $\IdCompOr$, which we
have characterized as the full subcategory of $\SLat$ consisting of finite inhabited distributive lattices
(\cref{def:idcompor}), is Reedy. If this were so, we could simply study $\PSh{\Or}$ by way of the equivalent
$\PSh{\IdCompOr}$. However, this is not the case:

\begin{proposition} \label{idcompor-not-reedy}
  There is no Reedy structure on $\IdCompOr$.
\end{proposition}
\begin{proof}
  We consider the following morphism $u \co [1]^3 \to [1]^3$:
  \[
    u(x,y,z) \defeq (x \lor y, y \lor z, z \lor x) \rlap{.}
  \]
  For intuition, note that the image of $u$ computed in $\SLat$ is the non-distributive \emph{diamond lattice}
  $\mathfrak{M}_3$.

  Suppose that we do have a Reedy structure on $\IdCompOr$. The unique map $[1]^2 \to 1$ is split
  epic and thus a lowering map (\cref{split-epi-is-lowering}). Every raising map must have the right
  lifting property against this map, so every raising map is monic.\footnote{If we only want to show
    $\IdCompOr$ is not \emph{elegant} Reedy, we are already done, as observed in \cite[Theorem
    8.12(2)]{campion23}: if $\IdCompOr$ were elegant we would have a (split epi, mono) factorization of $u$,
    which would necessarily be preserved by the inclusion $\IdCompOr \to \SLat$, but $u$'s (split
    epi, mono) factorization in $\SLat$ is $\mathfrak{M}_3$.} Take a Reedy factorization of $u$:
  \[
    \begin{tikzcd}
      {[1]^3} \ar[lowering=above]{dr}[below left]{f} \ar{rr}{u} && {[1]^3} \rlap{.} \\
      & L \ar[raising=above,tail]{ur}[below right]{m}
    \end{tikzcd}
  \]
  $L$ is a sub-semilattice of $[1]^3$ that forms a distributive lattice and contains the image of $u$. Note
  that $\lor$, $\bot$, and $\top$ are computed in $L$ as in $[1]^3$, but $\land$ may not be; we write
  $\land_L$ for the meet in $L$. We show that in fact $L = [1]^3$.

  Consider the set $S \defeq \{011, 101, 110\} \subseteq L \subseteq [1]^3$. Let $v,v',v''$ be any pairwise
  distinct elements of $S$ and note that we have
  \[
    (v \land_L v') \lor (v \land_L v'') = v \land_L (v' \lor v'') = v \land_L \top = v.
  \]
  This implies the following.
  \begin{enumerate}[label=(\alph*)]
  \item $v \land_L v' \neq v \land_L v''$: otherwise we have
    $(v \land_L v') \lor (v \land_L v'') = v \land_L v''$ and thus $v = v \land_L v''$, but $v$ and $v''$
    are incomparable.
  \item $v \land_L v' \neq \bot$: otherwise we again have
    $(v \land_L v') \lor (v \land_L v'') = v \land_L v''$.
  \end{enumerate}
  Thus the meets $011 \land_L 101$, $011 \land_L 110$, and $011 \land_L 110$ are pairwise distinct and lie
  outside the image of $u$, which by a cardinality argument implies that $L$ is the whole of $[1]^3$.

  The lowering map $f$ of our supposed factorization must then be $u$ itself; it remains to show that $u$
  cannot be a lowering map.  Consider the semilattice morphism $t \co [1]^3 \to [2]$ defined by
  $t(x,y,z) \defeq x \lor 2y \lor 2z$. We have the following commutative diagram in $\IdCompOr$, where $d_1$
  and $s_1$ are the simplex face and degeneracy maps from \cref{def:simplex-degeneracy-face}:
  \[
    \begin{tikzcd}
      {[1]^3} \ar{d}[left]{u} \ar{r}{t} & {[2]} \ar{r}{s_1} & {[1]} \ar[raising,tail]{d}{d_1} \\
      {[1]^3} \ar{rr}[below]{t} && {[2]} \rlap{.}
    \end{tikzcd}
  \]
  The face map $d_1$ is split monic and therefore a raising map. If $u$ were a lowering map, this
  square would have a diagonal lift. But as $t$ is surjective, there can be no diagonal $[1]^3 \to [1]$ making
  the lower triangle commute.
\end{proof}

\subsection{Dedekind cubes}
\label{sec:negative-dedekind}

As mentioned in the introduction, it is an open question whether the cubical-type model structure for
presheaves on the Dedekind cube category $\Ded$ is equivalent to the Kan-Quillen model structure $\KanSimp$;
see \citeauthor{streicher21}\ \cite{streicher21} for further discussion. In this appendix, we show that $\Ded$
supports no relatively elegant embedding in a Reedy category, thus that our argument for $\Or$ admits no naive
adaptation to the two-connection case.

\begin{definition}
  The \emph{Dedekind cube category} $\Ded$ is the Lawvere theory of bounded distributive lattices.
\end{definition}

$\Ded$ admits an alternative description arising from the duality between finite bounded distributive lattices
and finite posets \cite{wraith93}, analogous to the description of $\Or$ as a full subcategory of $\SLat$:

\begin{proposition}
  $\Ded$ is equivalent to the full subcategory of $\Pos$ consisting of posets of the form $[1]^n$ for
  $n \in \BN$.
\qed
\end{proposition}

We will only need this latter description.

The Dedekind cube category attracted attention \cite{spitters16,sattler19,kapulkin20,streicher21,hackney22} in
the HoTT community following \citeauthor{cohen15}'s interpretation of HoTT in De Morgan cubical sets
\cite{cohen15}. As \citeauthor{orton18}\ note \cite[Remark 3.2]{orton18}, this interpretation does not require
all the structure of De Morgan cubes; in particular, it can be repeated with $\Ded$. The name ``Dedekind'' was
coined by Awodey in reference to the fact that the cardinality of $\Hom{\Ded}{[1]^n}{[1]}$ is the $n$th
Dedekind number.

\subsubsection{A no-go theorem}

We begin by identifying a property shared by all categories $\CatC$ with a relatively elegant functor
$i \co \CatC \to \CatR$; the contrapositive will show that no such functor exists out of $\Ded$.

\begin{definition}
  A \emph{sieve} on an object $a$ of a small category $\CatC$ is a set of morphisms
  $\CS \subseteq \Slice{\CatC}{a}$ such that $g \in \CS$ implies $gf \in \CS$ for any composable
  $f \in \CatC^\to$. We regard the collection $\Sieve{\CatC}{a}$ of sieves on $a \in \CatC$ as a poset ordered
  by inclusion. A sieve is \emph{principal} if it is of the form
  $\principal{f} \defeq \{gf \mid g \in \Slice{\CatC}{b} \}$ for some $f \co b \to a$; we write
  $\PSieve{\CatC}{a} \subseteq \Sieve{\CatC}{a}$ for the subposet of principal sieves on $a$.
\end{definition}

Recall that $\Sieve{\CatC}{a}$ is isomorphic to the poset of subobjects of $\yo{a} \in \PSh{\CatC}$. The
principal sieve $\principal{f}$ on a map $f \co b \to a$ corresponds to the subobject $\im f \mono
\yo{a}$. Given a relatively elegant $i \co \CatC \to \CatR$, the following lemma deduces a well-foundedness
property of these subobjects in $\PSh{\CatC}$ from the well-foundedness of the Reedy category $\CatR$.

\begin{lemma}
  \label{relative-elegant-principal-sieves-well-founded}
  Let $\CatC$ be a category, and let $\CatR$ be a Reedy category elegant relative to some
  $i \co \CatC \to \CatR$. Then for any $a \in \CatC$, there exists a strictly monotone map
  $d \co \PSieve{\CatC}{a} \to \BN$. In particular, $\PSieve{\CatC}{a}$ is well-founded.
\end{lemma}
\begin{proof}
  Given a principal sieve $\principal{f} \in \PSieve{\CatC}{a}$ generated by $f \co b \to a$, we define
  $d(\principal{f})$ to be the degree of $i(f)$, \ie, the degree of the intermediate object in its Reedy
  factorization. To see that this definition is independent of the choice of representative $f$ and that $d$
  is order-preserving, it suffices to check that for any $f \co b \to a$ and $f' \co b' \to a$, if
  $\principal{f'} \subseteq \principal{f}$ then $d(\principal{f'}) \leq d(\principal{f})$. If
  $\principal{f'} \subseteq \principal{f}$, then there exists some $g \co b' \to b$ such that $f' = fg$. Upon
  Reedy factorizing $i(f') = m'e'$ and $i(f) = me$, orthogonality gives us a map as shown:
  \[
    \begin{tikzcd}
      i(b') \ar[lowering=above]{d}[left]{e'} \ar{r}{i(g)} & i(b) \ar{r}{e} & c \ar[raising]{d}{m} \\
      c' \ar[dashed]{urr} \ar[raising=above]{rr}[below]{m'} && i(a) \rlap{.}
    \end{tikzcd}
  \]
  By \cref{reedy-cancellation}, the lift is a raising map, so
  $d(\principal{f}) = \deg{c'} \le \deg{c} = d(\principal{f'})$.

  To see that $d$ is strictly monotone, suppose that additionally $\deg{c'} = \deg{c}$. Then the diagonal
  above is an isomorphism. By $\CatR^-$-projectivity of $i(b)$ (\cref{relative-elegant-projective}) and
  fullness of $i$, we obtain a lift as below:
  \[
    \begin{tikzcd}
      & & i(b') \ar[lowering]{d}{e'} \\
      i(b) \ar[dashed]{urr}{i(h)} \ar{r}[below]{e} & c \ar{r}[below]{\cong} & c' \rlap{.}
    \end{tikzcd}
  \]
  Then $f = f'h$, so $\principal{f} \subseteq \principal{f'}$.
\end{proof}

\subsubsection{Principal sieves in Dedekind cubes}
\label{sec:dedekind-not-relative-elegant}

Now we show that the poset of principal sieves on $[1]^3 \in \Ded$ is not well-founded. We embed a poset model
of the circle $\Crown{n} \mono [1]^n$ in each cube, then exhibit a chain of subobjects of $[1]^n$
(for any $n \ge 3$) induced by maps $\cdots \to \Crown{n_2} \to \Crown{n_1} \to \Crown{n}$ that cannot
stabilize.

\begin{definition}
  The \emph{fence} $\Fence \in \Pos$ is the poset whose elements are integers and whose order is generated by
  the inequalities $i \le i - 1$ and $i \le i + 1$ for all even $i \in \BZ$.
\end{definition}

\begin{definition}
  The $n$th \emph{crown poset} $\Crown{n} \in \Pos$ is the quotient of $\Fence$ identifying $i,j \in \Fence$
  whenever $i = j \pmod{2n}$. We write $\crownproj{n} \co \Fence \to \Crown{n}$ for the quotient map.
\end{definition}

For example, $\Crown{4}$ is the following poset:
\[
  \begin{tikzcd}[row sep=large, column sep=small]
    1 & 3 & 5 & 7 \\
    0 \ar{u} \ar{urrr} & 2 \ar{u} \ar{ul} & 4 \ar{u} \ar{ul} & 6 \ar{u} \ar{ul} \rlap{.}
  \end{tikzcd}
\]

\begin{remark}
  Each crown poset is freely generated by a graph (though not the graphs usually known as \emph{crown graphs},
  which have more edges).
\end{remark}

The simplicial nerve $\nerve{\Simp}$ sends each crown poset to a simplicial set weakly equivalent to the
circle. As such, any map between crown posets can be associated a winding number. Concretely, we can define
the winding number on the level of posets as follows:

\begin{definition}
  Any poset map $f \co \Crown{m} \to \Crown{n}$ lifts to an endomap
  \[
    \begin{tikzcd}
      \Fence \ar{d}[left]{\crownproj{m}} \ar[dashed]{r}{\hat f} & \Fence \ar{d}{\crownproj{n}} \\
      \Crown{m} \ar{r}[below]{f} & \Crown{n}
    \end{tikzcd}
  \]
  which is unique modulo $2n$. The \emph{winding number} of $f$ is
  \begin{align*}
    \wind{f} &\defeq \frac{\hat f (2m) - \hat f (0)}{2n} \rlap{.}
  \end{align*}
\end{definition}

It is straightforward to check that $\wind{gf} = \wind{g}\wind{f}$ for
$\Crown{m} \overset{f}\to \Crown{n} \overset{g}\to \Crown{p}$, as we expect from a winding number. Because
$\Crown{m}$ is ``too short'' to wrap around $\Crown{n}$ when $m < n$, we have the following:

\begin{lemma}
  \label{crown-in-larger-crown-trivial}
  If $m < n$, then $\wind{f} = 0$ for any $f \co \Crown{m} \to \Crown{n}$.
\end{lemma}
\begin{proof}
  By induction, $\verts{\hat f (i) - \hat f (0)} \le i$ for every $i \in \BN$, so
  $\verts{\hat f (2m) - \hat f (0)} < 2n$.
\end{proof}

\begin{definition}
  For $n \ge 3$, define an poset embedding $c_n \co \Crown{n} \mono [1]^n$ by
  \begin{align*}
    c_n(i)_j = \left\{
    \begin{array}{ll}
      1 & \text{if $\floors{\frac{i}{2}} \le j \le \ceils{\frac{i}{2}}$} \\
      0 & \text{otherwise}
    \end{array}
    \right.
    \rlap{.}
  \end{align*}
\end{definition}

\begin{definition} \label{crown-extension} Given $m,n \ge 3$ and a monotone map
  $f \co \Crown{m} \to \Crown{n}$, define an extension
  \[
    \begin{tikzcd}
      \Crown{m} \ar[tail]{d}[left]{c_m} \ar{r}{f} & \Crown{n} \ar[tail]{d}{c_n} \\
      {[1]^m} \ar[dashed]{r}[below]{\overline f} & {[1]^n}
    \end{tikzcd}
  \]
  by setting
  \[
    \overline{f}(v) \defeq \left\{
      \begin{array}{ll}
        c_n(f(i)) &\text{if $v = c_m(i)$,} \\
        \bot &\text{if $v = \bot$,} \\
        \top &\text{otherwise.}
      \end{array}
    \right.
  \]
\end{definition}

The mapping $f \mapsto \overline{f}$ is the functorial action of a \emph{semifunctor} from the category of crown posets to $\Ded$: compositions are preserved, but not identities.

\begin{lemma} \label{crown-extension-pullback}
The diagram in \cref{crown-extension} is a pullback.
\end{lemma}
\begin{proof}
The three cases in the definition of $\overline{f}$ have disjoint values.
\end{proof}

\begin{theorem}
  \label{dedekind-not-relatively-elegant}
  There exists no Reedy category $\CatR$ with a fully faithful functor $i \co \Ded \to \CatR$ such that
  $\CatR$ is elegant relative to $i$.
\end{theorem}
\begin{proof}
  Suppose for sake of contradiction that we have some $i \co \Ded \to \CatR$ such that $\CatR$ is elegant
  relative to $i$. Choose any $n \ge 3$. For every $m \ge 2$ and $a \ge 1$, the identity function on $\Fence$
  induces a map $f_a \co \Crown{am} \to \Crown{m}$ with winding number $a$. We then have the following diagram
  in $\Pos$:
  \[
    \begin{tikzcd}
      \cdots \ar{r}{f_2} & \Crown{8n} \ar{d}{f_8} \ar{r}{f_2} & \Crown{4n} \ar{d}{f_4} \ar{r}{f_2} & \Crown{2n} \ar{d}{f_2} \\
      \cdots \ar{r}[below]{\id} & \Crown{n} \ar{r}[below]{\id} & \Crown{n} \ar{r}[below]{\id} & \Crown{n} \rlap{.}
    \end{tikzcd}
  \]
  Applying $\overline{(-)}$, we have a chain of principal sieves
  $\principal{\overline{f_2}} \supseteq \principal{\overline{f_4}} \supseteq \principal{\overline{f_8}}
  \supseteq \cdots$ on $[1]^n$. By \cref{relative-elegant-principal-sieves-well-founded}, this chain must
  stabilize; in particular, there must be some pair $a < b$ (both powers of 2) such that
  $\principal{\overline{f_a}} = \principal{\overline{f_b}}$. Then there exists a map
  \[
    \begin{tikzcd}
      {[1]^{a n}} \ar{dr}[below left]{\overline{f_a}} \ar[dashed]{rr}{g} & & {[1]^{b n}} \ar{dl}{\overline{f_b}}  \rlap{.} \\
      & {[1]^n}
    \end{tikzcd}
  \]
  By \cref{crown-extension-pullback}, we have an induced map of crown posets:
  \[
    \begin{tikzcd}
      \Crown{an} \ar{d}[left]{c_{an}} \ar[bend left]{rr}{f_a} \ar[dashed]{r}[below]{g'} & \Crown{bn} \pullback \ar{d}[left]{c_{bn}} \ar{r}[below]{f_b} & \Crown{n} \ar{d}{c_{n}} \\
      {[1]^{a n}} \ar[bend right]{rr}[below]{\overline{f_a}} \ar{r}{g} & {[1]^{b n}} \ar{r}{\overline{f_b}} & {[1]^n} \rlap{.}
    \end{tikzcd}
  \]
  But because $an < bn$, we must have $\wind{g'} = 0$ by \cref{crown-in-larger-crown-trivial}, which
  contradicts that $\wind{f_b} \wind{g'} = \wind{f_a} = a$.
\end{proof}

\begin{Backmatter}

\printbibliography

\printaddress

\end{Backmatter}

\end{document}